\documentclass[12pt,reqno]{amsart}
\usepackage{amsmath,a4wide}
\usepackage{xfrac}
\usepackage{stmaryrd,mathrsfs,bm,amsthm,mathtools,yfonts,amssymb,color}
\usepackage{xcolor}
\usepackage{courier}

\begingroup
\newtheorem{theorem}{Theorem}[section]
\newtheorem{lemma}[theorem]{Lemma}
\newtheorem{proposition}[theorem]{Proposition}
\newtheorem{corollary}[theorem]{Corollary}
\endgroup

\theoremstyle{definition}
\newtheorem{definition}[theorem]{Definition}
\newtheorem{remark}[theorem]{Remark}
\newtheorem{ipotesi}[theorem]{Assumption}

\numberwithin{equation}{section}
\numberwithin{subsection}{section}

\newcommand{\dv}{{\textup {div}}}
\newcommand{\ph}{\varphi}
\newcommand\Id{{\rm Id}\,}

\newcommand{\sing}{{\rm Sing}}
\newcommand{\reg}{{\rm Reg}}
\newcommand\supp{{\rm spt}}
\newcommand\res{\mathop{\hbox{\vrule height 7pt width .3pt depth 0pt
\vrule height .3pt width 5pt depth 0pt}}\nolimits}
\newcommand\ser{\mathop{\hbox{\vrule height .3pt width 5pt depth 0pt
\vrule height 7pt width .3pt depth 0pt}}\nolimits}

\newcommand{\mass}{{\mathbf{M}}}

\newcommand{\bOmega}{{\mathbf{\Omega}}}
\newcommand{\de}{\partial}
\newcommand{\bI}{{\mathbf{I}}}

\def\a#1{\left\llbracket{#1}\right\rrbracket}
\newcommand{\bE}{{\mathbf{E}}}
\newcommand{\gr}{{\rm Gr}}
\newcommand{\bh}{{\mathbf{h}}}
\newcommand{\im}{{\rm Im}}
\newcommand{\bT}{\mathbf{T}}
\newcommand{\cH}{{\mathcal{H}}}
\newcommand\rD{{\rm D}}

\def\I#1{{\mathcal{A}}_{#1}}
\newcommand{\Iq}{{\mathcal{A}}_Q}

\newcommand{\bG}{{\mathbf{G}}}
\newcommand{\cG}{{\mathcal{G}}}
\newcommand{\etaa}{{\bm{\eta}}}
\newcommand{\D}{\textup{Dir}}

\newcommand{\B}{{\mathbf{B}}}
\newcommand{\bC}{{\mathbf{C}}}

\newcommand\bS{\mathbf{S}}

\def\Xint#1{\mathchoice
{\XXint\displaystyle\textstyle{#1}}%
{\XXint\textstyle\scriptstyle{#1}}%
{\XXint\scriptstyle\scriptscriptstyle{#1}}%
{\XXint\scriptscriptstyle\scriptscriptstyle{#1}}%
\!\int}
\def\XXint#1#2#3{{\setbox0=\hbox{$#1{#2#3}{\int}$ }
\vcenter{\hbox{$#2#3$ }}\kern-.6\wd0}}
\def\mint{\Xint-}

\newcommand{\be}{{\mathbf{e}}}

\newcommand\Z{{\mathbb Z}}
\newcommand\N{{\mathbb N}}

\newcommand\C{{\mathbb C}}
\newcommand\R{{\mathbb R}}

\newcommand{\eps}{{\varepsilon}}
\newcommand{\bA}{\mathbf{A}}
\newcommand{\bmax}{\mathbf{m}}
\newcommand{\cR}{{\mathcal{R}}}
\newcommand{\cF}{{\mathcal{F}}}

\newcommand{\Lip}{{\rm {Lip}}}

\newcommand{\dist}{{\rm {dist}}}
\renewcommand{\d}{{\rm d}}

\newcommand{\p}{{\mathbf{p}}}

\newcommand{\dive}{{\text{div}}}

\newcommand{\osc}{{\textup{osc}}}
\newcommand{\bef}{\mathbf{f}}
\newcommand{\beg}{\mathbf{g}}
\newcommand{\bGam}{{\bm \Gamma}}
\newcommand\bmo{{\bm m}_0}

\newcommand{\cM}{{\mathcal{M}}}

\newcommand{\bU}{{\mathbf{U}}}
\newcommand{\bL}{{\mathbf{L}}}
\newcommand{\phii}{{\bm{\varphi}}}
\newcommand{\Phii}{{\bm{\Phi}}}

\newcommand{\cC}{{\mathcal{C}}}

\newcommand{\cV}{{\mathcal{V}}}

\newcommand{\cL}{{\mathcal{L}}}
\newcommand{\cK}{{\mathcal{K}}}
\newcommand{\cB}{{\mathcal{B}}}

\newcommand{\sC}{{\mathscr{C}}}
\newcommand{\sS}{{\mathscr{S}}}
\newcommand{\sW}{{\mathscr{W}}}
\newcommand{\sP}{{\mathscr{P}}}

\newcommand{\bD}{{\mathbf{D}}}
\newcommand{\bH}{{\mathbf{H}}}

\newcommand{\bF}{{\mathbf{F}}}
\newcommand{\bSigma}{{\mathbf{\Sigma}}}

\DeclareMathAlphabet{\mathpzc}{OT1}{pzc}{m}{it}

\renewcommand{\d}{\textup{d}}

\title[Almgren's type regularity for semicalibrated currents]{Almgren's type regularity for semicalibrated currents}
\author{Luca Spolaor}
\address{Max Planck Institute For Mathematics, Leipzig}
\email{Luca.Spolaor@mis.mpg.de}

\begin{document}

\begin{abstract}
In analogy with Almgren's Theorem for area minimizing currents of general dimension and codimension, we prove that an $m$-dimensional semicalibrated current in a $(n+m)$-dimensional $C^{3,\eps_0}$ manifold, semicalibrated by a $C^{2,\eps_0}$ $m$-form, has singular set of Hausdorff dimension at most $m-2$. 
\end{abstract}

\maketitle

Our goal is to study the dimension of the singular set of semicalibrated currents in any dimension and codimension. 

\begin{definition}\label{d:semicalibrated}
Let $\Sigma \subset R^{m+n}$ be a $C^2$ submanifold and $U\subset \R^{m+n}$ an open set.
A semicalibration (in $\Sigma$) is a $C^1$ $m$-form $\omega$ on $\Sigma$ such that 
  $\|\omega_x\|_c \leq 1$ at every $x\in \Sigma$, where $\|\cdot \|_c$ denotes the comass norm on $\Lambda^m T_x \Sigma$. 
  An $m$-dimensional integral current $T$ with $\supp (T)\subset \Sigma$ is {\em semicalibrated} by $\omega$ if $\omega_x (\vec{T}) = 1$ for $\|T\|$-a.e. $x$.
\end{definition}

In what follows, given an integer rectifiable current $T$, we denote by $\reg (T)$ the subset of $\supp (T)\setminus \supp (\partial T)$ consisting of those points $x$ for which there is a neighborhood $U$ such that $T\res U$ is a (constant multiple of) a regular submanifold. Correspondingly, $\sing (T)$ is
the set $\supp (T)\setminus (\supp (\partial T)\cup \reg (T))$. Observe that $\reg (T)$ is relatively open in $\supp (T) \setminus \supp (\partial T)$
and thus $\sing (T)$ is relatively closed.
The main result of this paper is then the following

\begin{theorem}\label{t:finale}
Let $T$ be an $m$-dimensional semicalibrated current. Assume in addition that
$\Sigma$ is of class $C^{3,\eps_0}$ and $\omega$ of class $C^{2,\eps_0}$ for some positive $\eps_0$. Then the Hausdorff dimension of $\sing(T)$ is at most $m-2$.
\end{theorem}

A typical example of semicalibrated current is given by almost complex cycles in almost complex manifolds. An almost complex manifold is a couple $(\Sigma^{2n}, J)$, where $\Sigma^{2n}$ is a smooth manifold and $J$ is an endomorphism of the tangent bundle such that $J^2=-\Id$. An almost complex cycle is an integral current $T\in \bI_m(\Sigma^{2n})$ such that $\de T=0$ and $J_x \vec{T}_x= \vec{T}_x$ for $\|T\|$-a.e. $x\in \Sigma^{2n}$. Object of this type where studied by Rivi\'ere and Tian in \cite{RT1,RT2}, in the particular case where the almost complex manifold is locally simplectic, and so almost complex cycles are locally Area Minimizing; our situation is more general since semicalibrated currents might not be minimizers of the area. 

To our knowledge, this is the first result about the regularity of semicalibrated currents in dimension higher than $2$, except for the case of $3$-dimensional Lagrangian cones, studied by Bellettini and Riviere in \cite{BeRi}, and for the case of the singular set of $J$-holomorphic map, which are stationary harmonic maps and so the $m-2$ dimensional bounds follows from classical theory (cf. \cite{RT1,Taubes}). For the two dimensional case several results have been proved by various authors, cf \cite{RT2, Be, Taubes}).

The starting point of our work is the famous result by Almgren (cf. \cite{Alm}), where the same dimensional bound of Theorem \ref{t:finale} is proved for the singular set of Area Minimizing Integral currents in general dimension and codimension. Our proof relies heavily on the ideas introduced by Almgren, indeed the strategy is the same and the tools used are taken mainly from the revisitation of Almgren's work done recently by Camillo De Lellis and Emanuele Spadaro (cf. \cite{DS1,DS2,DS3,DS4,DS5}). 
However our result is not a direct consequence of Almgren's Theorem, nor it is a generalization of it, since not every Area Minimizing current can be semicalibrated. On the other hand, the proof of Theorem \ref{t:finale} that we give here could be applied also in the case of Area Minimizing currents, while the opposite is not true. 

Notice that the union of two complex planes in $\C^2$ is calibrated and so semicalibrated, and this shows that Theorem \ref{t:finale} is optimal in terms of the dimension of the singular set (at least in even dimension). 

The structure of the singular set and the finiteness of its measure are still open questions in general, except in the case of surfaces, where it consists of isolated points (cf. \cite{DSS1,DSS2,DSS3,DSS4}).\\

Next we briefly describe the proof of Theorem \ref{t:finale}. 
The two main steps, which coincide with the two main parts of this article, are the following.

\begin{itemize}
\item[(CM)] In the first half of the paper we show how to construct a $C^{3,\gamma_0}$, $m$-dimensional surface, called Center Manifold, that carefully approximates the average of the sheets of the current. This comes together with a multiple valued Lipschitz map defined on it, called Normal Approximation, that coincides with the current itself, except on a suitably small set. For a nice explanation of why the center manifold is essential and how it is constructed see \cite{DLs}.

\item[(BU)] In the second part of the paper we prove Theorem \ref{t:finale} by contradiction. To do this, we first observe that it is enough to look at branch points, in particular there must be a singular point with at least one tangent plane and where the singular set has positive $\cH^{m-2+\alpha}$-density, for some positive $\alpha$. Next we construct a sequence of Center Manifolds and Normal Approximations that converges to a nontrivial Dir-minimizing function with a singular set of dimension $m-2+\alpha$, which contradicts the regularity of $\D$-minimizing functions. The key to the non-triviality of the blow up map is the asymptotic behaviour of the Frequency function (see again \cite{DLs} for an explanation of this), while the $(m-2+\alpha)$-dimension of its singular set comes from a "persistency property" of points of high multiplicity.
\end{itemize}

The rest of the paper is devoted to these two steps. A big part of it is taken from previous works by De Lellis and Spadaro, namely \cite{DS3,DS4,DS5}. Therefore, whenever possible, we will skip the proof of those results that are already contained in there and simply give the precise reference.

However three parts are substantially different with respect to \cite{DS4,DS5}, and these are
\begin{itemize}
\item[(i)] the construction of the building blocks of the center manifold, called interpolated functions, which are not obtained via a smoothing procedure, but are required to solve a specific second order linear elliptic system coming from the variational structure of the problem (cf. Definition \ref{d:smoothing} below); 
\item[(ii)] the persistence of points of maximum multiplicity, which is obtained as a simple corollary of an improved height bound valid in the neighborhood of any such point, rather than as a consequence of the monotonicity formula (cf. \cite{DS3,DS4}); 
\item[(iii)] the asymptotic analysis of the Frequency function, which requires the modification of its definition with an additional term in order to obtain its almost monotonicity. 
\end{itemize}
We remark that (ii) appears here for the first time, while (i) and (iii), in a simplified form, were already introduced in \cite{DSS3, DSS4}.

To conclude this introduction, we recall two important results about semicalibrated currents, whose proofs can be found in \cite{DSS1}.

\begin{lemma}[{\cite[Lemma 1.1]{DSS1}}]\label{l:riduzione}
Let $k\in \mathbb N\setminus\{ 0\}$, $\eps_0 \in [0,1]$, $\Sigma\subset \R^{m+n}$ be a $C^{k+1, \eps_0}$ $(m+\bar{n})$-dimensional submanifold, $V\subset \R^{m+n}$ an open subset and 
$\omega$ a $C^{k, \eps_0}$ $m$-form on
$V\cap \Sigma$. If $T$ is a cycle in $V\cap \Sigma$ semicalibrated by $\omega$, then $T$ is semicalibrated in $V$ by a $C^{k, \eps_0}$ form $\tilde{\omega}$. 
\end{lemma} 

In particular, thanks to this Lemma and the local nature of regularity, we can assume, without loss of generality, that $\Sigma=\R^n$. Moreover,  whenever we refer the reader to proofs contained in \cite{DS4,DS5}, this identification will be implicitly made.

\begin{lemma}[{\cite[Proposition 1.2]{DSS1}}]\label{l:quasi_minimalita}
Let $T$ be as in Definition \ref{d:semicalibrated} with $\Sigma = \R^{m+n}$. Then there is a positive constant $\bOmega\leq \|d\omega\|_0$ such that
\begin{equation}\label{e:Omega}
\mass (T) \leq \mass (T + \partial S) + \bOmega\, \mass (S)\qquad \forall S\in {\bf I}_{m+1} (\R^{m+n})\quad \mbox{with compact support.}
\end{equation}

Moreover, if $\chi\in C^\infty_c (\R^{m+n}\setminus \supp (\partial T), \R^{m+n})$, we have
\begin{equation}\label{e:var_prima_(b)}
\delta T (\chi) = T (d\omega \ser \chi)\,.
\end{equation}
\end{lemma}

Together with Camillo De Lellis and Emanuele Spadaro, we conscruted a Lipschitz approximation for integral currents satisfying  \eqref{e:Omega}, which is superlinear in the excess (cf. \cite{DSS2}). This is the first key step for the proof of the result. Equation \eqref{e:var_prima_(b)}, on the other hand, is the key ingredient for the present paper, dictating the equation satisfied by the center manifold, and the structure of the error in the analysis of the Frequency Function.

\subsection{Plan of the paper}

The paper is divided in two parts plus a preliminary section. 
\begin{itemize}
\item In the preliminary section we introduce the main notations and definitions, we present the Improved Height bound around point of maximal multiplicity and compare it with the usual stratified Height bound. The question of the stratification of the new Height bound remains open.
\item In the first part of the paper we construct the Center Manifold and the associated Normal Approximation with its estimate. In particular Section 2 contains the construction algorithm for the Center Manifold, while Section 3 states all the main estimates on the Normal Approximation. The proofs of the statements are all contained in sections 4 \& 5. 
\\
Many times in this part we will refer the reader to results contained in \cite{DS4}. We can do this because the particular structure of the current $T$, that is \eqref{e:var_prima_(b)}, enters only in the proofs of the main estimates for the building blocks of the Center Manifold in Proposition \ref{p:stime_chiave}, while the rest is essentially a consequence of the construction algorithm and the Height bounds of Section 1. There is however an additional technical difficulty, indeed many times we will need to use the harmonic approximation result in \cite[Theorem 4.2]{DSS2}, but this is given under different assumptions then the one used in \cite{DS4} (i.e., \cite[Theorem 1.6]{DS3}), and this will force us to slightly modify some proofs.  

\item In the second part we conclude the proof of Theorem \ref{t:finale}, by introducing a suitable contradiction sequence (Section 6), studying the asymptotic behaviour of the corresponding Frequency functions (Section 7) and proving its convergence to a nontrivial Dir-minimizing function (Section 8).\\
In this part the specific structure of our current enters only in the estimate on the Frequency function and in proving the minimality of the blow-up map, the rest is a consequence of the construction algorithm and the Height bounds of Section 1.  
\end{itemize}

The main differences with respect to \cite{DS4, DS5} are due to the variational formula \eqref{e:var_prima_(b)}, which is more complicate than the one of the $\D$-minimizing case. In particular it prevents us to apply the monotonicity method of \cite{DS4} to prove the Persistency of $Q$-points, which is instead achieved by an Improved Height bound (Theorem \ref{t:Imp_HeightBound} in Section 1) and a Blow-up Lemma (Lemma \ref{l:exdecay} in Section 6). 
Furthermore, it forces us to define the Center Manifold using an Elliptic sistem (Propositions \ref{p:pde} \& \ref{p:stime} in Section 5), and it is responsible for some additional errors in the almost monotonicity of the frequency and in the almost minimality of the Normal approximation (Sections 7 \& 8).

\subsection{Acknowledgements}  
The author is particularly grateful to Camillo De Lellis and Emanuele Spadaro, whose beautiful work made this paper possible, and to Guido De Philippis, for many useful discussions, especially about the Improved Height bound and its consequences. The author is also indebted to Philippe Logaritsch and Francesco Ghiraldin, for helping him out with technical details.
This work has been partially supported by the ERC grant RAM (Regularity for Area Minimizing currents), ERC 306247.

\section{Height bounds}

In this first section, after introducing some notations, we state and prove the height bound with improved exponent that will be crucial in the rest of the paper. 

\subsection{Notation, height and excess}  For open balls in $\R^{m+n}$ we use $\B_r (p)$. 
For any linear subspace $\pi\subset \R^{m+n}$, $\pi^\perp$ is
its orthogonal complement, $\p_\pi$ the orthogonal projection onto $\pi$, $B_r (q,\pi)$ the disk $\B_r (q) \cap (q+\pi)$
and  $\bC_r (p, \pi)$ the cylinder $\{(x+y):x\in \B_r (p), y\in \pi^\perp\}$ (in both cases $q$ is 
omitted if it is the origin and $\pi$ is omitted if it is clear from the context). We also assume that each $\pi$ is {\em oriented} by a
$k$-vector $\vec\pi :=v_1\wedge\ldots \wedge v_k$
(thereby making a distinction when the same plane is given opposite orientations) and with a slight abuse of notation we write $|\pi_2-\pi_1|$ for
$|\vec{\pi}_2 - \vec{\pi}_1|$ (where $|\cdot |$ stands for the norm associated to the usual inner product of $k$-vectors).

A primary role will be played by the $m$-dimensional plane $\mathbb R^m \times \{0\}$ with the standard orientation:
for this plane we use the symbol $\pi_0$ throughout the whole paper.

\begin{definition}[Excess and height]\label{d:excess_and_height}
Given an integer rectifiable $m$-dimensional current $T$ in $\mathbb R^{m+n}$ with finite mass and compact support and $m$-planes $\pi, \pi'$, we define the {\em excess} of $T$ in balls
and cylinders as
\begin{align}
\bE(T,\B_r (x),\pi) &:= \frac{1}{2\omega_m\,r^m}\int_{\B_r (x)} |\vec T - \vec \pi|^2 \, d\|T\|,\\
\qquad \bE (T, \bC_r (x, \pi), \pi') &:= \frac{1}{2\omega_m\,r^m} \int_{\bC_r (x, \pi)} |\vec T - \vec \pi'|^2 \, d\|T\|\, ,
\end{align}
and the {\it height function} in a set $A \subset \R^{m+m}$ as
\[
\bh(T,A,\pi) := \sup_{x,y\,\in\,\supp(T)\,\cap\, A} |\p_{\pi^\perp}(x)-\p_{\pi^\perp}(y)|\, .
\]
\end{definition}

In what follows all currents will have compact support and finite mass and will always be considered as currents
defined in the entire Euclidean space. As a consequence their restrictions to a set $A$ and their pushforward through
a map $\p$ are well-defined as long as $A$ is a Borel set and the map $\p$ is Lipschitz in a neighborhood of their support.

\begin{definition}[Optimal planes]\label{d:optimal_planes}
We say that an $m$-dimensional plane $\pi$ {\em optimizes the excess} of $T$ in a ball $\B_r (x)$ if
\begin{equation}\label{e:optimal_pi}
\bE(T,\B_r (x)):=\min_\tau \bE (T, \B_r (x), \tau) = \bE(T,\B_r (x),\pi).
\end{equation} 
Observe that in general the plane optimizing the excess is not unique and $\bh (T, \B_r (x), \pi)$ might
depend on the optimizer $\pi$. Since for notational purposes it is convenient to define
a unique ``height'' $\bh (T, \B_r (x))$, we call a plane $\pi$ as in \eqref{e:optimal_pi} {\em optimal} if in addition
\begin{equation}\label{e:optimal_pi_2}
\bh(T,\B_r(x),\pi) = \min \big\{\bh(T,\B_r (x),\tau): \tau \mbox{ satisfies \eqref{e:optimal_pi}}\big\}
= : \bh(T,\B_r(x))\, ,
\end{equation}
i.e. $\pi$ optimizes the height among all planes that optimize the excess. However \eqref{e:optimal_pi_2} does not play
any further role apart from simplifying the presentation. 

In the case of cylinders, instead, $\bE (T, \bC_r (x, \pi))$ will denote $\bE(T, \bC_r (x, \pi), \pi)$ (which coincides with
the cylindrical excess used in \cite{DS3} when $(\p_\pi)_\sharp T \res \bC_r (x, \pi)=
Q \a{B_r (\p_\pi (x), \pi)}$), whereas $\bh (T, \bC_r (x, \pi))$ will be used for $\bh (T, \bC_r (x, \pi), \pi)$. 
\end{definition}

\subsection{Old and new height bounds}

We start by proving that the standard stratified height bound for area minimizing currents (cf. \cite[Theorem A.1]{DS4}) holds also for semicalibrated currents.

\begin{lemma}[{\cite[Lemma A.1]{DSS3}}]\label{t:height_bound0} \label{l:density2}
There is a positive geometric constant $c(m,n)$ with the following property. If $T$ is a semicalibrated current, where $\bOmega\leq c (m,n)$, then
\begin{equation}\label{e:density}
\|T\| (\B_\rho (x)) \geq \omega_m (\Theta (T, p) - \textstyle{\frac{1}{4}}) \rho^m \geq \omega_m \textstyle{\frac{3}{4}} \rho^m \qquad \forall p\in \supp (T), \forall r\in \dist (p, \partial U)\, .
\end{equation}
\end{lemma}

\begin{proof} 
By \cite[Proposition 1.2]{DSS1} $\|T\|$ is an integral varifold with bounded mean curvature in the sense of Allard, where $C \bOmega$ bounds the mean curvature for some geometric constant $C$. It follows from Allard's monotonicity formula that $e^{C \bOmega r} \|T\| (\B_r (x))$  is monotone nondecreasing in $r$, from which the first inequality in \eqref{e:density} follows. The second inequality is implied by $\Theta (T, p)\geq 1$ for every $p\in \supp (T)$: this holds because the density is an upper semicontinuous function which takes integer values $\|T\|$-almost everywhere.
\end{proof}

For the proof of the next statement we refer to \cite[Theorem A.1]{DS4}: although there $T$ satisfies the stronger assumption of being area minimizing, a close inspection of the proof given in \cite{DS4} shows that the only property of area minimizing currents relevant to the arguments is the validity of the density lower bound \eqref{e:density}.

\begin{theorem}[{\cite[Theorem A.2]{DSS3}}]\label{t:height_bound} Let $Q$, $m$ and $n$ be positive integers. Then there are
$\eps >0, c>0$ and $C$ geometric constants with the following property. Assume that $\pi_0 = \R^m\times \{0\}\subset \R^{m+n}$ and that:
\begin{itemize}
\item[(h1)] $T$ is an integer rectifiable $m$-dimensional semicalibrated current with $U = \bC_r (x_0)$ and $\bOmega \leq c$;
\item[(h2)] $\partial T\res \bC_r (x_0) = 0$, $(\p_{\pi_0})_\sharp T\res \bC_r (x_0) = Q\a{B_r (\p_{\pi_0} (x_0))}$ 
and $E:= \bE (T, \bC_r (x_0)) < \eps$.
\end{itemize}
Then there are $k\in \N$, points $\{y_1, \ldots, y_k\}\subset \R^{m+n}$ and positive integers $Q_1, \ldots, Q_k$ such that:
\begin{itemize}
\item[(i)] having set $\sigma:= C E^{\sfrac{1}{2m}}$, the open sets $\bS_i := \R^m \times (y_i +\, ]-r \sigma, r \sigma[^n)$
are pairwise disjoint and $\supp (T)\cap \bC_{r (1- \sigma |\log E|)} (x_0) \subset \cup_i \bS_i$;
\item[(ii)] $(\p_{\pi_0})_\sharp [T \res (\bC_{r(1-\sigma |\log E|)} (x_0) \cap \bS_i)] =Q_i  \a{B_{r (1-\sigma |\log E|)}(\p_{\pi_0} (x_0), \pi_0)}$ $\forall i\in \{1, \ldots , k\}$.
\item[(iii)] for every $p\in \supp (T)\cap \bC_{r(1-\sigma |\log E|)} (x_0)$ we have $\Theta (T, p) < \max \{Q_i\} + \frac{1}{2}$.
\end{itemize}
\end{theorem}

Next we introduce a sharp version of the height bound for currents with bounded generalized mean curvature, around points of maximal multiplicity. A similar result, in a different context, can be found in \cite{Nes}, while other Poincar\'e-types inequalities can be found in \cite{Men1,Men2}.

\begin{theorem}[Improved Height Bound]\label{t:Imp_HeightBound}
 Let $Q$, $m$ and $n$ be positive integers. Then there are
$\eps , c,C>0$ geometric constants with the following property. Assume that $\pi_0 = \R^m\times \{0\}\subset \R^{m+n}$ and that:
\begin{itemize}
\item[(ih1)] $T$ is an integer rectifiable $m$-dimensional current with generalized mean curvature bounded by $c$ in $\B_{2r}(x)\subset \R^{n+m}$, that is 
\begin{equation}\label{e:GMC}
 \int \dive_T X\,d\|T\|=- \int X \cdot \bH \,d\|T\|\qquad \forall X\in C^\infty_c(\B_{4r})\,,
\end{equation}
and $\bA:=\|\bH\|_\infty\leq c$;
\item[(ih2)] $\partial T\res \bC_{4r} (x_0) = 0$, $(\p_{\pi_0})_\sharp T\res \bC_{4r} (x_0) = Q\a{B_{2r} (\p_{\pi_0} (x_0))}$ 
and $E:= \bE (T, \bC_{4r} (x_0)) < \eps$;
\item[(ih3)] $\pi_0$ is an optimal plane for $T$ and $\Theta(T,x_0)=\lim_{r\to 0}\frac{\|T\|(\B_r(x_0))}{\omega_m r^m}=Q$.
\end{itemize}
Then the following bound holds
\begin{equation}\label{e:1/2bound}
\bh(T, \bC_r(x_0,\pi_0)) \leq C(E^{\sfrac12}+(r\bA)^{\sfrac12})r.
\end{equation}
\end{theorem}
 
\begin{remark}
Notice that, thanks to Lemma \ref{l:quasi_minimalita}, Theorem \ref{t:Imp_HeightBound} covers the case of semicalibrated currents with $\bOmega\leq c$. Indeed, if $\bOmega\leq c$, then the linear functional $X\mapsto T(d\omega \ser X)$ is bounded and therefore, by the Riesz representation theorem, $T$ has generalized mean curvature bounded by $c$.
\end{remark}

The proof of this result is obtained combining an $L^\infty-L^2$ bound with a Poincar\'e inequality. In the sequel we will assume without loss of generality that $x_0=0$ and $r=1$. Furthermore we let $(\be_i)_{i=1\dots,m+n}$ denote the standard basis of $\R^{m+n}$ and $p=(x,y)=(x^1,\dots,x^m,y^1,\dots,y^n)$ a generic point of $\R^{m+n}$.

The $L^\infty-L^2$ bound follows from the simple observation that the coordinates of the support of the current solve an elliptic PDE.

\begin{lemma}[$L^\infty-L^2$ bound]\label{l:L2Linf}
There are
$\eps >0, c>0$ and $C>0$ geometric constants with the following property. Let $T$ be as in Theorem \ref{t:Imp_HeightBound}, then 
\begin{equation}\label{e:Linfbound}
\sup_{(x,y)\in \bC_{1}\cap \supp(T)}|y^j|^2\leq C \Bigl(\int_{\bC_2}|y|^2\,d\|T\|+ \bA^2 \Bigr), \quad j=1,\dots,n.
 \end{equation}
\end{lemma}

\begin{proof}
Let us fix $j\in \{1,\dots,n\}$ and notice that, from \eqref{e:GMC} with $X=\xi \be_j$, for some $\xi\in C^\infty_c(\bC_{2})$, we get
\begin{equation}\label{e:var}
 \int \nabla y^j \nabla \xi\,d\|T\|= \int H^j\,\xi  \,d\|T\|\,.
\end{equation}
where $H_j:=H\cdot \be_j$. Using this identity it is immediate to see that there exists a geometric constant $C>0$ such that the function $|y^j| + C \|\bH\|_\infty |x|^2$ is subharmonic on $T$, therefore we can apply \cite[7.5 (6) Theorem, pg. 464]{All}  
to conclude
\[
\sup_{(x,y)\in \bC_{1}\cap \supp(T)}|y^j|^2\leq C\left( \int_{\bC_2}|y|^2\,d\|T\|+\int_{\bC_2}|x|^4\,\|\bH\|_{\infty}\, \,d\|T\| \right)\quad\|T\|\text{-a.e. }(x,y)\in \bC_1\,,
\] 
from which the lemma follows.
\end{proof}

The Poincar\'e inequality is actually true only in points of maximal density and comes from the monotonicity inequality. The precise statement is the following.

\begin{lemma}[Poincar\'e inequality]\label{l:Poincare}
There are
$\eps,c,C>0$ geometric constants such that, if $T$ is as in Theorem \ref{t:Imp_HeightBound}, then
\begin{equation}\label{e:poincare}
\int_{\bC_2}|y|^2\,d\|T\|\leq C (E+\bA).
\end{equation}
\end{lemma}

\begin{proof}
Testing \eqref{e:GMC} with a radial vector field $X(p):=\gamma(|p|) \; p$, and 
letting $\gamma$ converge to the characteristic function of the ball $\B_4$ 
we obtain (cf. \cite[17.3]{Sim})
\begin{align}
\int_{\B_4\setminus \B_s} \frac{1}{|p|^2}
&\left|\frac{p^\perp}{|p|}\right|^2 d \|T\|(p)
=  \frac{\|T\|(\B_4)}{\omega_m\,4^m}
-\frac{\|T\|(\B_s)}{\omega_m\,s^m} \notag\\
& - \frac12 \int_{\B_4}
p\cdot H(p) \, \left(\frac{1}{\max\{s,|p|\}^m} -\frac{1}{4^m} \right) d\|T\|(p).
\label{e:monotonicity}
\end{align}
Moreover, for $2^{-j-1} \leq s \leq 2^{-j}$, the last term in \eqref{e:monotonicity}
can be estimated by
\begin{equation}\label{e:stima stupida}
-\frac12 \int_{\B_{2^{-j}}}
p\cdot H(p) \, \left(\frac{1}{\max\{s,|p|\}^m} -\frac{1}{2^{-jm}} \right) d\|T\|(p)
\leq C\,2^{-j}\,\|H\|_{L^\infty} \frac{\|T\|(\B_{2^{-j}})}{2^{-jm}} \leq C\,2^{-j}\,\bA.
\end{equation}
Therefore, arguing via a dyadic decomposition we deduce for every $2^{-J}<s<4$
\begin{align*}
\int_{\B_4\setminus \B_s} \frac{1}{|p|^m}
\left|\frac{p^\perp}{|p|}\right|^2 d \|T\|(p) &\leq \sum_{j=0}^{J-1} 
\int_{\B_{2^{2-j}}\setminus \B_{2^{1-j}}} \frac{1}{|p|^m}
\left|\frac{p^\perp}{|p|}\right|^2 d \|T\|(p)\\
& \leq \sum_{j=0}^{J-1}\left[
\frac{\|T\|(\B_{2^{2-j}})}{\omega_m\,\left(2^{2-j}\right)^m}
-\frac{\|T\|(\B_{2^{1-j}})}{\omega_m\,\left(2^{1-j}\right)^m} +C\,2^{-j}\,\bA\right]\\
&\leq \frac{\|T\|(\B_{4})}{\omega_m\,4^m}
-\frac{\|T\|(\B_{s})}{\omega_m\,s^m} + C\,\bA\,
\end{align*}
where in the second line we used once more \eqref{e:stima stupida}. Letting $s\to 0$, we get
\[
 \int_{\B_4} \frac{1}{|p|^m}
\left|\frac{p^\perp}{|p|}\right|^2 d \|T\|(p)\leq \frac{\|T\|(\B_{4})}{\omega_m\,4^m}-\Theta(T,0)+C\,\bA.
\]
Next recall that, if assumption (ih2) holds, then 
\[
\bE(T, \bC_4)= \frac{\|T\|(\bC_4)}{\omega_m 4^m}-Q
\]
and since $\B_4\subset \bC_4$ and, by (ih3), $\Theta(T, 0)=Q$, we conclude
\begin{equation}\label{e:poincare1}
 \int_{\B_4} |p^\perp|^2 \,d \|T\|(p)\leq C (\bE(T, 
 \bC_4) + \bA).
\end{equation}
A simple computation gives
\begin{align*}
 |p^\perp|^2 
 &=|p-\p_{T_p\supp(T)}(p)+\p_{\pi_0}(p)-\p_{\pi_0}(p)|^2\\
 &\geq |p-\p_{\pi_0}(p)|^2-|\p_{T_y\supp(T)}(p)-\p_{\pi_0}(p)|^2-2 |p-\p_{\pi_0}(p)|\,|\p_{T_y\supp(T)}(p)-\p_{\pi_0}(p)|\\
 &\geq \frac{|y|^2}{2}-2|p|^2\cdot|\vec{T}-\vec{\pi_0}|^2\,,
\end{align*}
and therefore
\begin{align*}
  \int_{\B_{4}}  \left|y\right|^2 d \|T\|(p)
 &\leq C\, \left( \bE +\bA +\int_{\B_{4}}|\vec{T}-\vec{\pi_0}|^2\,d\|T\|\right) \\
 &\leq C\,(\bE+ \bA ).
\end{align*}
which concludes the proof of \eqref{e:poincare}, since, for $\eps$ small enough, Theorem \ref{t:height_bound} ensures that $\supp(T)\cap \bC_2\subset \B_4$.
\end{proof}

\begin{proof}[Proof of Theorem \ref{t:Imp_HeightBound}]
Combining Lemma \ref{l:L2Linf} and Lemma \ref{l:Poincare} with a simple scaling argument we get the desired conclusion.
\end{proof}

\begin{remark}
 In particular, if $T$ is an area minimizing current in a $k$-dimensional submanifold $M\subset \R^{m+n}$, $m< k\leq m+n$, then
 \[
 \bh(T, \bC_r(x_0,\pi_0)) \leq C(E^{\sfrac12}+r\bA)r.
 \] 
Notice that it is enough to prove that
 \begin{equation}\label{e:poincare_min}
 \int_{\bC_2}|y|^2\,d\|T\|\leq C  (E+\bA^2).
 \end{equation}
where $\bA$ is the norm of the second fundamental form of $M$.
To see this, we observe that in this situation
\begin{align*}
&  -\frac12 \int_{\B_r}
p\cdot H(p) \, \left(\frac{1}{\max\{s,|p|\}^n} -\frac{1}{4^n} \right) d\|T\|(y)
=\frac12 \int_{\B_4} |p^\perp \cdot H(p)| \, \left(\frac{1}{\max\{s,|p|\}^n} -\frac{1}{4^n} \right) d\|T\|(p)\\
&\leq \eps \int_{\B_4} \left|\frac{p^\perp}{|p|}\right|^2\, \left(\frac{1}{\max\{s,|p|\}^n} -\frac{1}{4^n} \right) d\|T\|(p) 
+\frac{C}{\eps}\int_{\B_4} |p|^2\,|H(p)|^2 \, \left(\frac{1}{\max\{s,|p|\}^n} -\frac{1}{4^n} \right) d\|T\|(p).
\end{align*}
Summing over the annuli as above we get
\[
 \int_{\B_4} \frac{1}{|p|^n}
\left|\frac{p^\perp}{|p|}\right|^2 d \|T\|(p)\leq \frac{\|T\|(\B_{4})}{\omega_n\,4^n}-\Theta(T,0)+C\,\bA^2+
\eps \int_{\B_4}\frac{1}{|p|^n} \,\left|\frac{p^\perp}{|p|}\right|^2 d\|T\|(p)
\]
that is
\[
  \int_{B_4} |y|^2 \,d \|T\|(p)\leq C\,(E +\,\bA^2).
\]
This last expression is analogous to \eqref{e:poincare1}, so we conclude as above. 

This seems to be the analogous of the improved monotonicity formula of \cite[Lemma A.1]{DS4}, and indeed we took advantage of the same orthogonality property of the mean curvature. 
\end{remark}

\part{Center Manifold and Normal Approximation}

\section{Center Manifold: Construction algorithm}\label{s:CM}

The goal of this section and the next is to specify the algorithm
leading to the center manifold and the normal approximation, together with their properties, which will be fundamental in the asymptotic analysis of the frequency function of the third section. The proofs of the various statements
are all deferred to a later section.

 The following assumptions will hold in all the statements of this section.

\begin{ipotesi}\label{ipotesi}
$\eps_0\in ]0,1]$ is a fixed constant and 
$T^0$ is an $m$-dimensional semicalibrated integral current of $\mathbb R^{m+n}$ with finite mass. Moreover
\begin{gather}
\Theta (0, T^0) = Q\quad \mbox{and}\quad \partial T^0 \res \B_{6\sqrt{m}} = 0,\label{e:basic}\\
\quad \|T^0\| (\B_{6\sqrt{m} \rho}) \leq \big(\omega_m Q (6\sqrt{m})^m + \eps_2^2\big)\,\rho^m
\quad \forall \rho\leq 1,\label{e:basic2}\\
\bE\left(T^0,\B_{6\sqrt{m}}\right)=\bE\left(T^0,\B_{6\sqrt{m}},\pi_0\right),\label{e:pi0_ottimale}\\
\bmo := \max \left\{\|d\omega\|_{C^{1,\eps_0}}^2, \bE\left(T^0,\B_{6\sqrt{m}}\right)\right\} \leq \eps_2^2 \leq 1\, .\label{e:small ex}
\end{gather}
$\eps_2$ is a positive number whose choice will be specified in each statement. Notice that $\bOmega$ of Lemma \ref{l:quasi_minimalita} satisfies $\bOmega \leq \bmo^{\sfrac12}$. 
\end{ipotesi}
Constants which depend only upon $m,n$ and $Q$ will be called geometric and usually denoted by $C_0$.

The next lemma is a standard consequence of the monotonicity formula and the lower density bound for semicalibrated currents in Lemma \ref{t:height_bound0} (to ensure the convergence of supports), together with Assumption \ref{ipotesi}. Its proof can be found in \cite{DS4}.

\begin{lemma}[{\cite[Lemma 1.6]{DS4}}]\label{l:tecnico1}
There are positive constants $C_0 (m,n,Q)$ and $c_0 (m,n, Q)$ with the following property.
If $T^0$ is as in Assumption \ref{ipotesi}, $\eps_2 < c_0$ and $T:= T^0 \res \B_{23\sqrt{m}/4}$, then:
\begin{align}
&\partial T \res \bC_{11\sqrt{m}/2} (0, \pi_0)= 0\, ,\quad (\p_{\pi_0})_\sharp T\res \bC_{11 \sqrt{m}/2} (0, \pi_0) = Q \a{B_{11\sqrt{m}/2} (0, \pi_0)}\label{e:geo semplice 1}\\
&\quad\qquad\qquad\mbox{and}\quad\bh (T, \bC_{5\sqrt{m}} (0, \pi_0)) \leq C_0 \bmo^{\sfrac{1}{2m}}\, .\label{e:pre_height}
\end{align}
In particular for each $x\in B_{11\sqrt{m}/2} (0, \pi_0)$ there is  a point
$p\in \supp (T)$ with $\p_{\pi_0} (p)=x$.
\end{lemma}

From now on we will always work with the current $T$ of Lemma \ref{l:tecnico1}.
We specify next some notation which will be recurrent in the following when dealing with cubes of $\pi_0=\R^m\times\{0\}$.
For each $j\in \N$, $\sC^j$ denotes the family of closed cubes $L$ of $\pi_0$ of the form 
\begin{equation}\label{e:cube_def}
[a_1, a_1+2\ell] \times\ldots  \times [a_m, a_m+ 2\ell] \times \{0\}\subset \pi_0\, ,
\end{equation}
where $2\,\ell = 2^{1-j} =: 2\,\ell (L)$ is the side-length of the cube, 
$a_i\in 2^{1-j}\Z$ $\forall i$ and we require in
addition $-4 \leq a_i \leq a_i+2\ell \leq 4$. 
To avoid cumbersome notation, we will usually drop the factor $\{0\}$ in \eqref{e:cube_def} and treat each cube, its subsets and its points as subsets and elements of $\mathbb R^m$. Thus, for the {\em center $x_L$ of $L$} we will use the notation $x_L=(a_1+\ell, \ldots, a_m+\ell)$, although the precise one is $(a_1+\ell, \ldots, a_m+\ell, 0, \ldots , 0)$.
Next we set $\sC := \bigcup_{j\in \N} \sC^j$. 
If $H$ and $L$ are two cubes in $\sC$ with $H\subset L$, then we call $L$ an {\em ancestor} of $H$ and $H$ a {\em descendant} of $L$. When in addition $\ell (L) = 2\ell (H)$, $H$ is {\em a son} of $L$ and $L$ {\em the father} of $H$.

\begin{definition}\label{e:whitney} A Whitney decomposition of $[-4,4]^m\subset \pi_0$ consists of a closed set $\bGam\subset [-4,4]^m$ and a family $\mathscr{W}\subset \sC$ satisfying the following properties:
\begin{itemize}
\item[(w1)] $\bGam \cup \bigcup_{L\in \mathscr{W}} L = [-4,4]^m$ and $\bGam$ does not intersect any element of $\mathscr{W}$;
\item[(w2)] the interiors of any pair of distinct cubes $L_1, L_2\in \mathscr{W}$ are disjoint;
\item[(w3)] if $L_1, L_2\in \mathscr{W}$ have nonempty intersection, then $\frac{1}{2}\ell (L_1) \leq \ell (L_2) \leq 2\, \ell (L_1)$.
\end{itemize}
\end{definition}

Observe that (w1) - (w3) imply 
\begin{equation}\label{e:separazione}
{\rm sep}\, (\bGam, L) := \inf \{ |x-y|: x\in L, y\in \bGam\} \geq 2\ell (L)  \quad\mbox{for every $L\in \mathscr{W}$.}
\end{equation}
However, we do {\em not} require any inequality of the form 
${\rm sep}\, (\bGam, L) \leq C \ell (L)$, although this would be customary for what is commonly 
called a Whitney decomposition in the literature.

\subsection{Parameters} The algorithm for the construction of the center manifold involves several parameters which depend in
a complicated way upon several quantities and estimates. We introduce these parameters and specify some relations among them
in the following

\begin{ipotesi}\label{parametri}
$C_e,C_h,\beta_2,\delta_2, M_0$ are positive real numbers and $N_0$ a natural number for which we assume
always
\begin{gather}
\beta_2 = 4\,\delta_2 = \min \left\{\frac{1}{2m}, \frac{\beta_0}{100},\eps_0\right\}, \quad
\mbox{where $\beta_0$ is the constant of \cite[Theorem~1.4]{DS3},}\label{e:delta+beta}\\
M_0 \geq C_0 (m,n,\bar{n},Q) \geq 4\,  \quad
\mbox{and}\quad \sqrt{m} M_0 2^{7-N_0} \leq 1\, . \label{e:N0}
\end{gather}
\end{ipotesi}

As we can see, $\beta_2$ and $\delta_2$ are fixed. The other parameters are not fixed but are subject to further restrictions in the various statements, respecting the following ``hierarchy''. As already mentioned, ``geometric constants'' are assumed to depend only upon $m, n$ and $Q$. The dependence of other constants upon the various parameters $p_i$ will be highlighted using the notation $C = C (p_1, p_2, \ldots)$.

\begin{ipotesi}[Hierarchy of the parameters]\label{i:parametri}
In all the statements of the paper
\begin{itemize}
\item[(a)] $M_0$ is larger than a geometric constant (cf. \eqref{e:N0}) or larger than a constant $C (\delta_2)$, see Proposition~\ref{p:splitting};
\item[(b)] $N_0$ is larger than $C (\beta_2, \delta_2, M_0)$ (see for instance \eqref{e:N0}, Proposition \ref{p:splitting} and Proposition \ref{p:compara});
\item[(c)] $C_e$ is larger than $C(\beta_2, \delta_2, M_0, N_0)$ (see the statements
of Proposition \ref{p:whitney}, Theorem \ref{t:cm} and Proposition \ref{p:splitting_II});
\item[(d)] $C_h$ is larger than $C(\beta_2, \delta_2, M_0, N_0, C_e)$ (see Propositions~\ref{p:whitney} and \ref{p:separ});
\item[(e)] $\eps_2$ is smaller than $c(\beta_2, \delta_2, M_0, N_0, C_e, C_h)$ (which will always be positive).
\end{itemize}
\end{ipotesi}

The functions $C$ and $c$ will vary in the various statements: 
the hierarchy above guarantees however that there is a choice of the parameters for which {\em all} the restrictions required in the statements of the next propositions are simultaneously satisfied. In fact it is such a choice which is then made in the second part of this work. To simplify our exposition, for smallness conditions on $\eps_2$ as in (e) we will use the sentence ``$\eps_2$ is sufficiently small''.

\subsection{The Whitney decomposition} 
Thanks to Lemma \ref{l:tecnico1}, for every $L\in \sC$,  we may choose $y_L\in \pi_0^\perp$ so that $p_L := (x_L, y_L)\in \supp (T)$ (recall that $x_L$ is the center of $L$). $y_L$ is in general not unique and we fix an arbitrary choice.
A more correct notation for $p_L$ would be $x_L + y_L$, however, following \cite{DS4}, we  abuse the notation slightly in using $(x,y)$ instead of $x+y$ and, consistently, $\pi_0\times \pi_0^\perp$ instead of $\pi_0 + \pi_0^\perp$.

\begin{definition}[Refining procedure]\label{d:refining_procedure}
For $L\in \sC$ we set $r_L:= M_0 \sqrt{m} \,\ell (L)$ and 
$\B_L := \B_{64 r_L} (p_L)$. We next define the families of cubes $\sS\subset\sC$ and $\sW = \sW_e \cup \sW_h \cup \sW_n \subset \sC$ with the convention that
$\sS^j = \sS\cap \sC^j, \sW^j = \sW\cap \sC^j$ and $\sW^j_{\square} = \sW_\square \cap \sC^j$ for $\square = h,n, e$. We define $\sW^i = \sS^i = \emptyset $ for $i < N_0$. We proceed with $j\geq N_0$ inductively: if no ancestor of $L\in \sC^j$ is in $\sW$, then 
\begin{itemize}
\item[(EX)] $L\in \sW^j_e$ if $\bE (T, \B_L) > C_e \bmo\, \ell (L)^{2-2\delta_2}$;
\item[(HT)] $L\in \sW_h^j$ if $L\not \in \mathscr{W}_e^j$ and $\bh (T, \B_L) > C_h \bmo^{\sfrac{1}{2m}} \ell (L)^{1+\beta_2}$;
\item[(NN)] $L\in \sW_n^j$ if $L\not\in \sW_e^j\cup \sW_h^j$ but it intersects an element of $\sW^{j-1}$;
\end{itemize}
if none of the above occurs, then $L\in \sS^j$.
We finally set
\begin{equation}\label{e:bGamma}
\bGam:= [-4,4]^m \setminus \bigcup_{L\in \sW} L = \bigcap_{j\geq N_0} \bigcup_{L\in \sS^j} L.
\end{equation}
\end{definition}
Observe that, if $j>N_0$ and $L\in \sS^j\cup \sW^j$, then necessarily its father belongs to $\sS^{j-1}$.

\begin{proposition}[Whitney decomposition {\cite[Proposition 1.11]{DS4}}]\label{p:whitney}
Let Assumptions \ref{ipotesi} and \ref{parametri} hold and let $\eps_2$ be sufficiently small.
Then $(\bGam, \mathscr{W})$ is a Whitney decomposition of $[-4,4]^m \subset \pi_0$.
Moreover, for any choice of $M_0$ and $N_0$, there is $C^\star := C^\star (M_0, N_0)$ such that,
if $C_e \geq C^\star$ and $C_h \geq C^\star C_e$, then 
\begin{equation}\label{e:prima_parte}
\sW^{j} = \emptyset \qquad \mbox{for all $j\leq N_0+6$.}
\end{equation}
Moreover, the following 
estimates hold with $C = C(\beta_2, \delta_2, M_0, N_0, C_e, C_h)$:
\begin{gather}
\bE (T, \B_J) \leq C_e \bmo^{}\, \ell (J)^{2-2\delta_2} \quad \text{and}\quad
\bh (T, \B_J) \leq C_h \bmo^{\sfrac{1}{2m}} \ell (J)^{1+\beta_2}
\quad \forall J\in \sS\,, \label{e:ex+ht_ancestors}\\
 \bE (T, \B_L) \leq C\, \bmo^{}\, \ell (L)^{2-2\delta_2}\quad \text{and}\quad
\bh (T, \B_L) \leq C\, \bmo^{\sfrac{1}{2m}} \ell (L)^{1+\beta_2}
\quad \forall L\in \sW\, . \label{e:ex+ht_whitney}
\end{gather}
\end{proposition}

\subsection{Construction algorithm} We will see below that in (a suitable portion of) each $\B_L$ the current $T$ can be approximated efficiently with a graph of a Lipschitz multiple-valued map. The average of the sheets of this approximating map will then be used as a local model for the center manifold, while the map itself, suitably reparametrized on the center manifold and extended, will give us the normal approximation.

\begin{definition}[$\pi$-approximations]\label{d:pi-approximations}
Let $L\in \sS\cup \sW$ and $\pi$ be an $m$-dimensional plane. If $T\res\bC_{32 r_L} (p_L, \pi)$ fulfills 
the assumptions of \cite[Theorem 1.1]{DSS2} in the cylinder $\bC_{32 r_L} (p_L, \pi)$, then the resulting map $f: B_{8r_L} (p_L, \pi)  \to \Iq (\pi^\perp)$ given by \cite[Theorem 1.4]{DSS2} is called a {\em $\pi$-approximation of $T$ in $\bC_{8 r_L} (p_L, \pi)$}. 
\end{definition}

\begin{lemma}[{\cite[Lemma 1.15]{DS4}}]\label{l:tecnico2}
Let the assumptions of Proposition \ref{p:whitney} hold and assume $C_e \geq C^\star$ and $C_h \geq C^\star C_e$ for a suitably large $C^\star (M_0, N_0)$. For each $L\in \sW\cup \sS$ we choose a plane $\pi_L$ which optimizes the excess and the height in $\B_L$. 
For any choice of the other parameters,
if $\eps_2$ is sufficiently small, then $T \res \bC_{32 r_L} (p_L, \pi_L)$ satisfies the assumptions of
\cite[Theorem 1.5]{DSS2} for any $L\in \sW\cup \sS$. 
\end{lemma}

As in \cite{DS4}, we wish to find a suitable smoothing of the average of the $\pi$-approximation $\etaa\circ f$. However the smoothing procedure is more complicated in our case: rather than smoothing by convolution, we need to solve a suitable
elliptic system of partial differential equations (cf. \cite{DSS3}) and then estimate it carefully.

\begin{definition}[Smoothing]\label{d:smoothing}
Let $L$ and $\pi_L$ be as in Lemma \ref{l:tecnico2} and denote by $f_L$ the corresponding $\pi_L$-approximation. We let $h_L$ be a solution (provided it exists) of 
\begin{equation}\label{e:ellittico-10}
\left\{
\begin{array}{ll}
\mathscr{L}_{L} h_L = \mathscr{F}_{L}\\ \\
\left. h_L\right|_{\partial B_{8r_L} (p_L, \pi_L)} = \etaa\circ f_L\, ,
\end{array}\right.
\end{equation}
where $\mathscr{L}_L$ is a suitable second order linear elliptic operator with constant coefficients and $\mathscr{F}_L$ a suitable affine map: the precise expressions for $\mathscr{L}_L$ and $\mathscr{F}_L$ depend on a careful Taylor expansion of the first variations formulas and are given in Proposition \ref{p:pde}. The map $h_L$ is the {\em tilted interpolating function} relative to $L$. 
\end{definition} 

In what follows we will deal with graphs of multivalued functions $f$ in several system of coordinates. These objects can
be naturally seen as currents $\bG_f$ (see \cite{DS2}) and in this respect we will use extensively the notation and results of \cite{DS2} (therefore $\gr (f)$ will denote the ``set-theoretic'' graph).

\begin{lemma}\label{l:tecnico3}
Let the assumptions of Proposition \ref{p:whitney} hold and assume $C_e \geq C^\star$ and $C_h \geq C^\star C_e$ (where $C^\star$ is the
constant of Lemma \ref{l:tecnico2}). For any choice of the other parameters,
if $\eps_2$ is sufficiently small the following holds. For any $L\in \sW\cup \sS$, there is a unique solution $h_L$ of \eqref{e:ellittico-10} and there is a smooth $g_L: B_{4r_L} (z_L, \pi_0)\to \pi_0^\perp$ such that 
$\bG_{g_L} = \bG_{h_L}\res \bC_{4r_L} (p_L, \pi_0)$.
\end{lemma}

Fix next a $\vartheta \in C^\infty_c(\left[-\frac{17}{16}, \frac{17}{16}\right]^m,[0,1])$ which is nonnegative and is identically $1$ on $[-1, 1]^m$. We use $\vartheta$ to construct a partition of unity for the Whitney decomposition and glue all the interpolated functions together.

\begin{definition}[Interpolating functions]\label{d:glued}
The maps $h_L$ and $g_L$ in Lemma \ref{l:tecnico3} will be called, respectively, the
{\em tilted $L$-interpolating function} and the {\em $L$-interpolating function}.
For each $j$ let $\sP^j := \sS^j \cup \bigcup_{i=N_0}^j \sW^i$ and
for $L\in \sP^j$ define $\vartheta_L (y):= \vartheta (\frac{y-x_L}{\ell (L)})$. The map
\begin{equation}
\varphi_j := \frac{\sum_{L\in \sP^j} \vartheta_L g_L}{\sum_{L\in \sP^j} \vartheta_L} \qquad \mbox{on $]-4,4[^m$},
\end{equation}
will be called the {\em glued interpolation} at the step $j$.
\end{definition}

\begin{theorem}[Existence of the center manifold]\label{t:cm}
Assume that the hypotheses of Lemma \ref{l:tecnico2} hold and
let $\kappa := \min \{\eps_0/2, \beta_2/4\}$. For any choice of the other parameters,
if $\eps_2$ is sufficiently small, then
\begin{itemize}
\item[(i)] $\|D\varphi_j\|_{C^{2, \kappa}} \leq C \bmo^{\sfrac{1}{2}}$ and $\|\varphi_j\|_{C^0}
\leq C \bmo^{\sfrac{1}{2m}}$, with $C = C(\beta_2, \delta_2, M_0, N_0, C_e, C_h)$.
\item[(ii)] if $L\in \sW^i$ and $H$ is a cube concentric to $L$ with $\ell (H)=\frac{9}{8} \ell (L)$, then $\varphi_j = \varphi_k$ on $H$ for any $j,k\geq i+2$. 
\item[(iii)] $\varphi_j$ converges in $C^3$ to a map $\phii$ and $\cM:= \gr (\phii|_{]-4,4[^m}
)$ is a $C^{3,\kappa}$ submanifold of $\Sigma$.
\end{itemize}
\end{theorem}

\begin{definition}[Whitney regions]\label{d:cm}
The manifold $\cM$ in Theorem \ref{t:cm} is called
{\em a center manifold of $T$ relative to $\pi_0$} and 
$(\bGam, \sW)$ the {\em Whitney decomposition associated to $\cM$}. 
Setting $\Phii(y) := (y,\phii(y))$, we call
$\Phii (\bGam)$ the {\em contact set}.
Moreover, to each $L\in \sW$ we associate a {\em Whitney region} $\cL$ on $\cM$ as follows:
\begin{itemize}
\item[(WR)] $\cL := \Phii (H\cap [-\frac{7}{2},\frac{7}{2}]^m)$, where $H$ is the cube concentric to $L$ with $\ell (H) = \frac{17}{16} \ell (L)$.
\end{itemize}
\end{definition}

\section{Normal Approximation: main estimates}

\subsection{The $\cM$-normal approximation and related estimates}

In what follows we assume that the conclusions of Theorem \ref{t:cm} apply and denote by $\cM$ the corresponding
center manifold. For any Borel set $\cV\subset \cM$ we will denote 
by $|\cV|$ its $\cH^m$-measure and will write $\int_\cV f$ for the integral of $f$
with respect to $\cH^m$. 
$\cB_r (q)$ denotes the geodesic balls in $\cM$. Moreover, we refer to \cite{DS2}
for all the relevant notation pertaining to the differentiation of (multiple valued)
maps defined on $\cM$, induced currents, differential geometric tensors and so on.

\begin{ipotesi}\label{intorno_proiezione}
We fix the following notation and assumptions.
\begin{itemize}
\item[(U)] $\bU :=\big\{x\in \R^{m+n} : \exists !\, y = \p (x) \in \cM \mbox{ with $|x- y| <1$ and
$(x-y)\perp \cM$}\big\}$.
\item[(P)] $\p : \bU \to \cM$ is the map defined by (U).
\item[(R)] For any choice of the other parameters, we assume $\eps_2$ to be so small that
$\p$ extends to $C^{2, \kappa}(\bar\bU)$ and
$\p^{-1} (y) = y + \overline{B_1 (0, (T_y \cM)^\perp)}$ for every $y\in \cM$.
\item[(L)] We denote by $\partial_l \bU := \p^{-1} (\de \cM)$ 
the {\em lateral boundary} of $\bU$.
\end{itemize}
\end{ipotesi}

The following is then a corollary of Theorem \ref{t:cm} and the construction algorithm.

\begin{corollary}\label{c:cover}
Under the hypotheses of Theorem \ref{t:cm} and of Assumption \ref{intorno_proiezione}
we have:
\begin{itemize}
\item[(i)] $\supp (\partial (T\res \bU)) \subset \partial_l \bU$, 
$\supp (T\res [-\frac{7}{2}, \frac{7}{2}]^m \times \R^n) \subset \bU$ 
and $\p_\sharp (T\res \bU) = Q \a{\cM}$;
\item[(ii)] $\supp (\langle T, \p, \Phii (q)\rangle) \subset 
\big\{y\, : |\Phii (q)-y|\leq C \bmo^{\sfrac{1}{2m}} 
\ell (L)^{1+\beta_2}\big\}$ for every $q\in L\in \sW$, where\\
$C= C(\beta_2, \delta_2, M_0, N_0,  C_e, C_h)$;
\item[(iii)]  $\langle T, \p, p\rangle = Q \a{p}$ for every $p\in \Phii (\bGam)$.
\end{itemize}
\end{corollary}

The main goal of this section is to couple the center manifold of Theorem \ref{t:cm} with a good approximating map defined on it.

\begin{definition}[$\cM$-normal approximation]\label{d:app}
An {\em $\cM$-normal approximation} of $T$ is given by a pair $(\cK, F)$ such that
\begin{itemize}
\item[(A1)] $F: \cM\to \Iq (\bU)$ is Lipschitz (with respect to the geodesic distance on $\cM$) and takes the special form 
$F (x) = \sum_i \a{x+N_i (x)}$, with $N_i (x)\perp T_x \cM$ for every $x$ and $i$.
\item[(A2)] $\cK\subset \cM$ is closed, contains $\Phii \big(\bGam\cap [-\frac{7}{2}, \frac{7}{2}]^m\big)$ and $\bT_F \res \p^{-1} (\cK) = T \res \p^{-1} (\cK)$.
\end{itemize}
The map $N = \sum_i \a{N_i}:\cM \to \Iq (\R^{m+n})$ is {\em the normal part} of $F$.
\end{definition}

In the definition above it is not required that the map $F$ approximates efficiently the current
outside the set $\Phii \big(\bGam\cap [-\frac{7}{2}, \frac{7}{2}]^m\big)$. However, all the maps constructed
in this paper will approximate $T$ with a high degree of accuracy
in each Whitney region: such estimates are detailed
in the next theorem. In order to simplify the notation, we will use $\|N|_{\cV}\|_{C^0}$ (or $\|N|_{\cV}\|_0$) to denote the number
$\sup_{x\in \cV} \cG (N (x), Q\a{0})$.

\begin{theorem}[Local estimates for the $\cM$-normal approximation]\label{t:approx}
Let $\gamma_2 := \frac{\beta_0}{4}$, with $\beta_0$ the constant
of \cite[Theorem 1.4]{DS3}.
Under the hypotheses of Theorem \ref{t:cm} and Assumption~\ref{intorno_proiezione},
if $\eps_2$ is suitably small (depending upon all other parameters), then
there is an $\cM$-normal approximation $(\cK, F)$ such that
the following estimates hold on every Whitney region $\cL$ associated to
a cube $L\in \sW$, with constants $C = C(\beta_2, \delta_2, M_0, N_0, C_e, C_h)$:
\begin{gather}
\Lip (N|
_\cL) \leq C \bmo^{\gamma_2} \ell (L)^{\gamma_2} \quad\mbox{and}\quad  \|N|
_\cL\|_{C^0}\leq C \bmo^{\sfrac{1}{2m}} \ell (L)^{1+\beta_2},\label{e:Lip_regional}\\
|\cL\setminus \cK| + \|\bT_F - T\| (\p^{-1} (\cL)) \leq C \bmo^{1+\gamma_2} \ell (L)^{m+2+\gamma_2},\label{e:err_regional}\\
\int_{\cL} |DN|^2 \leq C \bmo \,\ell (L)^{m+2-2\delta_2}\, .\label{e:Dir_regional}
\end{gather}
Moreover, for any $a>0$ and any Borel $\cV\subset \cL$, we have (for $C=C(\beta_2, \delta_2, M_0, N_0, C_e, C_h)$)
\begin{equation}\label{e:av_region}
\int_\cV |\etaa\circ N| \leq 
 C \bmo \left(\ell (L)^{m+3+\sfrac{\beta_2}{3}} + a\,\ell (L)^{2+\sfrac{\gamma_2}{2}}|\cV|\right)  + \frac{C}{a} 
\int_\cV \cG \big(N, Q \a{\etaa\circ N}\big)^{2+\gamma_2}\, .
\end{equation} 
\end{theorem}

From \eqref{e:Lip_regional} - \eqref{e:Dir_regional} it is not difficult to infer analogous ``global versions'' of the estimates.

\begin{corollary}[Global estimates]\label{c:globali} Let $\cM'$ be
the domain $\Phii \big([-\frac{7}{2}, \frac{7}{2}]^m\big)$ and $N$ the map of Theorem \ref{t:approx}. Then,  (again with $C = C(\beta_2, \delta_2, M_0, N_0, C_e, C_h)$)
\begin{gather}
\Lip (N|_{\cM'}) \leq C \bmo^{\gamma_2} \quad\mbox{and}\quad \|N|_{\cM'}\|_{C^0}
\leq C \bmo^{\sfrac{1}{2m}},\label{e:global_Lip}\\ 
|\cM'\setminus \cK| + \|\bT_F - T\| (\p^{-1} (\cM')) \leq C \bmo^{1+\gamma_2},\label{e:global_masserr}\\
\int_{\cM'} |DN|^2 \leq C \bmo\, .\label{e:global_Dir}
\end{gather}
\end{corollary}

\subsection{Height bound and separation}
We now analyze more in detail the consequences of the various stopping conditions for the cubes in $\sW$. 
We first deal with $L\in \sW_h$.

\begin{proposition}[Separation]\label{p:separ}
There is a constant $C^\sharp (M_0) > 0$ with the following property.
Assume the hypotheses of Theorem \ref{t:approx} and in addition
$C_h^{2m} \geq C^\sharp C_e$. 
If $\eps_2$ is sufficiently small, then the following conclusions hold for every $L\in \sW_h$:
\begin{itemize}
\item[(S1)] $\Theta (T, p) \leq Q - \frac{1}{2}$ for every $p\in \B_{16 r_L} (p_L)$.
\item[(S2)] $L\cap H= \emptyset$ for every $H\in \sW_n$
with $\ell (H) \leq \frac{1}{2} \ell (L)$;
\item[(S3)] $\cG \big(N (x), Q \a{\etaa \circ N (x)}\big) \geq \frac{1}{4} C_h \bmo^{\sfrac{1}{2m}}
\ell (L)^{1+\beta_2}$  for every  $x\in \Phii (B_{2 \sqrt{m} \ell (L)} (x_L, \pi_0))$.
\end{itemize}
\end{proposition}

A simple corollary of the previous proposition is the following.

\begin{corollary}\label{c:domains}
Given any $H\in \sW_n$ there is a chain $L =L_0, L_1, \ldots, L_j = H$ such that:
\begin{itemize}
\item[(a)] $L_0\in \sW_e$ and $L_i\in \sW_n$ for all $i>0$; 
\item[(b)] $L_i\cap L_{i-1}\neq\emptyset$ and $\ell (L_i) = \frac{1}{2} \ell (L_{i-1})$ for all $i>0$.
\end{itemize}
In particular,  $H\subset B_{3\sqrt{m}\ell (L)} (x_L, \pi_0)$.
\end{corollary}

We use this last corollary to partition $\sW_n$.

\begin{definition}[Domains of influence]\label{d:domains}
We first fix an ordering of the cubes in $\sW_e$ as $\{J_i\}_{i\in \mathbb N}$ so that their sidelenghts do not increase. Then $H\in \sW_n$
belongs to $\sW_n (J_0)$ (the domain of influence of $J_0$) if there is a chain as in Corollary \ref{c:domains} with $L_0 = J_0$.
Inductively, $\sW_n (J_r)$ is the set of cubes $H\in \sW_n \setminus \cup_{i<r} \sW_n (J_i)$ for which there is
a chain as in Corollary \ref{c:domains} with $L_0 = J_r$.
\end{definition}

\subsection{Splitting before tilting I} The following proposition contains a ``typical'' splitting-before-tilting phenomenon: the key assumption of the 
theorem (i.e. $L\in \sW_e$) is that the excess does not decay at some given scale (``tilting'') and the main conclusion \eqref{e:split_2} implies a certain amount of separation between the sheets of the current (``splitting'').

\begin{proposition}(Splitting I)\label{p:splitting}
There are functions $C_1 (\delta_2), C_2 (M_0, \delta_2)$ such that, if $M_0 \geq C_1 ( \delta_2)$, $N_0 \geq C_2 (M_0, \delta_2)$, if
the hypotheses of Theorem~\ref{t:approx} hold and if $\eps_2$ is chosen sufficiently small,
then the following holds. If $L\in \sW_e$, $q\in \pi_0$ with $\dist (L, q) \leq 4\sqrt{m} \,\ell (L)$ and $\Omega = \Phii (B_{\ell (L)/4} (q, \pi_0))$, then (with $C, C_3 = C(\beta_2, \delta_2, M_0, N_0, C_e, C_h)$):
\begin{align}
&C_e \bmo \ell(L)^{m+2-2\delta_2} \leq \ell (L)^m \bE (T, \B_L) \leq C \int_\Omega |DN|^2\, ,\label{e:split_1}\\
&\int_{\cL} |DN|^2 \leq C \ell (L)^m \bE (T, \B_L) \leq C_3 \ell (L)^{-2} \int_\Omega |N|^2\, . \label{e:split_2}
\end{align}
\end{proposition}

\subsection{Persistence of $Q$ points} We next state two important properties triggered by the existence of $p\in \supp (T)$ with $\Theta (p, T)=Q$,
both related to the splitting before tilting. The new height bound of Theorem \ref{t:Imp_HeightBound} enters here, where, combined with a decay statement for the excess a la De Giorgi, allows us to prove that the approximating maps to the current inherit points of maximal multiplicity from it.

\begin{proposition}(Splitting II)\label{p:splitting_II}
Let the hypotheses of Theorem~\ref{t:cm} hold and assume $\eps_2$ is sufficiently small. For any
$\alpha, \bar{\alpha}, \hat\alpha >0$, there is a constant $\eps_3=\eps_3(\alpha, \bar \alpha, \hat \alpha, \beta_2, \delta_2, M_0,N_0, C_e, C_h) >0$ as follows. If, for some $s\leq 1$
\begin{equation}\label{e:supL}
\sup \big\{\ell (L): L\in \sW, L\cap B_{3s} (0, \pi_0) \neq \emptyset\big\} \leq s\, ,
\end{equation}
\begin{equation}\label{e:many_Q_points}
\cH^{m-2+\alpha}_\infty \big(\{\Theta (T, \cdot) = Q\}\cap \B_{s}\big) \geq \bar{\alpha} s^{m-2+\alpha},
\end{equation}
and $C_e\geq C(\hat\alpha, M_0, N_0,\delta_2)$ and $\min \{s, \bmo\}\leq \eps_3$, then
\[
\sup \big\{ \ell (L): L\in \sW_e \mbox{ and } L\cap B_{19 s/16} (0, \pi_0)\neq \emptyset\big\} 
\leq \hat{\alpha} s\, .
\]
\end{proposition}

\begin{proposition}(Persistence of $Q$-points)\label{p:persistence}
Assume the hypotheses of Proposition \ref{p:splitting} hold.
For every $\eta_2>0$ there are $\bar{s}, \bar{\ell} > 0$, depending upon $\eta_2, \beta_2, \delta_2, M_0, N_0, C_e$ and $C_h$, such that,
if $\eps_2$ is sufficiently small, then the following property holds. 
If $L\in \sW_e, \ell (L)\leq \bar\ell$, $\Theta (T, p) = Q$ and
$\dist (\p_{\pi_0} (\p (p)), L) \leq 4 \sqrt{m} \,\ell (L)$, then
\begin{equation}\label{e:persistence1}
\mint_{\cB_{\bar{s} \ell (L)} (\p (p))} \cG \big(N, Q \a{\etaa\circ N}\big)^2 \leq \frac{\eta_2}{\ell(L)^{m-2}}
\int_{\cB_{\ell (L)} (\p (p))} |DN|^2\, .
\end{equation}
\end{proposition}

\subsection{Comparison between different center manifolds} 
We list here a final key consequence of the splitting before tilting phenomenon. $\iota_{0,r}$ denotes the map $z\mapsto \frac{z}{r}$.

\begin{proposition}[Comparing center manifolds]\label{p:compara}
There is a geometric constant $C_0$ and a function $\bar{c}_s (\beta_2, \delta_2, M_0, N_0, C_e, C_h) >0$ with
the following property. Assume the hypotheses of Proposition \ref{p:splitting}, $N_0 \geq C_0$, $c_s := \frac{1}{64\sqrt{m}}$
and $\eps_2$ is sufficiently small. If for some $r\in ]0,1[$:
\begin{itemize}
\item[(a)] $\ell (L) \leq  c_s \rho$ for every $\rho> r$ and every
$L\in \sW$ with $L\cap B_\rho (0, \pi_0)\neq \emptyset$;
\item[(b)] $\bE (T, \B_{6\sqrt{m} \rho}) < \eps_2$ for every $\rho>r$;
\item[(c)] there is $L\in \sW$ such that $\ell (L) \geq c_s r $ and $L\cap \bar B_r (0, \pi_0)\neq\emptyset$;
\end{itemize}
then
\begin{itemize}
\item[(i)] the current $T':= (\iota_{0,r})_\sharp T \res \B_{6\sqrt{m}}$ satisfies the assumptions
of Theorem \ref{t:approx} for
some plane $\pi$ in place of $\pi_0$;
\item[(ii)] for the center manifold $\cM'$ of $T'$ relative to $\pi$ and the 
$\cM'$-normal approximation $N'$ as in Theorem \ref{t:approx}, we have
\begin{equation}\label{e:restart}
\int_{\cM'\cap \B_2} |N'|^2 \geq \bar{c}_s \max
\big\{\bE (T', \B_{6\sqrt{m}}), \|d\omega\|_{C^{1,\eps_0}}^2\big\}\, . 
\end{equation}
\end{itemize}
\end{proposition}

 \section{Center Manifold: proofs of the main results}
 
\subsection{Technical preliminaries and the Whitney decomposition}

The proof of Lemma \ref{l:tecnico1} is analogous to the one of \cite[Lemma 1.6]{DS4}, with $\Sigma=\R^{n+m}$ and using the height bounds of Theorem \ref{t:height_bound} (one should observe that the monotonicity formula holds for $T$, since it is a varifold with bounded mean curvature by Lemma \ref{l:quasi_minimalita}). 

For what concerns Proposition \ref{p:whitney}, the proof is contained in \cite[Proposition 1.11]{DS4}, where the following, more exhaustive, result is given. Notice that \cite[Proposition 1.11]{DS4} is a consequnce of the Assumptions \ref{ipotesi} and of the construction alghoritm for the Whitney decomposition, in particular no additional assumption on the current is needed.

\begin{proposition}[Tilting of optimal planes]\label{p:tilting opt}
Assume that the hypotheses of Assumptions \ref{ipotesi} and \ref{parametri} hold, that
$C_e \geq C^\star$ and $C_h \geq C^\star C_e$, where $C^\star (M_0, N_0)$ is the constant of the previous section. 
If $\eps_2$ is sufficiently small, then 
\begin{itemize}
\item[(i)] $\B_H\subset\B_L \subset \B_{5\sqrt{m}}$ for all $H, L\in \sW\cup \sS$ with $H\subset L$.
\end{itemize}
Moreover, if $H, L \in \sW\cup\sS$ and either $H\subset L$ or $H\cap L \neq \emptyset$ and $\frac{\ell (L)}{2} \leq \ell (H) \leq \ell (L)$, then the following holds, for $\bar{C} = \bar{C} (\beta_2, \delta_2, M_0, N_0, C_e)$ and $C = C(\beta_2, \delta_2, M_0, N_0, C_e, C_h)$:
\begin{itemize}
\item[(ii)] $|\pi_H-\pi_L| \leq \bar{C} \bmo^{\sfrac{1}{2}} \ell (L)^{1-\delta_2}$;
\item[(iii)] $|\pi_H-\pi_0| \leq  \bar C \bmo^{\sfrac{1}{2}}$;
\item[(iv)] $\bh (T, \bC_{36 r_H} (p_H, \pi_0)) \leq C \bmo^{\sfrac{1}{2m}} \ell (H)$ and $\supp (T) \cap \bC_{36 r_H} (p_H, \pi_0) \subset \B_H$; 
\item[(v)] $\bh (T, \bC_{36r_L} (p_L, \pi_H))\leq C \bmo^{\sfrac{1}{2m}} \ell (L)^{1+\beta_2}$
and $\supp (T) \cap \bC_{36 r_L} (p_L, \pi)\subset \B_L$.
\end{itemize}
In particular, the conclusions of Proposition~\ref{p:whitney} hold.
\end{proposition} 

 We finally come to Lemma \ref{l:tecnico2}, which is a corollary of our assumptions together with Proposition \ref{p:tilting opt}. It is the analogous of the first half of \cite[Lemma 1.15]{DS4} and its proof is contained in the following more general result, when $H=L$.
 
 \begin{proposition}[Existence of interpolating functions {\cite[Proposition 4.2]{DS4}}]\label{p:gira_e_rigira}
Assume the conclusions of Proposition \ref{p:tilting opt} apply.
If $\eps_2$ is sufficiently small, then the following holds. 
Let $H, L\in \sW\cup \sS$ be such that either $H\subset L$ or $H\cap L \neq \emptyset$ and
$\frac{\ell (L)}{2} \leq \ell (H) \leq \ell (L)$. Then, $(\p_{\pi_H})_\sharp T\res \bC_{32r_L} (p_L, \pi_H) = Q \a{B_{32r_L} (p_L, \pi_H))}$ and $T$ satisfies the assumptions of \cite[Theorem 1.4]{DS3} in the cylinder $\bC_{32 r_L} (p_L, \pi_H)$. 
\end{proposition}

We next generalize slightly the terminology given in the previous sections.

\begin{definition}\label{d:f_HL}
Let $H$ and $L$ be as in Proposition \ref{p:gira_e_rigira}. After applying \cite[Theorem~1.5]{DSS2} to $T_L\res \bC_{32r_L} (p_L, \pi_H)$ in the cylinder $\bC_{32 r_L} (p_L, \pi_H)$ we denote by $f_{HL}$ the corresponding $\pi_H$-approximation. However, rather then defining $f_{HL}$ on the disk $B_{8 r_L}  (p_L, \pi_H)$, by applying a translation we assume that the domain of $f_{HL}$ is the disk $B_{8r_L} (p_{HL}, \pi_H)$ where $p_{HL} = p_H + \p_{\pi_H} (p_L - p_H)$. Note in particular that $\bC_r (p_{HL}, \pi_H)$ equals $\bC_r (p_L, \pi_H)$, whereas $B_{8r_L} (p_{HL}, \pi_H) \subset p_H + \pi_H$ and $p_H\in \B_{8r_L}(p_{HL},\pi_H)$. 
\end{definition}

In particular observe that $f_{LL} = f_L$.

\subsection{From the first variation to the elliptic system} In this paragraph we derive the elliptic system $\mathscr{L}$ by studying which information the first variation of $T$ induces on the mean of the tilted $(H,L)$-interpolating function. This part, together with the next subsection, is different than \cite{DS4}, and closer to \cite{DSS3}.

In what follows we will consider elliptic systems of the following form. Given a vector valued map $v: p_H+ \pi_H\supset \Omega \to \pi_H^\perp$ and after introducing an orthonormal system of coordinates $x^1,\dots,x^m$ on $\pi_H$ and $y^1, \ldots , y^{n}$ on $\pi_H^\perp$, the system is given by the $n$ equations
\begin{equation}\label{e:elliptic}
\Delta v^k + \underbrace{({\bf L}_1)^k_{ij} \partial_j v^i + ({\bf L}_2)^k_i v^i}_{=:\mathscr{E}^k (v)} = \underbrace{({\bf L}_3)^k_i (x-x_H)^i +
({\bf L}_4)^k}_{=:\mathscr{F}^k} \, ,
\end{equation}
where we follow Einstein's summation convention and the tensors $\bL_i$ have constant coefficients. After introducing the operator $\mathscr{L} (v) = \Delta v + \mathscr{E} (v)$ we summarize the corresponding elliptic system \eqref{e:elliptic} as
\begin{equation}\label{e:ellittico}
\mathscr{L} (v) = \mathscr{F}\, .
\end{equation}
We then have a corresponding  weak formulation for $W^{1,2}$ solutions of \eqref{e:ellittico}, namely $v$ is a weak solution in a domain $D$ if the integral
\begin{equation}\label{e:weak_form}
\mathscr{I} (v, \zeta) := \int (Dv : D\zeta + (\mathscr{F}-  \mathscr{E} (v))\cdot \zeta)\, 
\end{equation}
vanishes for smooth test functions $\zeta$ with compact support in $D$.

\begin{proposition}\label{p:pde}
Let $H$ and $L$ be as in Proposition \ref{p:gira_e_rigira} (including the possibility that $H=L$) and
let $f_{HL}$ be as in Definition \ref{d:f_HL}.
Then, there exist tensors with constant coefficients 
$\bL_1, \ldots, \bL_4$ and a constant $C=C(M_0,N_0,C_e,C_h)$,  with the following properties:
\begin{itemize}
\item[(i)] The tensors depend upon $\omega$ and $|{\bf L}_1| + |{\bf L}_2| + |{\bf L}_3| + |{\bf L_4}|\leq C \bmo^{\sfrac{1}{2}}$. In particular $\|\mathscr{F}_H(v)\|_{C^0(B_{8r_L})}\leq  C\bmo^{\sfrac12}\,r_L^{1+\eps_0}$.
\item[(ii)] If $\mathscr{I}_H$, $\mathscr{L}_H$ and $\mathscr{F}_H$ are defined through \eqref{e:elliptic}, \eqref{e:ellittico} and \eqref{e:weak_form}, then
\begin{align}\label{e:pde2}
\mathscr{I}_{H} (\etaa \circ f_{HL}, \zeta) &\leq C \bmo \,
r_L^{m+2+\beta_2}\,\|D\zeta\|_{0}
\end{align}
for all $\zeta\in C^\infty_c (B_{8r_L}(p_{HL}, \pi_H), \pi_H^\perp)$.
\end{itemize}
\end{proposition}

\begin{proof}
Set for simplicity $\pi= \pi_{H}$, $\pi^\perp:=\pi_{H}^\perp$
$r= r_L$, $p= p_{HL}$, $f=f_{HL}$, $B= B_{8r} (p, \pi)$, and fix coordinates $(x,y) \in \R^m\times\R^n$ such that $p_{H} = (0,0)$.
Set $E:= \bE \big(T, \bC_{32 r} (p, \pi)\big)$. By Proposition \ref{p:tilting opt}, $\supp(T)\cap\bC_{32 r} (p, \pi) \subset \B_L$
. Thus, by \eqref{e:ex+ht_whitney} and Proposition \ref{p:tilting opt}(v) we have
\[
E \leq  C \bmo \ell (L)^{2-2\delta_2}\quad\text{and}\quad
\bh (T, \bC_{32r} (p, \pi)) \leq C \bmax_0^{\sfrac{1}{2m}}\ell(L)^{1+\beta_2}\, .
\]
Recall that, by \cite[Theorem 1.5]{DSS2} we have
\begin{gather}
|Df| \leq C E^{\beta_0} + C \bmo^{\sfrac{1}{2}} r
\leq C \bmo^{\beta_0}\,
r^{\beta_0(2-2\delta_2)}\label{e:lip f}\\
|f| \leq C \bh (T, \bC_{32 r} (p, \pi)) + r\,E^{\sfrac{1}{2}} 
 \leq C \bmo^{\sfrac{1}{2m}}\, r^{1+\beta_2},\label{e:Linf}\\
\int_B |Df|^2 \leq C\, r^m E \leq C \bmo \, r^{m+2-2\delta_2},\label{e:Dir f}
\end{gather}
and 
\begin{align}
&|B\setminus K| \leq C E^{\beta_0} (E + r^2 \bOmega^2)r^m \leq C \bmo^{1+\beta_0}\,r^{m +(1+\beta_0) (2 - 2\delta_1)}\, ,\label{e:aggiuntiva_1}\\
&\left| \|T\| (\bC_{8r} (p, \pi)) - |B| - \frac{1}{2} \int_B |Df|^2\right| \leq C E^{\beta_0} (E + r^2 \bOmega^2) r^m\leq C \bmo^{1+\beta_0}\, r^{m +(1+\beta_0) (2 - 2\delta_2)}&\, , \label{e:aggiuntiva_2}
\end{align}
where $K\subset B$ is the set
\begin{equation}\label{e:recall_K}
B\setminus K = \p_\pi \left(\left(\supp (T)\Delta \supp(\bG_f)\right)\cap \bC_{8r_L} (p_L, \pi)\right)\, .
\end{equation}

Consider the vector field $\chi(x,y) := (0, \zeta(x))$
for some $\zeta$ as in the statement. Using Lemma \ref{l:quasi_minimalita}, we infer
\[
 \delta \bG_f (\chi) = \delta T (\chi) + \textup{Err}_0=T(d\omega\ser \chi)  +\textup{Err}_0= \bG_f(d\omega\ser \chi)+\textup{Err}_0 +\textup{Err}_1
\]
with
\begin{align}
|\textup{Err}_0 + \textup{Err}_1| &= |\delta T(\chi) - \delta \bG_f(\chi)| 
+ \big\vert T(d\omega \ser \chi) - \bG_f(d\omega \ser \chi)\big\vert\notag\\
&\leq C\, \big(\|D\zeta\|_{0} + \|d\omega\ser \chi\|_{0} \big)
\| T - \bG_f\|(\bC_{8r}(p,\pi))\notag\\
& \leq C\,\big(\|D\zeta\|_{0} + \|\zeta\|_{0} \big)
\,E^{\beta_0}\, (E + r^2\,\bmo)\, r^m \notag\\
&\leq C\,\|D\zeta\|_{0}\,\bmo^{1+\beta_0}
\,r^{m + (2 - 2\delta_2) (1+ \beta_0)}.
\end{align}
From \cite[Theorem~4.1]{DS2}
\[
\delta \bG_f(\chi) = Q \int D(\etaa\circ f) \colon D\zeta + \textup{Err}_2
\]
with
\begin{align*}
|\textup{Err}_2| &\leq C\,\int |D\zeta|\, |Df|^3\leq C\,\|D\zeta\|_{0} E^{1+\beta_0}\,r^m\\
&\leq  C\, \|D\zeta\|_{0}\,\bmo^{1+\beta_0}\,
r^{m + (2-2\delta_2)(1+\beta_0)}.
\end{align*}
Next we proceed to expand $\bG_f(dw\ser \chi)$. To this aim we write
\begin{align}
d\omega(x,y) & = \sum_{l=1}^n a_l (x,y)\, dy^l \wedge dx^1\wedge\dots\wedge dx^m \\
&+\sum_{j=1}^m\sum_{l<k} b_{lk,j}(x,y)\, dy^l \wedge dy^k \wedge dx^1\wedge\dots\wedge \widehat{dx^j}\wedge\dots\wedge dx^m\notag\\
&
+\text{terms with the wedge product of at least three $dy^l$}
\end{align}
and get
\begin{align}
d\omega \ser \chi & = \underbrace{\sum_{l=1}^n a_l\,\zeta^l\,
dx^1\wedge\dots\wedge dx^m}_{\omega^{(1)}}
+ \underbrace{\sum_{j=1}^m\sum_{l<k} b_{lk,j}\, \zeta^l dy^k \wedge dx^1\wedge\dots\wedge \widehat{dx^j}\wedge\dots\wedge dx^m}_{\omega^{(2)}}
+\omega^{(3)}.
\end{align}
We consider separately $\bG_f(\omega^{(1)}), \bG_f(\omega^{(2)}), \bG_f(\omega^{(3)})$.
We start with the latter. Using the decay of the energy in \eqref{e:Dir f}, we have for $\eps_2$ small enough,
\begin{align}
\bG_f(\omega^{(3)}) & \leq C\,\|d\omega\|_0\,\|\zeta\|_0\int_{B}|Df|^2\leq
C\,\bmo^2\, r^{m+3-2\delta_2} \|D\zeta\|_0.
\end{align}
Next
\begin{align}
\bG_f(\omega^{(2)}) & = \sum_{j=1}^m\sum_{l<k} \sum_{i=1}^{Q}
(-1)^{j-1}\int \zeta^l(x) \,
b_{lk,j}(x,f_i(x))\, \frac{\de f_i^k}{\de x^j}  \notag\\
& = Q\, \sum_{j=1}^m\sum_{l<k} 
\int \zeta^l(x) \,
b_{lk,j}(0,0)\, \frac{\de (\etaa\circ f)^k}{\de x^j} 
+ \textup{Err}_3,\notag\\
& = \int \bL_1\,D(\etaa \circ f) \cdot \zeta + \textup{Err}_3
\end{align}
with
\begin{align}
|\textup{Err}_3| & \leq C \|\zeta\|_0\,\|D(d\omega)\|_0
\int_B \left(r\,|Df| + |f|\,|Df| \right)\, dx\notag \\
&\leq C \|D\zeta\|_0\,\bmo\, 
\big(r + \osc(f) + \bh(T, \bC_{8r}(0,\pi))  \big)\, r^{m+1}\, E^{\beta_0}\notag \\
&\leq C \|D\zeta\|_0\,\bmo^{1+\beta_0}\, r^{m+2+(2-2\delta_1)\beta_0}
\end{align}
and $\bL_1 : \R^{n\times m} \to \R^n$ given by
\[
\bL_1 A\cdot e_l := Q\,\sum_{j=1}^m\sum_{k=1}^n b_{lk,j}(0,0)\, A_{kj}  \quad \forall\;A=(A_{kj})_{k=1,\ldots,n}^{j=1,\dots,m}\in \R^{n\times m}.
\]
Finally
\begin{align}
\bG_f(\omega^{(1)}) & = \sum_{l} \sum_{i=1}^{Q}
\int \zeta^l(x) \,
a_{l}(x,f_i(x))\,dx\notag\\
& = Q\sum_{l} \int \zeta^l(x) \,
\left(a_{l}(0,0) + D_x a_l(0,0) \cdot x +
\,D_y a_l(0,0) \cdot (\etaa\circ f) \right) dx + \textup{Err}_4,\notag\\
& = \int \big(\bL_2\,(\etaa\circ f)  + \bL_3\, x + \bL_4 \big) \cdot \zeta
+ \textup{Err}_4
\end{align}
where $\bL_2:\R^n \to \R^n$, $\bL_3:\R^m \to \R^n$
 $\bL_4\in \R^n$ are given by
\begin{gather}
\bL_2 \,v \cdot e_l := \sum_{k=1}^n \frac{\de a_l}{\de y^k}(0,0) \, v^k
\quad \forall\;v\in\R^n,\;\forall\;l=1,\ldots,n\\
\bL_3\, p \cdot e_l := \sum_{j=1}^2 \frac{\de a_l}{\de x^j}(0,0) \, p^j
\quad \forall\;w\in\R^n,\;\forall\;l=1,\ldots,n\\
\bL_4 \cdot e_l := a_l(0,0)\quad\forall\;l=1,\ldots,n
\end{gather}
and arguing as above
\begin{align}
|\textup{Err}_4| \leq C \|\zeta\|_0\,[D (d\omega)]_{\eps_0}
\int_B \left(r^{1+\eps_0} + |f|^{1+\eps_0} \right) \,dx
\leq C \|D\zeta\|_0\,\bmo\, r^{m+2+\eps_0}.
\end{align}
By the choice of our parameters the conclusion immediately follows. In particular, since $\mathscr{F}_{H} (p)=\bL_3 p+ \bL_4$ and $\omega \in C^{2,\eps_0}$, we have 
\[
r_L\,\|\bL_3\|_0+\|\bL_4\|_0\leq C\,r_L\, \|D(d\omega)\|_0+\|d\omega\|_0\leq c\,\bmo^{\sfrac12}\, r_L^{1+\eps_0}\,, 
\]
so that the second conclusion of (i) holds.
\end{proof}

\subsection{Existence of the tilted interpolating functions and their main estimates}

We start generalizing the definition of the tilted interpolating functions $h_L$. More precisely we consider

\begin{definition}\label{d:tilted_HL}
Let $H$ and $L$ be as in Proposition \ref{p:gira_e_rigira}, assume that the conclusions of Proposition \ref{p:pde} applies and let $\mathscr{L}_H$ and $\mathscr{F}_H$ be the corresponding operator and map as given by Proposition \ref{p:pde} in combination with \eqref{e:elliptic}, \eqref{e:ellittico} and \eqref{e:weak_form}. Let $f_{HL}$ be as in Definition \ref{d:f_HL} and fix coordinates $(x,y)\in \pi_H \times \pi_H^\perp$ as in the proof of Proposition \ref{p:pde}. We then let $h_{HL}$ be the solution of
\begin{equation}\label{e:ell_def}
\left\{
\begin{array}{l}
\mathscr{L}_H h_{HL} = \mathscr{F}_H\\ \\
\left.h_{HL}\right|_{\partial B_{8r_L} (p_{HL}, \pi_H)} = \etaa \circ f_{HL}\, .
\end{array}\right.
\end{equation}
\end{definition}

In order to show that the maps $h_{HL}$ are well defined, we recall the following lemma, proved in \cite{DSS3}.

\begin{lemma}[{\cite[Lemma 6.6]{DSS3}}]\label{l:exist}
Under the assumptions of Definition \ref{d:tilted_HL}, if $\eps_2$ is sufficiently small, then the elliptic system
\begin{equation}\label{e:esiste}
\left\{
\begin{array}{l}
\mathscr{L}_H v = F\ \\ \\
\left.v\right|_{\partial B_{8r_L} (p_{HL}, \pi_H)} = g\, .
\end{array}\right.
\end{equation}
has a unique solution for every $F\in W^{-1,2}$ and every $g\in W^{1,2} (B_{8 r_L} (p_{HL}, \pi_H))$. Observe moreover that $\|Dv\|_{L^2} \leq C_0 r_L (\|F\|_{L^2} + \bmo^{\sfrac{1}{2}} \|g\|_{L^2}) + C_0 \|Dg\|_{L^2}$ whenever $F\in L^2$. 
\end{lemma}

Observe that $h_{HH} = h_H$. We next record three fundamental estimates, which regard, respectively, the $L^\infty$ norms of derivatives of solutions of $\mathscr{L}_H (v) = F$, the $L^\infty$ norm of ${h}_{HL}- \etaa\circ {f}_{HL}$ and the $L^1$ norm of ${h}_{HL}-\etaa \circ f_{HL}$. 

\begin{proposition}\label{p:stime}
Let $H$ and $L$ be as in Proposition \ref{p:pde} and assume the conclusions in there apply. Then the following estimates
hold for a constant $C = C(m_0, N_0, C_e, C_h)$ for $\hat{B} := B_{8r_L} (p_{HL}, \pi_H)$ and $\tilde{B}:=B_{6r_L} (p_{HL}, \pi_H)$:
\begin{align}
&\|{h}_{HL} - \etaa\circ {f}_{HL}\|_{L^1 (\hat{B})} \leq C \bmo \,\ell (L)^{m+3 + \beta_2}\label{e:stima-L1}\\
&\|{h}_{HL} - \etaa \circ {f}_{HL}\|_{L^\infty (\tilde{B})}\leq C \bmo^{\sfrac12} \, \ell (L)^{3 + \beta_2}\, .\label{e:stima-Linfty}
\end{align}
Moreover, if $\mathscr{L}_H$ is the operator of Proposition \ref{p:pde}, $r$ a positive number no larger than $1$ and $v$ a solution of $\mathscr{L}_H (v) =F $ in $B_{8r} (q, \pi_H)$, then
\begin{equation}\label{e:L1-Linfty}
\|v\|_{L^\infty (B_{6r} (q, \pi_H))} \leq \frac{C_0}{r^m} \|v\|_{L^1 (B_{8r} (q, \pi_H))} + C r^2 \|F\|_{L^\infty (B_{8r} (q, \pi_H))}\, 
\end{equation}
and, for $l\in \mathbb N$
\begin{equation}\label{e:higher}
\|D^l v\|_{L^\infty (B_{6r} (q, \pi_H))} \leq \frac{C_0}{r^{m+l}} \|v\|_{L^1 (B_{8r} (q, \pi_H))} + C r^2 \sum_{j=0}^l r^{j-l} \|D^j F\|_{L^\infty (B_{8r} (q, \pi_H))},
\end{equation}
where the latter constants depend also upon $l$. 
\end{proposition}

\begin{proof}
{\bf Proof of \eqref{e:L1-Linfty}}. The estimate will be proved for a linear constant coefficient operator of the form 
$\mathscr{L} = \Delta + {\bf L}_1 \cdot D + {\bf L}_2$ when ${\bf L}_1$ and ${\bf L}_2$  are sufficiently small. We can then assume $\pi_H= \R^m$ and $q=0$. Besides, if we define $u (x) := v (rx)$ we see that $u$ just satisfies $\Delta u + r \mathbf{L}_1 \cdot D u + r^2 \mathbf{L}_2 \cdot u = 0$ and thus, without loss of generality, we can assume $r=1$. We thus set $B= B_8 (0)\subset \R^m$.

We recall the following interpolation estimate on the ball of radius $1$, see \cite[Theorem 1]{Nir}. For $0\leq j< k$ and $\frac{j}{k}\leq a\leq 1$ we have, for a constant $C_0=C_0 (k,j,q,r)$,
 \begin{equation}\label{e:Nir}
 \|D^ju\|_{L^p(B_1)}\leq C \|D^ku\|^a_{L^s(B_1)}\;\|u\|^{1-a}_{L^q(B_1)}+C\;\|u\|_{L^q(B_1)}\, ,
 \end{equation}
 where  
 \[
  \textstyle \frac{1}{p}=\frac{j}{m}+a\big(\frac{1}{s}-\frac{k}{m}\big)+(1-a)\frac{1}{q}\, .
 \]
We apply the estimate \eqref{e:Nir} for $j=1$, $k=2$, $q=1$ and $p=s>m$ to be fixed later. Notice in particular that, for $p>m$, 
\[
\frac{1}{2}\leq a=\frac{mp+p-m}{mp+2p-m}<1\,.
\] 
Using Young's inequality and a simple scaling argument, we achieve the inequality
\begin{equation}\label{e:una}
\|Du\|_{L^p (B_\rho (x))} \leq C_0 \rho \|D^2 u\|_{L^p (B_\rho (x))} + C_0 \rho^{-\sfrac mp -m-1} \|u\|_{L^1 (B_\rho (x))}\, .
\end{equation}
Moreover, by Sobolev embedding:
\begin{equation}\label{e:due}
\|u\|_{L^p (B_\rho (x))} \leq C_0 \rho \|Du\|_{L^p (B_\rho (x))} + C_0 \rho^{\sfrac mp -m} \|u\|_{L^1 (B_\rho (x))}\, .
\end{equation}
Next, recall the standard C\'alderon-Zygmund estimates for second order derivatives of solutions of the Laplace equations: if $B_{2\rho} (x)\subset B$, then
\begin{equation}\label{e:tre}
\|D^2 u\|_{L^p (B_\rho (x))} \leq C_0 \|\Delta u\|_{L^p (B_{2\rho} (x))} + C_0 \rho^{\sfrac mp -m-2} \|u\|_{L^1 (B_{2\rho} (x))}\, .
\end{equation}
Now, recall that $\Delta u = - {\bf L}_1 \cdot Du - {\bf L}_2 \cdot u + F$. Using the fact that $|{\bf L}_1| + |{\bf L}_2| \leq C_0 \bmo^{\sfrac{1}{2}}$, we can combine all the inequalities above to conclude 
\begin{equation}\label{e:quattro}
\rho^{m+2-\sfrac mp} \|D^2 u\|_{L^p (B_\rho (x))} \leq C_0 \rho^{m+2-\sfrac mp} \bmo \|D^2 u\|_{L^p (B_{2\rho} (x))} + C_0\|u\|_{L^1 (B_8)} + C_0 \|F\|_{L^\infty (B_8)}\, .
\end{equation}
Define next
\begin{equation}
S := \sup \{ \rho^{m+2-\sfrac mp} \|D^2 u \|_{L^p (B_\rho (x))} : B_{2\rho (x)} \subset B_8\}
\end{equation}
and let $\varrho$ and $\xi$ be such that $B_{2 \varrho} (\xi) \subset B_8$ and 
\begin{equation}\label{e:ottimalita'}
\varrho^{m+2-\sfrac mp} \|D^2 u\|_{L^p (B_\varrho (\xi))} \geq \frac{S}{2}\, .
\end{equation}
We can cover $B_{\varrho} (\xi)$ with $\bar N_0$ balls $B_{\varrho/2} (x_i)$ with $x_i \in B_{\varrho} (\xi)$, where $\bar N_0$ is only a geometric constant. We then can apply \eqref{e:quattro} to conclude that
\[
\frac{S}{2} \leq C_0 \bar  N_0 \bmo^{\sfrac{1}{2}} S + C_0 \bar N_0 \|u\|_{L^1 (B_8)} + C_0 \bar N_0 \|F\|_{L^\infty (B_8)}\, .
\]
Therefore, when $\bmo^{\sfrac{1}{2}}$ is smaller than a geometric constant we conclude $S\leq C_0 \|u\|_{L^1 (B_8)}+C_0 \|F\|_{L^\infty (B_8)}$. By definition of $S$, we have reached the estimate
\[
\rho^{m+2-\sfrac mp} \|D^2 u\|_{L^p (B_\rho (x))} \leq C_0 \|u\|_{L^1 (B_8)} + C_0 \|F\|_{L^\infty (B_8)} \qquad \mbox{whenever $B_{2\rho} (x) \subset B_8$.}
\]
Of course, with a simple covering argument, this implies
\begin{equation}\label{e:cinque}
\|D^2 u\|_{L^p (B_6)}\leq C_0 \|u\|_{L^1 (B_8)}  + C_0 \|F\|_{L^\infty (B_8)}\, .
\end{equation}
Next, again using the interpolation inequality \eqref{e:una} we get 
\[
\|Du\|_{L^p (B_6)} \leq C_0 \|u\|_{L^1 (B_8)} + C_0 \|F\|_{L^\infty (B_8)}\, .
\] 
So, for every $p\geq m$, 
\[
\|Du\|_{L^p (B_6)} \leq C_0\|D u\|_{W^{1,p} (B_6)} \leq C_0\|u\|_{L^1 (B_8)} + C_0 \|F\|_{L^\infty (B_8)}\, .
\]
Choosing $p$ big enough depending on $m$ so that $W^{1,p}(B_6)\hookrightarrow C^0(B_6)$, we finally achieve
\[
\|u\|_{L^\infty (B_6)}\leq C_0 \|u\|_{W^{1,p} (B_6)} \leq C_0 \|u\|_{L^1 (B_8)} + C_0 \|F\|_{L^\infty (B_8)}\, .
\]

\medskip

{\bf Proof of \eqref{e:higher}.} As in the previous step, we can, without loss of generality, assume $r=1$. Note that a byproduct of the argument given above is also the estimate
\[
\|Du\|_{L^1 (B_6)} \leq C_0 \|u\|_{L^1 (B_{8})} + C_0 \|F\|_{L^\infty (B_8)}\, .
\]
In fact, by a simple covering and scaling argument one can easily see that 
\[
\|Du\|_{L^1 (B_\tau)}\leq C_0 (\tau) \|u\|_{L^1 (B_8)} + C_0 (\tau) \|F\|_{L^\infty (B_8)} \qquad \mbox{for every $\tau<8$.}
\] 
We can then differentiate the equation and use the proof of the previous paragraph to show 
\[
\|Du\|_{L^\infty (B_\sigma)} \leq C_0 (\sigma, \tau) \|Du\|_{L^1 (B_\tau)} + C_0 (\sigma, \tau) \|DF\|_{L^\infty (B_\tau)}\, .
\] 
Again, arguing as above, a byproduct of the proof is also the estimate
\[
\|D^2 u\|_{L^1 (B_\sigma)} \leq C_0 (\sigma, \tau) \|Du\|_{L^1 (B_\tau)} + C_0 (\sigma, \tau) \|DF\|_{L^\infty (B_\tau)}\, .
\] 
This can be applied inductively to get estimates for all higher derivatives.

\medskip

{\bf Proof of \eqref{e:stima-L1}.} Let $B := B_{8 r_L} (p_{HL}, \pi_H)$. We use the coordinates introduced in the proof of Proposition \ref{p:pde}. We set
$w:= {h}_{HL} - \etaa\circ  f_{HL}$ and observe that
\[
\left\{
\begin{array}{l}
\mathscr{L} w = \mathscr{F}_H - \mathscr{L}_H (\etaa \circ {f}_{HL})\\ \\
w|_{\partial B} =0
\end{array}
\right.
\]
Next, for $1<p<\infty$, we define the continuous (by Calderon-Zygmund theory) linear
operator $T: L^p(B) \to W^{1,p}_0(B) \cap W^{2,p} (B)$ by $T(g) = \psi$ where
\[
\begin{cases}
-\Delta \psi = g & \text{in } B\\ \\
\psi = 0 & \text{on } \partial B.
\end{cases}
\]
For $p$ big enough depending on $m$, we can apply the Sobolev embedding $W^{1,p}(B)\hookrightarrow C^0(B)$ to the derivative of $\zeta \in W^{2,p}\cap W^{1,p}_0$ to conclude that
\[
\|D \zeta - {\textstyle{\mint}} D \zeta\|_0 \leq C_0 r_L^{1-\sfrac{m}{p}} \|D^2\zeta\|_{L^p}\, .
\]
On the other hand, by interpolation and Poincar\'e we conclude
\[
\|D\zeta\|_{L^p} \leq \|\zeta\|_{L^p}^{\sfrac{1}{2}}\|D^2 \zeta\|_{L^p}^{\sfrac{1}{2}} 
\leq \frac{\varepsilon}{2r_L} \|\zeta\|_{L^p} + \frac{r_L}{2\varepsilon}\|D^2 \zeta\|_{L^p}
\leq C_0 \varepsilon \|D\zeta\|_{L^p} +  \frac{r_L}{2\varepsilon}\|D^2 \zeta\|_{L^p}\, ,
\]
for every positive $\varepsilon$. Choosing the latter accordingly we achieve $\|D\zeta\|_{L^p} \leq C_0 r_L \|D^2 \zeta\|_{L^p}$ and
thus $\|D\zeta\|_0 \leq C_0 r_L^{1-\sfrac mp} \|D^2 \zeta\|_{L^p}$.

We now use these bounds in \eqref{e:pde2} to get 
\begin{align*}
\left|\int_B (Dw : D\zeta - \bL_1\,Dw\cdot \zeta - \bL_2\, w\cdot \zeta)\right|
\leq 
C \bmo \,r_L^{m+2+\beta_2}\,r_L^{1-\frac{m}{p}}\,\|D^2\zeta\|_{L^p}.
\end{align*}
Then, we can estimate the $L^{p'}$-norm of $w$ as follows:
\begin{align*}
\|w\|_{L^{p'}(B)} & = \sup_{\|h\|_{L^p(B)}=1} \int_B w \,h
= - \sup_{\|h\|_{L^p(B)}=1} \int_B w \,\Delta\,T(h)\\
&\leq \sup_{\|h\|_{L^p(B)}=1} \int_B  D \,w \cdot D T(h)\\
&\leq C \bmo \,
r_L^{m+3+\beta_2-\sfrac{m}{p}}\,\sup_{\|h\|_{L^p(B)}=1} \|D^2T(h)\|_{L^p}\\
&+ \sup_{\|h\|_{L^p(B)}=1} \int_B 
(- \bL_1\,Dw\cdot T(h) - \bL_2\, w\cdot T(h))\, .
\end{align*}
Recalling the Calderon-Zygmund estimates we have 
\begin{align*}
&\|D^2 T (h)\|_{L^p} \leq C_0 \|h\|_{L^p}\\ 
&\|D T (h)\|_{L^p} \leq C_0 r_L \|h\|_{L^p}\\
&\|T (h)\|_{L^p} \leq C_0 r_L^2 \|h\|_{L^p}\, . 
\end{align*}
Integrating by parts we then achieve
\begin{align*}
\|w\|_{L^{p'}(B)} &\leq C \bmo \,
r_L^{m+3+\beta_2-\sfrac{m}{p}}
+\sup_{\|h\|_{L^p(B)}=1} \int_B w \cdot (\bL_1 DT(h) - \bL_2 T(h))\\
&\leq C \bmo \,
r_L^{m+3+\beta_2-\sfrac{m}{p}} + C \bmo^{\sfrac{1}{2}} \|w\|_{L^{p'} (B)}\, .
\end{align*}
Therefore, if $\bmo^{\sfrac{1}{2}}$ is sufficiently small, that is $\eps_2$ is sufficiently small, we deduce that
\begin{gather*}
\|w\|_{L^1}\leq C\,r_L^{\sfrac{m}{p}}
\|w\|_{L^{p'}(B)} \leq C\,\bmo\,r_L^{m+3+\beta_2}.\label{e:C-Z}
\end{gather*}

\medskip

{\bf Proof of \eqref{e:stima-Linfty}.} The estimate follows easily when $\eps_2$ is sufficiently small, from \eqref{e:stima-L1} and \eqref{e:L1-Linfty}, recalling
that $\|\mathscr{F}_H(p)\|_0 \leq C \bmo^{\sfrac{1}{2}} r_L^{1+\eps_0}$. 
\end{proof}

We are now ready to state the key construction estimates, from which Theorem \ref{t:cm} follows easily as in \cite[Section 4.4]{DS4}.

\begin{proposition}[Construction estimates {\cite[Proposition 4.4]{DS4}}]\label{p:stime_chiave}
 Assume the conclusions of Propositions \ref{p:tilting opt} and \ref{p:gira_e_rigira} apply 
and set $\kappa = \min \{\beta_2/4, \eps_0/2\}$. Then,
the following holds for any pair of cubes $H, L\in \sP^j$ (cf. Definition \ref{d:glued}), where
$C = C (\beta_2, \delta_2, M_0, N_0, C_e, C_h)$:
\begin{itemize}
\item[(i)] the conclusions of Lemma \ref{l:tecnico3} hold;
\item[(ii)] $\|g_H\|_{C^0 (B)}\leq C\, \bmo^{\sfrac{1}{2m}}$ and
$\|Dg_H\|_{C^{2, \kappa} (B)} \leq C \bmo^{\sfrac{1}{2}}$, for $B = B_{4r_H} (x_H, \pi_0)$;
\item[(iii)] if $H\cap L\neq \emptyset$,
then $\|g_H-g_L\|_{C^i (B_{r_L} (x_L))} \leq C \bmo^{\sfrac{1}{2}} \ell (H)^{3+\kappa-i}$ 
for every $i\in \{0, \ldots, 3\}$;
\item[(iv)] $|D^3 g_H (x_H) - D^3 g_L (x_L)| \leq C \bmo^{\sfrac{1}{2}} |x_H-x_L|^\kappa$;
\item[(v)] $\|g_H-y_H\|_{C^0} \leq C \bmo^{\sfrac{1}{2m}} \ell (H)^{1+\beta_2}$ and 
$|\pi_H - T_{(x, g_H (x))} \bG_{g_H}| \leq C \bmo^{\sfrac{1}{2}} \ell (H)^{1-\delta_2}$
$\forall x\in H$;
\item[(vi)] if $L'$ is the cube concentric to $L\in \sW^j$ with $\ell (L')=\frac{9}{8} \ell (L)$, 
then
\[
\|\varphi_i - g_L\|_{L^1 (L')} \leq C\, \bmo\, \ell (L)^{m+3+\beta_2/3} \quad \text{for all }\; i\geq j.
\]
\end{itemize}
\end{proposition}

\begin{proof}
We start by fixing $H, L, J$ so that $H\in \sS\cup \sW$, $L$ is an ancestor of $H$ (possibly $H$ itself) and $J$ is  the father of $L$. We denote by $B'$ the ball $B_{8r_J} (p_{HJ}, \pi_H)$, by $B$ the ball $B_{8r_L} (p_{HL}, \pi_H)$, by $\bC'$ the cylinder $\bC_{8r_J} (p_J, \pi_H)$ and by $\bC$ the cylinder $\bC_{8r_L} (p_L, \pi_H)$. Observe that $B\subset B'$ (this just requires $M_0$ sufficiently large, given the estimate $|p_J-p_L|\leq 2\sqrt{m}\ell (J)$) and thus $\bC\subset \bC'$. Next, set $E:= \bE (T, \bC_{32 r_L} (p_L, \pi_H))$ and $E':= \bE (T, \bC_{32r_J} (p_J, \pi_H))$ and recalling Proposition \ref{p:tilting opt} we record
\begin{gather}
E \leq C \bmo \ell (L)^{2-2\delta_2} \leq C \bmo  \ell (J)^{2-2\delta_2}\quad\text{and}\quad
E' \leq  C \bmo \ell (J)^{2-2\delta_2}\\
\bh (T, \bC) \leq C \bmax_0^{\sfrac{1}{2m}}\ell(L)^{1+\beta_2}\leq 
C \bmax_0^{\sfrac{1}{2m}}\ell(J)^{1+\beta_2}\quad\text{and}\quad
\bh (T, \bC') \leq C \bmax_0^{\sfrac{1}{2m}}\ell(J)^{1+\beta_2}\, .
\end{gather}
Next let $\bar K$ be the projection of $\gr (f_{HL})\cap \gr (f_{HJ})$ onto $p_{HL} + \pi_H$. We can estimate
\[
|B\setminus \bar K|\leq \cH^m (\gr (f_{HL})\setminus \supp (T)) + \cH^m (\gr (f_{HJ})\setminus \supp (T))
\]
and recalling the estimates of \cite[Theorem 1.5]{DSS2} we achieve
\[
|B\setminus \bar K| \leq C_0 r_J^m (E^{\beta_0} (E + C_0 \bmo r_J^2) + E'^{\beta_0} (E' + C_0 \bmo r_J^2))
\leq C \bmo ^{1+\beta_0}  \ell (J)^{m+2}\, .
\]
In particular $\bar K$ is certainly nonempty, provided $\eps_2$ is small enough. Using the estimates of \cite[Theorem 1.5]{DSS2} on the oscillation of $f_{HL}$ and $f_{HJ}$ we conclude that
\[
\|\etaa\circ f_{HL} - \etaa\circ f_{HJ}\|_{L^\infty (B)} \leq C \bmo^{\sfrac{1}{2m}}\, \ell (J)^{1+\beta_2}\, .
\]
Set therefore $\zeta := \etaa \circ f_{HL} - \etaa\circ  f_{HJ}$ and conclude that
\[
\|\zeta\|_{L^1 (B)}
 \leq \|\etaa\circ f_{HL} - \etaa\circ f_{HJ}\|_{L^\infty (B)} \, |B\setminus \bar K|
\leq C \bmo^{1+\beta_0+\sfrac{1}{2m}} \ell(J)^{m+3+\beta_2}\, .
\]
If we define $\xi:= h_{HL}- h_{HJ}$ we can use \eqref{e:stima-L1} of Proposition \ref{p:stime} and the triangular inequality to infer
\[
\|\xi\|_{L^1 (B)} \leq C \bmo  \ell(J)^{m+3+\beta_2}\, .
\]
In turn, again by Proposition \ref{p:stime}, this time using the fact that $\mathscr{L}_H \xi =0$ and \eqref{e:higher}, we infer
\begin{align}\label{e:induttive_1}
\|D^l ({h}_{HL} - {h}_{HJ})\|_{C^0 (\hat B)} &\leq C\bmo  \ell (J)^{3+\beta_2 - l} 
\leq C\bmo  \ell (J)^{3+2\kappa-l} \qquad \mbox{for $l=0,1,2,3,4$,}
\end{align}
where $\hat B = B_{6r_L} (p_{HL}, \pi_H)$.
Interpolating we get easily also
\begin{equation}\label{e:induttive_2}
[D^3 ({h}_{HL} - {h}_{HJ})]_{0,\kappa, \hat B} \leq C\bmo \ell (J)^{\kappa}\, .
\end{equation}
Fix now a chain of cubes $H = H_j \subset H_{j-1} \subset \ldots \subset H_{N_0} =: L$, where each $H_{j-1}$ is the father of $H_j$.
Summing the estimates above and using the fact that $\ell (H_j) = 2^{-j}$, we infer
\begin{align}
&\|D^l (h_{HL} - h_H)\|_{C^0 (\tilde{B})} \leq C \,\bmo\,\ell(L)^{3+2\kappa-l} \qquad \mbox{for $l=0,1,2,3$}\label{e:piena}\\
&[D^3 (h_{HL} - h_H)]_{0, \kappa, \tilde{B}} \leq C\, \bmo\,\ell(L)^{\kappa} \, ,\label{e:holder}
\end{align}
where $\tilde B = B_{6r_H} (p_H, \pi_H)$. 
Observe that, assuming that we have fixed coordinates so that $p_H = (0,0)$ we also know, arguing as in the proof of Proposition \ref{p:pde}, that, if we set $\bar{B} := B_{8r_L} (p_{HL}, \pi_H)$, then
\[
\|\etaa\circ {f}_{HL}\|_{L^\infty (\bar{B})} \leq C \bmo^{\sfrac{1}{2m}} \ell(L)^{1+\beta_2}\, .
\]
In particular, applying \eqref{e:stima-Linfty} of Proposition \ref{p:stime}, we conclude
\[
\|h_{HL}\|_{C^0(\hat B)} \leq  C \bmo^{\sfrac{1}{2m}} \ell(L)^{1+\beta_2}\,.
\]

We next estimate the derivatives of $h_{HL}$. Let $E:= \bE (T, \bC_{32 r_L} (p_L, \pi_H))$
and recall the discussion above and the estimates of \cite[Theorem 1.5]{DSS2} to conclude that
\begin{equation}
\int_{\bar{B}} |Df_{HL}|^2 \leq C_0 r_L^m E \leq C \bmo \ell (L)^{m+2-2\delta_2}\,. 
\end{equation}
We thus conclude that $\|D\etaa\circ {f}_{HL}\|_{L^2 (\bar{B})} \leq C \bmo^{\sfrac{1}{2}} \,\ell(L)^{\sfrac{m}{2}+1-\delta_2}$. 
We can now use the Lemma \ref{l:exist} to estimate $\|D {h}_{HL}\|_{L^2}\leq C \bmo^{\sfrac{1}{2}} \,\ell(L)^{\sfrac{m}{2}+1-\delta_2}$ and thus
$\|D {h}_{HL}\|_{L^1} \leq C \bmo^{\sfrac{1}{2}} \,\ell(L)^{m+1-\delta_2}$.  If we differentiate the equation defining ${h}_{HL}$ we then find
\[
\mathscr{L}_H \partial_j {h}_{HL}^i = ({\mathbf L}_3)^i_{j}
\]
and we can thus apply \eqref{e:L1-Linfty} of Proposition \ref{p:stime}, with $v=Dh_{HL}$, to conclude that 
\begin{equation}\label{e:capotribu}
\|D^l {h}_{HL}\|_{L^\infty (B_{6r_L})} \leq C \bmo^{\sfrac{1}{2}} \,\ell(L)^{2-\delta_2-l}\leq C\,\bmo^{\sfrac12} \qquad \mbox{for $l =1,2,3,4$,}
\end{equation}
where we used that, for the starting cubes $L=H_{N_0}$, $ C (M_0) \leq \ell(L)$.

Using 
\eqref{e:holder} we then get $\|D{h}_H\|_{C^{2, \kappa} (B)} \leq \|D h_H - D  h_{HL}\|_{C^{2, \kappa} (B)} + \|Dh_{HL}\|_{C^{2, \kappa} (B^{N_0})}\leq C \bmo^{\sfrac{1}{2}}$, where we used the fact that for the biggest cube $L$ the bound $\ell(L)>C(N_0,M_0)>0$ holds. This, together with $|\pi_H - \pi_0| \leq C \bmo^{\sfrac{1}{2}}$ in Proposition \ref{p:tilting opt}, allows us to apply \cite[Lemma B.1]{DS4} and achieve the existence of the function $g_H$ and the bound $\|D g_H\|_{C^{2, \kappa} (B)} \leq C \bmo^{\sfrac{1}{2}}$. With a similar argument using the bound $\|h_{HL}\|_{C^0 (B^{N_0})}\leq C \bmo^{\sfrac{1}{2m}}$ proved above, we achieve also the bound
$\|h_H\|_{C^0 (B)} \leq  C \bmo^{\sfrac{1}{2m}}$. Hence again by \cite[Lemma B.1]{DS4} $\|g_H\|_{C^0 (B)} \leq C \bmo^{\sfrac{1}{2m}}$. This shows obviously Proposition \ref{p:stime_chiave}(i)-(ii).

\medskip

Next, recalling the height estimate (and taking into account that $p_H = 0 \in \supp (T)$) we have
\[
\|\etaa \circ f_H\|_{L^\infty(B_{8r_H} (p_H, \pi_H))} \leq C \bmo^{\sfrac{1}{2m}} \,  \ell (H)^{1+\beta_2}\,
\] 
and thus
\[
\|\etaa\circ f_H\|_{L^1 (B_{8r_H} (p_H, \pi_H))} \leq C \bmo^{\sfrac{1}{2m}}\, \ell (H)^{m+1+\beta_2}
\]
Using \eqref{e:stima-L1} we conclude 
\[
\|h_H\|_{L^1 (B_{8r_H} (p_H, \pi_H))} \leq C \bmo^{\sfrac{1}{2m}}\ell (H)^{m+1+\beta_2}\, 
\]
and with the help of \eqref{e:L1-Linfty} we achieve
\begin{equation}\label{e:h-Linfty_2}
\|h_H\|_{L^\infty (B_{6r_H} (p_H, \pi_H))} \leq C  \bmo^{\sfrac{1}{2m}}  \ell (H)^{1+\beta_2}\, .
\end{equation}
Combining \eqref{e:h-Linfty_2} with \cite[Lemma B.1]{DS4} we also get
\begin{equation}\label{e:C0-100}
\|g_H - \p_{\pi_0}^\perp (p_H)\|_{C^0} \leq C \bmo^{\sfrac{1}{2m}}  \ell (H)^{1+\beta_2}\, .
\end{equation}
Finally, recall that, if $E:= \bE (T, \bC_{32 r_H} (p_H, \pi_H))$, then
\[
\int_{B_{8r_H} (p_H, \pi_H)} |Df_H|^2 \leq C \bmo \ell (H)^{m+2-2\delta_2}\, ,
\] 
from which clearly we get
\[
\int_{B_{8r_H} (p_H, \pi_H)} |D \etaa \circ f_H|^2 \leq C \bmo  \ell (H)^{m+2-2\delta_1}\, .
\]
By the last estimate in Lemma \ref{l:exist}, we deduce
\[
\int_{B_{8r_H} (p_H, \pi_H)} |D {h}_H|^2 \leq C \bmo  \ell (H)^{m+2-2\delta_2}\, .
\]
Thus we conclude the existence of a point $p$ such that 
\begin{equation}\label{e:bound_tangente}
|D h_H (p)|\leq C \bmo^{\sfrac{1}{2}}  \ell (H)^{1-\delta_2}\,.
\end{equation}
Using the bound on $\|D^2 h_H\|_0$, we conclude $|D{h}_H (q)| \leq C \bmo^{\sfrac{1}{2}}  \ell (H)^{1-\delta_2}$ for all $q$'s in the domain of ${h}_H$.  This implies the estimate
\[
|T_p \bG_{h_H} - \pi_H| \leq C \bmo^{\sfrac{1}{2}} \ell (H)^{1-\delta_2}\, 
\qquad \forall p\in \gr (h_H)\cap \bC_{6r_H} (p_H, \pi_H)\, .
\]
Since however $\gr (g_H) \subset  \gr (h_H)\cap \bC_{6r_H} (p_H, \pi_H)$, we then conclude (v). 

We next come to (iii). Fix therefore two cubes $H$ and $L$ as in the statement and set $r:= r_L$. Observe that, by the hoelder estimates in (ii) and the interpolation lemma \cite[Lemma C.2]{DS4}, it suffices to show that $\|g_H- g_L\|_{L^1 (B)} \leq C \bmo^{\sfrac{1}{2}} \ell (H)^{m+3+\kappa}$, where $B=B_r (x_L, \pi_0)$. Consider now the two corresponding tilted interpolating functions, namely
$h_L$ and $h_H$. Given the estimate upon $h_L$ proved in the previous paragraph, we can find a function $\hat{h}_L: B_{7 r} (p_{HL}, \pi_H) \to \pi_H^\perp$ such that
$\bG_{\hat{h}_L} = \bG_{h_L} \res \bC_{6r} (p_L, \pi_H)$ (in this paragraph $\hat\cdot$ will always denote the reparametrization on $\pi_H$). Obviously $\bG_{\hat{h}_L} \res \bC_{r} (z_L, \pi_0)= \bG_{g_L}$. We can therefore apply \cite[Lemma B.1]{DS4} to conclude that
\[
\|g_H - g_L\|_{L^1 (B)} \leq C \|h_H - \hat{h}_L\|_{L^1 (B_{5r} (p_L, \pi_H))}\, .
\]
Consider next the tilted interpolating function $h_{HL}$ and observe that, by \eqref{e:induttive_1}, we know 
\[
\|h_H - h_{HL}\|_{L^1 (B_{5r} (p_H, \pi_H))} \leq C \bmo^{\sfrac{1}{2}} \ell (H)^{m+3+\beta_2}\, .
\]
Indeed, by \eqref{e:stima-L1}, it is enough to see that
\begin{align*}
\|\etaa\circ f_{HH} - \etaa\circ f_{HL}\|_{L^1 (B_{5r} (p_H, \pi_H))} 
& \leq \|\etaa\circ f_{HH} - \etaa\circ f_{HL}\|_{L^\infty (B_{5r} (p_H, \pi_H))} \, |B_{5r} (p_H, \pi_H)\setminus \bar K|\\
&\leq 
C \bmo^{1+\beta_0+\sfrac14}  \ell(H)^{m+3+\beta_2}\, ,
\end{align*}
where, as above, we used the estimates of \cite[Theorem 1.5]{DSS2} on the oscillation of $f_{HH}$ and $f_{HL}$ (with $\bar K$ the projection of $\gr (f_{HH})\cap \gr (f_{HL})$ onto $p_{HL} + \pi_H$) and $B_{5r} (p_H, \pi_H) \subset  B_{8r} (p_{HL}, \pi_H)$.
Hence, since $\beta_2 \geq \kappa$, we are reduced to show
\begin{equation}\label{e:goal}
\|h_{HL} - \hat{h}_L\|_{L^1 (B_{5r} (p_H, \pi_H))}\leq C \bmo^{\sfrac{1}{2}} \ell (H)^{m+3+\kappa}\, .
\end{equation}
In turn, consider the $\pi_H$-approximating function $f_{HL}$ and the $\pi_L$-approximating function $f_{LL} = f_L$. 
We set $\bef_{HL} (x) = \etaa\circ f_{HL} (x)$
and recall that, by Proposition \ref{p:stime}, we have
\begin{equation}\label{e:pezzo_1}
\|h_{HL}- \bef_{HL}\|_{L^1 (B_{8r_L} (p_{HL}, \pi_H))} \leq C  \bmo^{\sfrac{1}{2}}  \ell (L)^{m+3+\beta_2}\, .
\end{equation}
Similarly, we set $\bef_{L} (x) = \etaa\circ f_{L} (x)$
and we get
\[
\|h_{L}- \bef_{L}\|_{L^1 (B_{8r_L} (p_L, \pi_L))} \leq C  \bmo^{\sfrac{1}{2}}  \ell (L)^{m+3+\beta_2}\, .
\]
Next we denote by $\hat{\bef}_L$ the map $\hat{\bef}_L : B_{6r_L} (p_{HL}, \pi_H)\to \pi_H^\perp$ such that $\bG_{\hat{\bef}_L} = \bG_{\bef_L} \res \bC_{6r_L} (p_L, \pi_H)$
and we use again \cite[Lemma B.1]{DS4} to infer
\begin{equation}\label{e:pezzo_2}
\|\hat h_{L}- \hat \bef_{L}\|_{L^1 (B_{6r_L} (p_{HL}, \pi_H))} \leq C \|h_{L}- \bef_{L}\|_{L^1 (B_{8r_L} (p_L, \pi_L)}
\leq C  \bmo^{\sfrac{1}{2}}  \ell (L)^{m+3+\beta_2}\, .
\end{equation}
In view of \eqref{e:pezzo_1} and \eqref{e:pezzo_2}, \eqref{e:goal} is then reduced to
\begin{equation}\label{e:goal_2}
\|\bef_{HL} - \hat{\bef}_L\|_{L^1 (B_{5r_L} (p_{HL}, \pi_H))}\leq C \bmo^{\sfrac{1}{2}} \d (H)^{2(1+\beta_0)\gamma_0-\beta_2 -2} \ell (H)^{5+\kappa}\, .
\end{equation}
Consider now the map $\hat{f}_L: \B_{6r_L} (p_{HL}, \pi_H)\to \Iq (\pi_H^\perp)$ such that $\bG_{\hat{f}_L} = \bG_{f_L}\res \bC_{6r_L} (p_L, \pi_H)$. Let $A$ and $\hat{A}$ be the projections on $p_H+\pi_H = p_{HL} +\pi_H$ of the Borel sets $\gr (f_{HL}))\setminus \supp (T)$ and $\gr (\hat f_{L})\setminus \supp (T) \subset \gr (f_L)\setminus \supp (T)$. We know that
\begin{align*}
|A\cup A'| \leq &\|\bG_{f_{HL}} - T\| (\bC_{8r_L} (p_L, \pi_H)) + \|\bG_{f_L} - T \| (\bC_{8r_L} (p_L, \pi_L))\\
 \leq &C \bmo^{1+\beta_0}  \ell (H)^{m+2}\, .
\end{align*}
On the other hand, it is not difficult to see, thanks to the height bound, that $\|\etaa\circ f_{HL} - \etaa \circ {\hat{f}_L}\|_\infty\leq C \bmo^{\sfrac{1}{2m}}  \ell (H)^{1+\beta_2}$. We thus conclude that
\[
\|\etaa\circ f_{HL} - \etaa \circ{\hat{f}_L}\|_{L^1 (B_{6r_L} (p_{HL}, \pi_H))} \leq C \bmo^{\sfrac{1}{2}}  \ell (H)^{m+3+\beta_2}\, .
\]
Define the function $\beg (x) := \etaa \circ \hat{f}_L (x)$, we can thus conclude that
\begin{equation}\label{e:pezzo_3}
\|\bef_{HL}- \beg\|_{L^1 (B_{6r_L} (p_{HL}, \pi_L))}
\leq C  \bmo^{\sfrac{1}{2}} \ell (L)^{m+3+\beta_2}\, .
\end{equation}
Thus, \eqref{e:goal_2} is now reduced to
\begin{equation}\label{e:goal_3}
\|\beg - \hat{\bef}_L\|_{L^1 (B_{5r_L} (p_{HL}, \pi_H))}\leq C \bmo^{\sfrac{1}{2}}  \ell (H)^{m+3+\kappa}\, .
\end{equation}
Denoting by ${\rm An}$ the distance $|\pi_H-\pi_L|$, by $\hat{B}$ the ball $B_{6r_L} (p_{HL}, \pi_H)$ and by $\tilde{B}$ the ball $B_{8r_L} (p_L, \pi_L)$, we then have, by \cite[Lemma 5.6]{DS4}
\[
\|\beg - \hat{\bef}_L\|_{L^1 (\hat B)}\leq C_0 ({\rm osc}\, (f_L) + r_L {\rm An}) \left(\int |Df_L|^2 + r_L^2 (\|D\Psi_{p_L}\|^2_{C^0 (\tilde{B})} + {\rm An}^2)\right)\, .
\]
Recall that $D\Psi_{p_L} (p_L) =0$ and thus $\|D\Psi_{p_L}\|^2_{C^0 (\tilde{B})}\leq C_0 \bmo r_L^2$. Recalling the estimate on $|\pi_H-\pi_L|$ and upon the Dirichlet energy of $f_L$, we then conclude
\[
\int |Df_L|^2 + r_L^2 (\|D\Psi_{p_L}\|^2_{C^0 (\tilde{B})} + {\rm An}^2)\leq C \bmo  \ell (H)^{4-2\delta_1}\, .
\]
On the other hand
\[
{\rm osc}\, (f_L) + r_L {\rm An}\leq C \bmo^{\sfrac{1}{2m}}  \ell (H)^{1+\beta_2}\, .
\]
Thus \eqref{e:goal_3} follows by our choice of the various parameters, in particular $\beta_2-2\delta_1\geq \sfrac{\beta_2}{4}=\kappa$.

Next, if $L\in \sW^j$ and $i\geq j$,
consider the subset $\sP^i (L)$ of all cubes in $\sP^i$ which intersect $L$. If $L'$ is the cube concentric to $L$
with $\ell (L')=\frac{9}{8} \ell (L)$,
we then have by definition of $\varphi_j$:
\begin{equation}\label{e:L1_importante}
\|\varphi_i - g_L\|_{L^1 (L')} \leq C \sum_{H\in \sP^i (L)} \|g_H - g_L\|_{L^1(B_{r_L} (p_L, \pi_0))}
\leq C \bmo\, \ell (H)^{m+3+2\kappa}\, ,
\end{equation}
which is the claim of (vi).

As for (iv), it is a rather simple consequence of (iii). Indeed fix $H$ and $L$
as in the statements and consider $H = H_i \subset H_{i-1} \subset \ldots \subset H_{N_0}$ and $L = L_j \subset L_{j-1}\subset \ldots \subset L_{N_0}$ so that
$H_l$ is the father of $H_{l+1}$ and $L_l$ is the father of $L_{l+1}$. We distinguish two cases:
\begin{itemize}
\item[(A)] If $H_{N_0}\cap L_{N_0}\neq \emptyset$, we let $i_0$ be the smallest index so that $H_{i_0}\cap L_{i_0}\neq \emptyset$;
\item[(B)] $H_{N_0} \cap L_{N_0} =\emptyset$.
\end{itemize}
In case (A) observe that $\max \{\ell (H_{i_0}), \ell (L_{i_0})\}\leq d (x_H, x_L) := d$. On the other hand, by (iii) with $l=3$ we have
\begin{align*}
|D^3 g_H (x_H) - D^3 g_{H_{i_0}} (x_{H_{i_0}})|
&\leq \sum_{l=i_0}^{i-1} |D^3 g_{H_l} (x_{H_l}) - D^3 g_{H_{l+1}} (x_{H_{l+1}})|\\
&\leq  C \bmo^{\sfrac{1}{2}} \ell (H_{i_0})^\kappa \sum_{l=i_0}^{i-1} 2^{(i_0-l)\kappa} \leq C \bmo^{\sfrac{1}{2}}  d^\kappa\\
|D^3 g_L (x_L) - D^3 g_{L_{i_0}} (x_{L_{i_0}})|
&\leq \sum_{l=i_0}^{j-1} |D^3 g_{L_l} (x_{L_l}) - D^3 g_{L_{l+1}} (x_{L_{l+1}})|\\
&\leq C \bmo^{\sfrac{1}{2}}  \ell (L_{i_0})^\kappa \sum_{l=i_0}^{j-1} 2^{(i_0-l)\kappa} \leq C \bmo^{\sfrac{1}{2}} d^\kappa 
\end{align*}
\begin{align*}
&|D^3 g_{L_{i_0}} (z_{L_{i_0}}, w_{L_{i_0}}) - D^3 g_{H_{i_0}} (z_{H_{i_0}}, w_{H_{i_0}})| \leq C \bmo^{\sfrac{1}{2}} \d (H_{i_0})^{2(1+\beta_0)\gamma_0-\beta_2  -2} \ell (H_{i_0})^{\kappa}\\
\leq &\bmo^{\sfrac{1}{2}} \d (H)^{2(1+\beta_0)\gamma_0-\beta_2  -2} d^\kappa\, .
\end{align*}
The triangle inequality implies then the desired estimate.

In case (B) we first notice that by the very same argument we have the estimates
\begin{align*}
|D^3 g_H (x_H) - D^3 g_{H_{N_0}} (x_{H_{N_0}})|&\leq C \bmo^{\sfrac{1}{2}} d^\kappa\\
|D^3 g_L (x_L) - D^3 g_{L_{N_0}} (x_{L_{N_0}})|&\leq  C \bmo^{\sfrac{1}{2}} d^\kappa\, .
\end{align*}
Next we find a chain of cubes $H_{N_0} =J_0, J_1, \ldots, J_N = L_{N_0}$, all distinct and belonging to $\sS^{N_0}$, so that 
\begin{itemize}
\item $J_l\cap J_{l+1}\neq \emptyset$ and thus $\ell (H_{N_0}) \leq \ell (J_l)\leq \ell (L_{N_0})$;
\item $N$ is smaller than a constant $C (N_0, \bar{Q})$. 
\end{itemize}
Using again (iii) and arguing as above we conclude
\begin{align*}
&| D^3 g_{H_{N_0}} (x_{H_{N_0}}) - D^3 g_{L_{N_0}} (x_{L_{N_0}})|\\
\leq &\sum_{l=1}^{N} |D^3 g_{J_l} (x_{J_l}) - D^3 g_{J_{l-1}} (x_{J_{l-1}})|
\leq C N \bmo^{\sfrac{1}{2}} d^\kappa\, .
\end{align*}
Again, using the triangular inequality we conclude (iv).

\end{proof}

\begin{proof}[Proof of Theorem \ref{t:cm}]
The proof is a simple consequence of the estimates of Proposition \ref{p:stime_chiave}, and it can be found in \cite[Section 4.4]{DS4}. 
Notice in particular that, since the estimates of Proposition \ref{p:stime_chiave} are exactly the same as the ones in \cite[Proposition 4.4]{DS4}, and so is the definition of the $\ph_j$, we can apply the proof of \cite{DS4} without any change to prove (i). Analogously (ii) follows directly from the definition of $\ph_j$ and the algorithm of the Whitney decomposition. Finally, (iii) follows from (i) and (ii).
\end{proof}

\section{Normal Approximation: proofs of the main results}

\subsection{Construction of the normal approximation}
The algorithm for the construction of the normal approximation is analogous to that of \cite[Section 6]{DS4} (in fact it is even simpler, since $\Sigma=\R^{n+m}$). Moreover the resulting estimates are also the same, because they depend only on the construction and estimates for the Center Manifold (these estimates are precisely the reason why we built it!). 
 
In particular, the only relevant properties needed for the proofs of \cite[Corollary 2.2, Theorem 2.4 \& Theorem 2.5]{DS4} are \cite[Proposition 4.1, Theorem 1.17 \& Theorem A.1]{DS4}, which are identical in their statements to Proposition \ref{p:gira_e_rigira}, Theorem \ref{t:cm} and the Height Bound Theorem \ref{t:height_bound}. As a consequence Corollary \ref{c:cover}, Theorem \ref{t:approx} and Corollary \ref{c:globali} then follow, with exactly the same proofs.

\subsection{Vertical separation and splitting before tilting}

The proofs of Proposition \ref{p:separ} and Corollary \ref{c:domains} are analogous respectively to those of \cite[Proposition 3.1 \& Corollary 3.2]{DS4}. Indeed, once again, except for the estimates on the Center Manifold and the Normal approximation, which are the same as in \cite{DS4}, the only other ingredient needed here is the Height bound of Theorem \ref{t:height_bound}.

The proof of Proposition \ref{p:splitting} is slightly different than the one in \cite[Proposition 3.4]{DS4}, in the sense that we have to choose our parameters more carefully. This is because the condition for the Lipschitz approximation of $T$ to be close to a $\D$-minimizing function is more delicate when $T$ is semicalibrated (cf. \cite[Theorem 4.2]{DSS2}).

\begin{proof}[Proof of Proposition \ref{p:splitting}]
As customary we use the convention that constants denoted by $C$ depend upon all the parameters but $\eps_2$, whereas constants denoted by $C_0$ depend only upon $m,n,\bar{n}$ and $Q$.

Given $L\in \sW^j_e$, let us consider
its ancestors $H\in \sS^{j-1}$ and $J\in \sS^{j-6}$. 
Set $\ell = \ell(L)$,$\pi = \hat{\pi}_H$ and $\bC := \bC_{8r_J} (p_J, \pi)$,  and
let $f:B_{8r_J} (p_J, \pi)\to \Iq (\pi^\perp)$ be the $\pi$-approximation of Definition \ref{d:pi-approximations},
which is the result of \cite[Theorem 1.4]{DS3} applied to $\bC_{32r_J} (p_J, \pi)$
(recall that Proposition~\ref{p:gira_e_rigira}(i) ensures the applicability of \cite[Theorem 1.4]{DS3} in the latter cylinder).
We let $K\subset B_{8r_J} (p_J, \pi)$ denote the set of \cite[Theorem 1.4]{DS3} and recall that
$\bG_{f\vert_K} = T\res K\times \pi^\perp$. Observe that $\B_L\subset \B_H \subset \bC$  (this requires, as usual, $\eps_2 \leq c( \beta_2, \delta_2, M_0, N_0, C_e, C_h)$).
The following are simple consequences of Proposition~\ref{p:tilting opt} and the fact that $J\notin \sW$:
\begin{gather}
E := \bE (T, \bC_{32 r_J} (p_J, \pi)) \leq
 C_e \bmo \,\ell (J)^{2-2\delta_2}\leq 
C_0\, C_e \bmo \,\ell^{2-2\delta_2}\, ,\label{e:eccesso J}\\
\bh (T, \bC, \pi) 
\leq C_0\,C_h\,\bmo^{\sfrac{1}{2m}} \ell^{1+\beta_2} .\label{e:altezza J}
\end{gather}
Moreover, since $\B_L \subset \bC$, $L\in \sW_e$ and $r_L / r_J = 2^{-6}$, we have
\begin{equation}\label{e:eccesso alto}
c C_e \,\bmo\,r_L^{2-2\delta} \leq E\, ,
\end{equation}
where $c$ is only a geometric constant.
We divide the proof of Proposition~\ref{p:splitting} in three steps.

\medskip

{\bf Step 1: decay estimate for $f$.}
Let $2\rho:= 64 r_H - C^\sharp \bmo^{\sfrac{1}{2m}} \ell^{1+\beta_2}$:
since $p_H \in \supp (T)$, it follows from \eqref{e:altezza J} that, upon choosing $C^\sharp$ appropriately,
$\supp (T) \cap \bC_{2\rho} (p_H, \pi)\subset \B_H\subset \bC$ (observe that $C^\sharp$ depends
upon the parameters $\beta_2, \delta_2, M_0, N_0, C_e$ and $C_h$, but not on $\eps_2$).
Setting $B=B_{2\rho}(x, \pi)$ with $x = \p_{\pi} (p_H)$,
using the Taylor expansion in \cite[Corollary 3.3]{DS2} and the estimates
in \cite[Theorem 1.4]{DS3}, we then get
\begin{align}
\D (B, f) &\leq 2 |B|\, \bE (T, \bC_{2\rho} (x_H, \pi)) +
C \bmo^{1+\beta_0} \ell^{m+2+\sfrac{\beta_0}{2}}\nonumber\\
& \leq 2 \omega_m (2\rho)^m \bE (T, \B_H) + C\bmo^{1+\beta_0} \ell^{m+2+\sfrac{\beta_0}{2}}\, .\label{e:Dir-Ex}
\end{align}
Consider next the cylinder $\bC_{64 r_L} (p_L, \pi)$ and
set $x':= \p_{\pi}(p_L)$. Recall that $|x-x'|\leq |p_H-p_L| \leq C \ell (H)$, where $C$ is a geometric constant (cp.~Proposition~\ref{p:tilting opt}) and set
$\sigma:= 64 r_L+ C \ell (H) =  32 r_H + C \ell (H)$.
If $\lambda$ is such that $(1+\lambda)^{(m+2)}< 2^{\delta_2}$ and $M_0$ is chosen sufficiently
large (thus fixing a lower bound for 
$M_0$ which depends only on $\delta_2$) we reach
\[
\sigma \leq \left(\frac{1}{2} +\frac{\lambda}{4}\right) \,64\, r_H
\leq \left(1+\frac{\lambda}{2}\right) \rho + C^\sharp \bmo^{\sfrac{1}{2m}} \ell^{1+\beta_2}\, .
\]
In particular, choosing $\eps_2$ sufficiently small we conclude $\sigma \leq (1+\lambda) \rho$ and thus also $\B_L\subset \bC_{64 r_L} (x', \pi)\subset \bC_{(1+\lambda)\rho} (x, \pi) =: \bC'$. Define $B':= B_{(1+\lambda)\rho} (x, \pi)$, set
$A := \mint_{B'} D (\etaa \circ f)$, let
$\mathcal{A}: \pi\to \pi^\perp$ be the linear map $x\mapsto A\cdot x$
and let $\tau$ be the plane corresponding to $\bG_{\mathcal{A}}$.
Using \cite[Theorem 3.5]{DS2}, we can  estimate
\begin{align}
{\textstyle{\frac{1}{2}}} \int_{B'} \cG (Df, Q \a{A})^2 &\geq |B'| \, \bE (T, \bC', \tau) - C \bmo^{1+\beta_0} \ell^{m+2+\sfrac{\beta_0}{2}}\nonumber\\
&\geq |B'| \bE (T, \B_L, \tau) - C \bmo^{1+\beta_0} \ell^{m+2+\sfrac{\beta_0}{2}}\nonumber\\
&\geq \omega_m ((1+\lambda)\rho)^m \bE (T, \B_L) - C \bmo^{1+\beta_0} \ell^{m+2+\sfrac{\beta_0}{2}}\, . \label{e:da_sotto}
\end{align} 
We thus conclude
\begin{align}
\D (B, f) &\leq 2 \omega_m (2\rho)^m  \bE (T, \B_H) +C  \bmo^{1+\beta_0} \rho^{m+2}\, .\label{e:da_sopra_again}\\
\int_{B'} \cG (Df, Q \llbracket {A}\rrbracket)^2&\geq 2 \omega_m ((1+\lambda) \rho)^m \bE (T, \B_L) - C  \bmo^{1+\beta_0} \rho^{m+2}\, .\label{e:da_sotto_again}
\end{align}

\medskip

%

{\bf Step 2: harmonic approximation.} From now on, to simplify our notation,
we use $B_s (y)$ in place of $B_s (y, \pi)$. Set $p:=\p_{\pi} (p_J)$. 
From \eqref{e:eccesso alto} we infer that, recall that $C_e>1$)  
\begin{equation}\label{e:problematica}
8 r_J\, \bOmega \leq M_0 \sqrt{m} 2^{-6} \ell (L) \bmo^{\sfrac{1}{2}} \leq C(M_0,m)
\frac{\ell(L)^{\delta_2}}{C_e^{\sfrac12}}\,C_e^{\sfrac12} \bmo^{\sfrac{1}{2}}\,  \ell (L)^{1-\delta_2}\leq C(M_0, m)\, \,2^{-N_0 \delta_2}\, \bE^{\sfrac12},
\end{equation}
Therefore, for every positive $\bar{\eta}$, we can choose $N_0$ big enough depending on $M_0$, so that $\sfrac{\ell(L)^{\delta_2}}{C_e^{\sfrac12}}$ is small and we can apply \cite[Theorem 4.2]{DSS2} to the cylinder $\bC$
and find a $\D$-minimizing map $u: B_{8r_J} (p, \pi)
\to \Iq (\pi^\perp)$ such that
\begin{gather}
(8\,r_J)^{-2} \int_{B_{8r_J} (p)} \cG (f,u)^2 + \int_{B_{8r_J} (p)} (|Df|-|Du|)^2 \leq \bar{\eta} \,E \,(8\,r_J)^m,\label{e:armonica_vicina}\\
\int_{B_{8r_J} (p)} |D (\etaa \circ f) - D (\etaa \circ u)|^2 \leq \bar{\eta}\, E \, (8\,r_J)^m\, .
\label{e:armonica_vicina_2}
\end{gather}
Now, since $D (\etaa\circ u) = \etaa\circ Du$ is harmonic we have $D (\etaa \circ u) (x) = \mint_{B'} (\etaa \circ Du)$. So we can combine respectively \eqref{e:armonica_vicina} with \eqref{e:da_sopra_again} and \eqref{e:armonica_vicina_2} with \eqref{e:da_sotto_again} to infer
\begin{gather}
\D(B, u)\leq 2 \omega_m (2\rho)^m  \bE (T, \B_H) +C  \bmo^{1+\beta_0} \rho^{m+2} +C_0 \bar\eta^{\sfrac{1}{2}} E \rho^m  \label{e:stima_intermedia_3}\\
\int_{B_{(1+\lambda) \rho} (x)} \cG \big(Du, Q\llbracket\etaa \circ Du (x)\rrbracket\big)^2 \geq
2 \omega_m ((1+\lambda) \rho)^m \bE (T, \B_L) - C  \bmo^{1+\beta_0} \rho^{m+2} - C_0 \bar\eta^{\sfrac{1}{2}} E \rho^m\, .\label{e:stima_intermedia_2}
\end{gather}
Next, recall that $\bE (T, \B_L) \geq C_e \bmo \ell (L)^{2-2\delta_2} \geq 2^{2\delta_2-2} \bE (T, \B_H)$ and that, by \eqref{e:eccesso J}, $E\rho^m\leq C_0\,C_e\,\bmo\, \rho^{m+2-2\delta_2}$.
Using once again $L\in \sW_e$ and this last inequality,
\begin{align}
 \bmo \rho^{m+2} 
&\leq \frac{\rho^{2\delta_2}}{C_e} \bE(T_{L},\B_{L})\rho^m\stackrel{\eqref{e:stima_intermedia_2}}{\leq} \frac{C_0}{C_e} \int_{B'}|Du|^2 + \frac{\rho^{2\delta_1}}{C_e} \bmo^{1+\beta_0} \rho^{m+2}+C_0\,\bar \eta^{\sfrac12}\,\bmo \,\rho^{m+2}  \, ,\label{e:pallosissima}
\end{align}
that is, for $N_0$ sufficiently big, and so $\bar \eta$ sufficiently small,
\begin{equation}\label{e:sfavone}
\bmo \rho^{m+2} \leq C_0\,\int_{B'}|Du|^2
\quad\text{and}\quad 
E\rho^m\leq C_0\, \int_{B'}|Du|^2
\end{equation}
We can therefore combine \eqref{e:stima_intermedia_3} and \eqref{e:stima_intermedia_2} with \eqref{e:sfavone}, to achieve
\begin{align}
\int_{B_{(1+\lambda) \rho} (x)} \cG \big(Du, Q\llbracket D(\etaa \circ u) (x)\rrbracket\big)^2
&\geq \big(2^{2\delta_2-2 - m} - C_0 \bar{\eta}^{\sfrac12} - C_0 \bmo^{\beta_0}\big)\int_{B_{2\rho} (x)} |Du|^2\, .
\nonumber
\end{align}
It is crucial that the constant $C_0$ in front of $\bar \eta$ depends only upon $Q, m, \bar{n}$ and not on $N_0$. 
So, if $N_0$ is chosen sufficiently large, depending only upon $\lambda$ (and hence upon $\delta_2$) and $\bar \eta$, we can require that
$2^{2\delta_2 -2-m} - \bar{\eta}^{\sfrac12} \geq 2^{3\delta_2/4 -2 -m}$. We then require $\eps_2$ to be sufficiently small
so that  $2^{3\delta_2/4 -2 -m} - C \bmo^{\beta_0}  \geq 2^{\delta_2-2-m}$.
We can now apply \cite[Lemma 7.1]{DS4} and \cite[Proposition 7.2]{DS4} to $u$ and conclude
\begin{align*}
\hat{C}^{-1} \int_{B_{(1+\lambda) \rho} (x)} |Du|^2 \leq \int_{B_{\ell/8} (q)} \cG (Du, Q \a{D (\etaa\circ u})^2
\leq \hat{C} \ell^{-2} \int_{B_{\ell/8} (q)} \cG (u, Q \a{\etaa\circ u})^2\, ,
\end{align*}
for any ball $B_{\ell/8} (q) = B_{\ell/8} (q, \pi) \subset B_{8r_J} (p, \pi)$,
where $\hat{C}$ depends upon $\delta_2$ and $M_0$. In particular, being these constants independent of
$\eps_2$ and $N_0$, we can use the previous estimates and reabsorb error terms (possibly choosing $\eps_2$ even smaller and $N_0$ larger) to conclude
\begin{align}\label{e:stima dritta}
\bmo \, \ell^{m+2-2\delta_2} &\stackrel{\eqref{e:pallosissima}}{\leq} \frac{\tilde C}{C_e} \ell^m \,\bE (T, \B_L) \leq \frac{\bar{C}}{C_e} \int_{B_{\ell/8} (q)} 
\cG (Df, Q\a{D (\etaa \circ f)})^2\notag\\
&\leq \frac{\check{C}}{C_e} \ell^{-2} \int_{B_{\ell/8} (q)} \cG (f, Q \a{\etaa \circ f})^2,
\end{align}
where $\tilde C$, $\bar{C}$ and $\check{C}$ are constants which depend only upon $\delta_2$, $M_0$ and $N_0$, but not
on $C_e$ and $\eps_2$.

\medskip

{\bf Step 3: Estimate for the $\cM$-normal approximation.}
Now, consider any ball $B_{\ell/4} (q, \pi_0)$ with $\dist (L, q)\leq 4\sqrt{m} \,\ell$
and  let $\Omega:= \Phii (B_{\ell/4} (q, \pi_0))$.
Observe that $\p_{\pi} (\Omega)$ must contain a ball
$B_{\ell/8} (q', \pi)$, because of the estimates on $\phii$ and $|\pi_0-\hat{\pi}_H|$, and in turn
it must be contained in $B_{8r_J} (p, \pi)$.
Moreover, $\p^{-1} (\Omega) \cap \supp (T) \supset \bC_{\ell/8} (q', \pi)\cap \supp (T)$
and, for an appropriate geometric constant $C_0$, 
$\Omega$ cannot intersect a Whitney region $\cL'$
corresponding to an $L'$ with $\ell (L') \geq C_0 \ell (L)$.
In particular, Theorem~\ref{t:approx} implies that
\begin{equation}\label{e:masses}
\|\bT_F - T\| (\p^{-1} (\Omega)) + \|\bT_F - \bG_f\| (\p^{-1} (\Omega)) \leq C \bmo^{1+\gamma_2} 
\ell^{m+2+\gamma_2}.
\end{equation}
Let now $F'$ be the map such that 
$\bT_{F'} \res (\p^{-1} (\Omega)) = \bG_f \res (\p^{-1} (\Omega))$ and let $N'$ be the corresponding
normal part, i.e. $F' (x) = \sum_i \a{x+N'_i (x)}$.
The region over which $F$ and $F'$ differ is contained in the projection
onto $\Omega$ of $(\im (F) \setminus \supp (T))
\cup (\im(F') \setminus \supp (T))$ and
therefore its $\cH^m$ measure is bounded as in \eqref{e:masses}.
Recalling the height bound on $N$ and $f$, we easily conclude $|N|+|N'| \leq C \bmo^{\sfrac{1}{2m}}
\ell^{1+\beta_2}$, which in turn implies
\begin{equation}
\int_{\Omega} |N|^2 \geq \int_{\Omega} |N'|^2 - C \bmo^{1+\sfrac{1}{m}+\gamma_2} \ell^{m+4+2\beta_2+\gamma_2}\, .
\end{equation}
On the other hand, let $\phii': B_{8r_J} (p, \pi)\to\pi^\perp$
be such that $\bG_{\phii'} =\a{\cM}$ and $\Phii'(z) = (z, \phii'(z))$; then,
applying \cite[Theorem 5.1 (5.3)]{DS2}, we conclude
\[
|N' (\Phii' (z))| \geq \frac{1}{2\sqrt{Q}}\, \cG (f(z), Q \a{\phii' (z)}) \geq 
\frac{1}{4\sqrt{Q}}\, \cG (f(z), Q\a{\etaa\circ f (z)})\,,
\]
which in turn implies
\begin{align}
\bmo\, \ell^{m+2-2\delta_2} &\stackrel{\eqref{e:stima dritta}}{\leq} C \ell^{-2} \int_{B_{\ell/8} (q', \pi)}  \cG (f, Q \a{\etaa \circ f})^2
\leq C \ell^{-2} \int_\Omega |N'|^2\nonumber\\ 
&\leq C \ell^{-2} \int_\Omega |N|^2 + C \bmo^{1+\gamma_2+\sfrac{1}{2m}} \ell^{m+2+2\beta_2+\gamma_2}\, .
\end{align}
For $\eps_2$ sufficiently small, this leads to the second inequality of \eqref{e:split_2}, while
the first one comes from Theorem~\ref{t:approx} and $\bE (T, \B_L) \geq C_e \bmo \,\ell^{2-2\delta_2}$.

We next complete the proof showing \eqref{e:split_1}. Since $D (\etaa \circ f) (z) =
\etaa \circ Df (z)$ for a.e. $z$, we obviously have
\begin{equation}
\int_{B_{\ell/8} (q', \pi)} \cG (Df, Q \a{D(\etaa \circ f)})^2 \leq  
\int_{B_{\ell/8} (q', \pi)} \cG (Df, Q \a{D\phii'})^2\, .
\end{equation}
Let now $\vec{\bG}_f$ be the orienting tangent $m$-vector to $\bG_f$ and $\tau$
 the one to $\cM$. For a.e. $z$ we have the inequality
\[
C_0 \sum_i |\vec{\bG}_f (f_i (z)) - \vec{\tau} (\phii' (z))|^2 \geq  \cG (Df (z), Q \a{D\phii' (z)})^2\, ,  
\]
for some geometric constant $C_0$, because $|\vec{\bG}_f (f_i (z)) - \vec{\tau} (\phii' (z))| \leq C \bmo^{\gamma_2}$
(thus it suffices to have $\eps_2$ sufficiently small). Hence
\begin{align}
\int_{B_{\ell/8} (q', \pi)} &\cG (Df, Q \a{D\phii'})^2 \leq
C \int_{\bC_{\ell/8} (q', \pi)} |\vec{\bG}_f (z) - \vec\tau (\phii' (\p_{\pi} (z))|^2 d\|\bG_f\| (z)\nonumber\\
&\leq C \int_{\bC_{\ell/8} (q', \pi)} |\vec{T} (z) - \vec\tau (\phii' (\p_{\pi} (z))|^2 d\|T\| (z)
+ C \bmo^{1+\beta_0} \ell^{m+2+ \beta_0}\, .\label{e:quasi finale}
\end{align}
Now, thanks to the height bound and to the fact that $|\vec{\tau} - \pi|\leq |\vec{\tau} - \pi_H| + |\pi_H - \pi| \leq C \bmo^{\sfrac{1}{2}} \ell^{1-\delta_2}$
in the cylinder $\hat\bC = \bC_{\ell/8} (q', \pi)$, we have the inequality
\[
|\p (z) - \phii' (\p_{\pi} (z))|
\leq C \bmo^{\sfrac{1}{2m} + \sfrac{1}{2}} \ell^{2+\beta_2-\delta_2} 
\leq C \bmo^{\sfrac{1}{2m} + \sfrac{1}{2}} \ell^{2+\sfrac{\beta_2}{2}} \qquad
\forall z\in \supp (T)\cap \hat\bC\, .
\]
Using $\|\phii'\|_{C^2} \leq C \bmo^{\sfrac{1}{2}}$
we then easily conclude from \eqref{e:quasi finale} that
\begin{align*}
 \int_{B_{\ell/8} (p, \pi)} \cG (Df, Q \a{D\phii'})^2 &\leq C_0 \int_{\hat\bC} 
|\vec{T} (z) - \vec\tau (\p (z))|^2 d\|T\| (z)
+ C \bmo^{1+\beta_0} \ell^{m+2+ \sfrac{\beta_2}{2}}\, \\
& \leq C_0 \int_{\p^{-1} (\Omega)} |\vec{\bT}_F (z) - \vec{\tau} (\p (z))|^2 d\|\bT_F\| (z) + 
C \bmo^{1+\gamma_2} \ell^{m+2+\gamma_2},
\end{align*}
where we used \eqref{e:masses}.

Since $|DN| \leq C \bmo^{\gamma_2} \ell^{\gamma_2}$,
$|N|\leq C \bmo^{\sfrac{1}{2m}} \ell^{1+\beta_2}$
on $\Omega$ and $\|A_{\cM}\|^2 \leq C \bmo$,
applying now  \cite[Proposition 3.4]{DS2} we conclude
\begin{equation*}
\int_{\p^{-1} (\Omega)}
|\vec{\bT}_F (x) - \tau (\p (x))|^2 d\|\bT_F\| (x)
\leq (1+ C \bmo^{2\gamma_2} \ell^{2\gamma_2}) \int_{\Omega} |DN|^2 + C \bmo^{1+\sfrac{1}{m}} \ell^{m+2+2\beta_2}\, .
\end{equation*}
Thus, putting all these estimates together we achieve
\begin{equation}
\bmo \, \ell^{m+2-2\delta_2} \leq C (1 + C \bmo^{2\gamma_2} \ell^{2\gamma_2}) \int_{\Omega} |DN|^2
+ C \bmo^{1+\gamma_2} \ell^{m+2+\gamma_2}\, .
\end{equation}
Since the constant $C$ might depend on the various other parameters but not on $\eps_2$, we conclude that
for a sufficiently small $\eps_2$ we have
\begin{equation}
\bmo \ell^{m+2-2\delta_2} \leq C \int_\Omega |DN|^2\, .
\end{equation}
But $\bE (T, \B_L) \leq C \bmo\, \ell^{2-2\delta_2}$ and thus \eqref{e:split_1}  follows.

\end{proof}

\subsection{Persistence of $Q$-points}
We start with the following Proposition, which is the analogous of \cite[Theorem 1.7]{DS3}, but its proof is new.

\begin{proposition}[Persistence of $Q$-points]\label{p:persistence_pre}
For every $\hat{\delta}$, there is $\bar{s}\in ]0, \frac{1}{4}[$ such that, for every $s<\bar{s}$, there exist $\hat{\eps} (s,\hat{\delta})>0$ and $\tilde{\eps}>0$ with the following property. If $T$ is a semicalibrated current as in \cite[Theorem~1.5]{DSS2}, $E := \bE(T,\bC_{4\,r} (x)) < \hat{\eps}$,
$r^2 \bOmega^2 \leq \tilde{\eps}\, E$ and $\Theta (T, (p,q)) = Q$ at some $(p,q)\in \bC_{r/2} (x)$, then the approximation
$f$ of \cite[Theorem~1.5]{DSS2} satisfies
\begin{equation}\label{e:persistence}
\int_{B_{sr} (p)} \cG (f, Q \a{\etaa\circ f})^2 \leq \hat{\delta} s^m r^{2+m} E\, .
\end{equation}
\end{proposition} 

The proof of this Proposition will be achieved combining the new height bound of Theorem \ref{t:Imp_HeightBound} with the following blow up lemma. We remark that this approach is different from the one in \cite{DS4}.

\begin{lemma}[Blow up lemma]\label{l:exdecay}
Let $T$ be a semicalibrated current such that $\bOmega^2 \leq C^\star E$ and $\Theta (T, 0) = Q$. For every $s\in (0,\sfrac14)$ there exist $\eps=\eps(s), C, \gamma>0$ and $\tilde \eps>0$ such that if $E:=\bE(T, \bC_1)<\eps$ and $\bOmega \leq \tilde \eps E^{\sfrac12}$ then
 \begin{equation}\label{e:dec}
  s^{2-\gamma}\; \bE(T, \bC_s)\leq C\;E.
 \end{equation}
\end{lemma}

\begin{proof}
 Suppose \eqref{e:dec} is not true, then for every $\gamma, C>0$  we can find some $s\in ]0,\sfrac14[$ such that there exists 
 a sequence of currents $T_k$ such that $E_k:=\bE(T_k, \bC_1)\to 0$, $\bOmega_k\leq \tilde \eps_k\,E_k$, with $\tilde \eps_k\to 0$, and
 \begin{equation}\label{e:main_contr}
  s^{2-\gamma-m} \be_{T_k}(\bC_s)= s^{2-\gamma}\; \bE(T_k, \bC_s)\geq C\;E_k=C \be_{T_k}(\bC_1)\,,
 \end{equation}
 where $\be_T(A, \pi):= \frac{1}{2}\int_A |\vec{T}-\pi|^2\,d\|T\|$ and $\be_T(A):=\min_{\pi}\be_T(A,\pi)$.
By hypothesis we can apply \cite[Theorem~1.5]{DSS2} to each $T_k$ to find Lipschitz maps $f_k \colon B_s\subset \pi_0\to \I{Q}(\pi_0^\perp)$ with the following estimates
\begin{gather}
\|f_k\|_{L^\infty(B_s)}\leq C\, \bh(T_k, B_1)+C s\,E_k^{\sfrac12}\leq C s\,E^{\sfrac12}\label{e:heighttt}\\
\frac{1}{s^m}\int_{B_s} |Df_k|^2=2 \, \omega_m\, E_k+C\,E_k^{1+\beta_0}\label{e:enegyyy}\\
\|T_k-\bG_{f_k}\|(\bC_s)\leq C\,s^m\, \,E_k^{1+\beta_0}\label{e:massessss}\,,
\end{gather}
 where in the first estimate we used the improved height bound Theorem \ref{t:Imp_HeightBound} and in all the estimates we used $\bOmega_k^2\leq E_k$.
 Next we divide the proof in two steps.
 
 \medskip

 \emph{Step 1.} We claim that, up to a subsequence, the rescaled maps $u_k:=\frac{f_k}{E_k^{\sfrac12}}$ converge weakly in $W_{loc}^{1,2}$ to a locally $\D$-minimizing function $u\in W^{1,2}(B_s, \I{Q})$ and
 \begin{equation}\label{e:en_conv}
 \lim_{k\to \infty}\int_{B_s}|Du_k|^2=\int_{B_s}|Du|^2
 \end{equation}
 Indeed, by the bounds \eqref{e:heighttt} and \eqref{e:enegyyy}, the weak convergence in $W^{1,2}_{loc}$ immediately follows (we will not relabel the subsequences). For what concerns the second half of the claim, thanks to the new height bound it is a simplified version of the arguments in \cite[Sections 3 \& 4]{DS3} with the additional errors coming from the almost minimality of $T$. Therefore, set $D_s:=\liminf_k\int_{B_s}|Du_k|^2$ and suppose by contradiction that there exists $\rho<s$ such that 
 \item[(i)]  either $\int_{B_\rho} |Du|^2<D_s$
 \item[(ii)] or $u|_{B_rho}$ is not $\D$-minimizing.\\
 Then we can find $\rho_0>0$ such that, for every $\rho>\rho_0$, there exists a function 
 $v\in W^{1,2}(B_s,\I{Q})$ such that 
 \begin{equation}\label{e:contra1}
 v|_{\de B_\rho}=u|_{\de B_\rho}\quad \text{and}\quad \gamma_\rho:=D_\rho-\int_{B_\rho}|Dv|^2>0\,.
 \end{equation}
 Consider the function $\psi_k$ given by
 \[
 \psi_k(r):=\frac{\mass((T_k-\bG_{f_k})\res \de \bC_r)}{E_k}+\int_{\de B_r}|Du_k|^2+\int_{\de B_r}|Du|^2+\frac{\int_{\de B_r}\cG(u_k,u)^2}{\int_{B_s}\cG(u_k,u)^2}
 \]
 Using \eqref{e:massessss}, we deduce that $\liminf_k \int_{\rho_0}^s\psi_k(r)\,dr<\infty$ and so by Fatou's Lemma, there is $\rho \in (\rho_0,s)$ and a subsequence such that $\lim_k\psi_k(\rho)<\infty$. It follows that
 \item[(a)] $\int_{\de B_\rho}\cG(u_k,u)^2\to 0$,
 \item[(b)] $\int_{\de B_\rho}|Du_k|^2+\int_{\de B_\rho}|Du|^2\leq M$ for some $M<\infty$,
 \item[(c)]  $\|(T_k-\bG_{f_k})\res \de \bC_\rho\|\leq C \,E_k$.\\
 Next, using \cite[Lemma 3.5]{DS3}, for every $\eta>0$ we can find a Lipschitz function $ v_\eta$ on $B_\rho$ such that
 \[
 \int_{B_\rho} \left(\cG(v,v_\eta)^2+(|Dv|-|Dv_\eta|)^2\right)+\int_{\de B_\rho}\left(\cG(v,v_\eta)^2+(|Dv|-|Dv_\eta|)^2\right)\leq \eta
 \]
By \cite[Lemma 3.6]{DS3}, for every $\theta>0$, we linearly interpolate between $v_\eta$ and $u_k$, to achieve new Lipschitz functions $\xi_k$ satisfying
 \begin{align*}
 \int_{B_\rho} |D\xi_k|^2
 &\leq \int_{B_\rho}|Dv_\eta|^2 + \theta \int_{\de B_\rho}\left(|D u_k|^2+|D v_\eta|^2 \right)+\frac{C_0}{\theta} \int_{\de B_\rho}\cG(u_k, v_\eta)^2\\
 &\stackrel{(a),(b)}{\leq} \int_{B_\rho}|Dv|^2 +C_0\, \eta+C_0\, \theta M+C_0  \,\eta\,\theta^{-1}\,.
 \end{align*}
 We first choose $\theta$ and then $\eta$ to guarantee that
 \begin{equation}\label{e:contrrr2}
\limsup_k\int_{B_\rho}|D\xi_k|^2\leq \int_{B_\rho} |Dv|^2+\frac{\gamma_\rho}{2}\,. 
 \end{equation}
 Finally set $z_k:=E_k^{\sfrac12} \xi_k$ and consider the currents $\bG_{z_l}$. Since $z_k|_{\de B_\rho}=f_k|_{\de B_\rho}$, $\de Z_k=\bG_{f_k}\res \de \bC_\rho$, and therefore from (c), $\|\de (T_k\res \bC_{\rho}-Z_k)\|\leq C\, E_k$. From the Isoperimetric Inequality there exists an $m$-dimensional Integral current $R_k$ such that 
 \[
 \de R_k=\de (T_k\res \bC_{\rho}-Z_k)\quad\text{and}\quad \|R_k\|\leq C\, E_k^{\sfrac{m}{m-1}}\,.
 \]
 Set finally $W_k:=T_k\res (\bC_s\setminus \bC_{\rho})+Z_k+R_k$ and observe that, by construction, $\de W_k=\de T_k$. Using the various estimates proved above together with the Taylor expansion for the area of a graph (cf. \cite{DS2}) we achieve
 \begin{align*}
 \limsup_k\frac{\|W_k\|-\|T_k\|}{E_k}
 &\stackrel{\eqref{e:massessss}}{\leq} \limsup_k\frac{\|W_k\|-\|\bG_{f_k}\|+CE_k^{1+\beta_0}}{E_k}\\
 &\leq \limsup_k \frac{1}{E_k}\left( \|R_k\|+\int_{B_\rho}\frac{|Dz_k|^2}{2}-\int_{B_\rho}\frac{|Df_k|^2}{2} \right)\\
&\leq \limsup_k \int_{B_\rho}\frac{|D\xi_k|^2}{2}-D_\rho \\
&\stackrel{\eqref{e:contrrr2}}{\leq} \int_{B_\rho} |Dv|^2+\frac{\gamma_\rho}{2}-D_\rho  \stackrel{\eqref{e:contra1}}{\leq} -\frac{\gamma_\rho}{2}\,.
 \end{align*} 
On the other hand, using \cite[Lemma 3.1]{DSS2} we deduce that
\[ 
\limsup_k\frac{\|W_k\|-\|T_k\|}{E_k}\geq -C \bOmega_k \,E_k^{\sfrac12}\geq -C\,\tilde{\eps}_k
\]
which, for $k$ sufficiently big gives a contradiction. 

\medskip

\emph{Step 2.} We can now conclude the proof of the Lemma. Recall that, by \cite[Theorem~1.5]{DSS2}, 
\begin{equation}\label{e:maincontr2}
 \left|\int_{B_s}|Df_k|^2- \be_{T_k}(\cC_s,\pi_0)\right|\leq C\, s^m\,E_k^{1+\beta_0}
\end{equation}
 Let $u_k,u$ be as in Step 1, and observe that \eqref{e:maincontr2} and \eqref{e:enegyyy} imply
 \[
 \int_{B_s} |Du|^2=\lim_{k\to \infty} \int_{B_s}|Du_k|^2=\lim_{k\to \infty}\frac{\be_{T_k}(\cC_s)}{E_k}\geq \frac{C}{s^{2-\gamma}}>0\,.   
 \]
In particular $Du$ is not identically $0$. Analogously, from \eqref{e:massessss}, we have
\[
\int_{B_1} |Du|^2\leq 2 \omega_m\,. 
\]
and so, using once again \eqref{e:maincontr2}, we conclude
\begin{align}
\int_{B_s}|Df_k|^2+s^m\,E_k^{1+\beta_0}
&= \be_{T_k}(\cC_s,\pi_0)\geq \frac{C}{s^{2-\gamma-m}} E_k \,2\omega_m\,\int_{B_1}|Du|^2\,. 
\end{align}
Rescaling by $E_k$ and taking the limit as $k\to \infty$, we achieve
\[
\int_{B_s}|Du|^2\geq \frac{C}{s^{2-\gamma-m}}\int_{B_1}|Du|^2\,,
\]
which contradicts the decay estimate for $\D$-minimizing functions of \cite[Theorem 3.9]{DS1}. 
\end{proof}

\begin{proof}[Proof of Proposition \ref{p:persistence_pre}]
 We can assume $r=1$. Choose $\bar s$ such that $C \bar{s}=\hat{\delta}^{\sfrac1\gamma}$. Then for every $s\in ]0,s[$ choose $\hat{\eps},\tilde \eps$
 such that Lemma \ref{l:exdecay} holds. Then we have by Theorem \ref{t:Imp_HeightBound} and the estimate on the oscillation of $f$ in \cite[Theorem 1.5]{DSS2}
 \begin{align*}
  \int_{B_s(p)}\cG(f, Q\a{\eta\circ f})^2
  &\leq C \,s^{n}\,\sup_{B_{s}(p)} |f|^2 \leq C s^{m+2} \bE(T,\bC_{4s}(p))\\
  &\leq C s^{m+\gamma} \bE(T,\bC_{\frac{1}{4}}(p))\leq \hat{\delta} s^m E
 \end{align*}

\end{proof}

The proof of Proposition \ref{p:splitting_II} is very similar to that of \cite[Proposition 3.5]{DS4}. Since there is however a difference in the choice of the parameters (as in Proposition \ref{p:splitting}), we prefer to give it.

\begin{proof}[Proof of Proposition \ref{p:splitting_II}]
We argue by contradiction. Assuming the proposition
does not hold, there is a sequence $T_k$ satisfying the Assumption~\ref{ipotesi} and radii $s_k$ for which
\begin{itemize}
\item[(a)] either $\bmo (k) := \max \{\bE (T_k, \B_{6\sqrt{m}}), \|d\omega_k\|_{C^{2,\eps_0}}^2\} \to 0$
and $1\geq \bar{s} = \lim_k s_k >0$;  or $s_k \downarrow 0$;
\item[(b)] the sets $\Lambda_k := \{\Theta (x, T_k) = Q\}\cap \B_{s_k}$ satisfy
$\cH^{m-2+\alpha}_\infty (\Lambda_k) \geq \bar{\alpha} s_k^{m-2+\alpha}$;
\item[(c)] denoting by $\sW (k)$ and $\sS (k)$ the families of 
cubes in the Whitney decompositions related
to $T_k$ with respect to $\pi_0$,
$\sup \big\{\ell (L): L\in \sW(k), L\cap B_{3s} (0, \pi_0) \neq \emptyset\big\} \leq s_k$;
\item[(d)] there exists $L_k \in \sW_e (k)$ with
$L_k \cap B_{19 s/16} (0, \pi_0) \neq \emptyset$ and $\hat\alpha s_k < \ell (L_k) \leq s_k$.
\end{itemize}
It is not difficult to see that $\bE (T_k, \B_{6\sqrt{m}s_k}) \leq C\, C_e \bmo (k) s_k^{2-2\delta_2}$,
where the constant $C$ depends only on $\beta_2, \delta_2, M_0, N_0$. 
Indeed this follows obviously if $s_k \geq c (M_0, N_0)>0$. Otherwise there is some ancestor $H'_k$ of $L_k$ with $s_k \leq \ell (H'_k) \leq C_0 s_k$
for which $\B_{6\sqrt{m}s_k}\subset \B_{H'_k}$ and $\bE(T,\B_{H'_k})\leq C_e\, \bmo \,\ell(H'_k)^{2-2\delta_2}$.

Consider now the ancestors $H_k$ and $J_k$ of $L_k$ as in the proof of Proposition~\ref{p:splitting}, and the corresponding Lipschitz approximations $f_k$. 
Consider next the radius $\rho_k := 5/4 s_k  + 2 r_{L_k}$ and observe that 
\cite[Theorem 1.5]{DSS2} can be applied to the cylinder $\hat{\bC}_k:= \bC_{5\rho_k } (0, {\pi}_{H_k})$: again as above, either $s_k \geq c (M_0, N_0)$, and the theorem
can be applied using the estimates on the height of $T$ in $\bC_{5\sqrt{m}} (0, \pi_0)$ and of its excess in $\B_{6\sqrt{m}}$, or $s_k$ is smaller and then we can use the ancestor $H'_k$ of the argument above. 
We thus have
\begin{equation}\label{e:decays}
\bE (T_k, \hat{\bC}_k, {\pi}_{H_k})
\leq C \, C_e\,\bmo (k)\, s_k^{2-2\delta_2}
\quad\text{and}\quad \bh (T_k, \hat{\bC}_k (0, {\pi}_{H_k}), {\pi}_{H_k}) \leq C\, C_h \bmo (k)^{\sfrac{1}{2m}} s_k^{1+\beta_2}.
\end{equation}
We denote by $g_k$ the ${\pi}_{H_k}$ approximation in the cylinder $\bC_k := \bC_{\rho_k} (0, {\pi}_{H_k})$.
With a slight abuse of notation, we assume that $f_k$ and $g_k$ are defined on the same plane and we also denote by
$B_k$ be the ball on which $f_k$ is defined. On $B_k$, which is contained in the domain of definition of $g_k$, the two maps $g_k$ and $f_k$ coincide outside of a set of measure at most $C\, C_e^{1+\beta_0} \bmo (k)^{1+\beta_0} s_k^{m +( 2-2\delta_2)(1 + \beta_0)}$  and their oscillation is estimated with $C \bmo^{\sfrac{1}{2m}} s_k^{1+\beta_2}$. We can therefore conclude that
\[
\int_{B_k} \cG (f_k, g_k)^2 \leq C\, C_e^{1+\beta_0} \bmo (k)^{1+\beta_0 + \sfrac{1}{2m}} s_k^{m+4 +2\beta_2-2\delta_2(1 + \beta_0)}
\]
where $C$ doe not depend on $C_e$.
Since $L_k\in \sW_e$, we have
\begin{equation}\label{e:dal basso}
E_k := \bE (T_k, \bC_k, {\pi}_{H_k})\geq c_0 \bE (T_k, \B_{L_k})\geq c_0 C_e \bmo (k) \ell (L_k)^{2-2\delta_2}
\geq c_0(\hat\alpha) \bmo (k) s_k^{2-2\delta_2}.
\end{equation}
Moreover, applying Proposition~\ref{p:splitting} and arguing as in Step 1 and Step 2 in its proof, we find a ball 
$B'_k\subset {\pi}_{H_k}$ contained in $B_{5s_k/4}$ and with radius at least $\ell (L_k)/8$ such that
\begin{equation}
\int_{B'_k} \cG (f_k, Q\a{\etaa\circ f_k})^2 \geq \bar{c}\, C_e\,\bmo (k)\, \ell (L_k)^{m+4-2\delta_2}
\geq c_1 (\hat\alpha) \,C_e\,\bmo (k) \, s_k^{m+4-2\delta_2}\, 
\end{equation}
where $c_1(\hat \alpha)$ depends only upon $M_0$, $N_0$, $\delta_2$ and $\hat \alpha$ (cf.~\eqref{e:stima dritta}). From \eqref{e:decays} we conclude that
\begin{equation}
\int_{B'_k} \cG (g_k, Q\a{\etaa\circ g_k})^2 \geq c (\hat\alpha) \, s_k^{m+2} E_k\, ,
\end{equation}
where the constant $c (\hat{\alpha})$ is positive and depends also upon $\beta_2, \delta_2, M_0, N_0$, but not on $C_e$.

Next note that by \eqref{e:dal basso}, we have that 
\[
\bOmega_k^2\,s_k^2\leq c_0(\hat{\alpha})^{-1}\,\frac{2^{-N_0\delta_2}}{C_e} E_k
\]
with $c_0(\hat{\alpha})$ depending only on $\tilde \alpha$. In particular, for any given $\eta>0$, we can choose $N_0$ big enough depending only on $\hat \alpha$ and $\eta$, so that we can apply \cite[Theorem 4.2]{DSS2}. We thus find a sequence of multivalued 
maps $u_k$ on 
$B_{5s_k/4}$ so that each $u_k$ is $\D$-minimizing and 
\begin{equation}
s_k^{-2} \int_{B_{5s_k/4}} \cG (g_k, u_k)^2 + \int_{B_{5s_k/4}} (|Dg_k| - |Du_k|)^2 = 
o (E_k) s_k^m\, .  
\end{equation}
Up to rotations (so to get ${\pi}_{H_k} = \pi_0$) and dilations (of a factor $s_k$) of the system
of coordinates,  we then end up with
a sequence of area-minimizing currents $S_k$ in $\R^{m+n}$,
functions $h_k $ and $u_k$
with the following properties:
\begin{enumerate}
\item the excess $E_k := \bE (S_k, \bC_5 (0, \pi_0))$ and the height $\bh (S_k, \bC_5 (0, \pi_0), \pi_0)$ converge to $0$;
\item $\bOmega_k^2 := \|d\omega_k\|_{C^{2,\eps_0}}^2 \leq C^\star E_k$ and hence it also converges to $0$;
\item $\Lip (h_k) \leq C E_k^{\beta_0}$;
\item $\|\mathbf{G}_{h_k} - S_k \| (\bC_{5/4} (0, \pi_0)) \leq C E_k^{1+\beta_0}$;
\item there exists a $\D$-minimizing function ${u}_k$ in $B_{5/4} (0, \pi_0)$ such that
\begin{equation}\label{e:vicinanza}
\int \left((|Dh_k| - |D{u}_k|)^2 + \cG (h_k, u_k)^2 \right) = o (E_k)\, . 
\end{equation}
\item for some positive constant $c (\hat{\alpha})$ (depending also upon $\beta_2, \delta_2, M_0, N_0$),
\begin{equation}\label{e:dal_basso_20}
\int_{B_{5/4}} \cG (h_k, \etaa \circ h_k)^2 \geq c E_k\, ;
\end{equation}
\item $\Xi_k := \{\Theta (S_k, y) = Q\} \cap \B_1$ has the property that
$\cH^{m-2+\alpha}_\infty (\Xi_k) \geq \bar\alpha>0$
and $0 \in \Xi_k$.
\end{enumerate}
Consider the projections $\bar{\Xi}_k := \p_{\pi_0} (\Xi_k)$. 
We are therefore in the position of applying Proposition \ref{p:persistence_pre} to conclude that, for every $\varpi >0$ there is a $\bar{s} (\varpi)>0$
(which depends also upon the various parameters $\alpha, \bar{\alpha}, \hat{\alpha}, \beta_2, \delta_2, M_0, N_0, C_e$ and $C_h$) such that
\begin{equation}\label{e:Qpunto}
\limsup_{k\to \infty} \max_{x\in \bar{\Xi}_k} \mint_{B_\rho (x)} \cG (h_k, Q\a{\etaa\circ h_k})^2 \leq \varpi E_k \qquad \forall \rho < \bar{s} (\varpi)\, .
\end{equation}
Up to subsequences we can assume that $\bar{\Xi}_k$ (and hence also $\Xi_k$) converges,
in the Hausdorff sense, to 
a compact set $\Xi$, which is nonempty. 
Moreover, by the compactness of Dir-minimizers, the maps 
$x\mapsto \hat{u}_k (x) = E_k^{-\sfrac{1}{2}} \sum_i \a{{u}_k)_i (x) - \etaa\circ u_k (x)}$
converge, strongly in $L^2 (B_{5/4})$,
to a $\D$-minimizing function $u$ with $\etaa\circ u =0$. 
Thus \eqref{e:vicinanza} and \eqref{e:dal_basso_20} easily imply that  
\begin{equation}\label{e:dal_basso_30}
\liminf_k \int_{B_{5/4}} \cG (\hat{u}_k, Q\a{0})^2 \geq \liminf E_k^{-1} \int_{B_{5/4}} \cG ({u}_k, \etaa \circ {u}_k)^2 \geq c > 0\, .
\end{equation}
From the strong $L^2$ convergence of $\hat{u}_k$ we then conclude that $u$ does not vanish identically.
On the other hand, by \eqref{e:Qpunto}, \eqref{e:vicinanza} and the strong convergence of $\hat{u}_k$ we conclude that, for any given $\delta>0$ there is a $\bar{s}>0$ such that
\[
\mint_{B_\rho (x)} \cG (u, Q\a{0})^2 \leq \varpi \qquad \forall x\in \Xi \quad \mbox{and}\quad \forall \rho<\bar{s} (\varpi)\, .
\]
Since $u$ is $\D$-minimizing and hence continuous, the arbitrariness of $\varpi$ implies $u \equiv Q \a{0}$ on $\Xi$. On the other hand,
$\cH^{m-2+\alpha}_\infty (\Xi) \geq \limsup_k \cH^{m-2+\alpha}_\infty (\Xi_k) \geq \bar\alpha >0$.
Then, by \cite[Theorem 0.11]{DS1} and the unique continuation for $\D$-minimizers \cite[Lemma 7.1]{DS4} we conclude $\bar\Xi=B_{5/4}$, which contradicts $u\not\equiv 0$.

\end{proof}

For what concerns the proof of Proposition \ref{p:persistence}, we observe that it follows from Proposition \ref{p:persistence_pre} combined with the estimates on $N$ of Theorem \ref{t:approx}, exactly as in \cite[Proposition 3.6]{DS4}. There is however one subtle point, that is the validity of the assumption $r^2 \bOmega^2 \leq \tilde \eps E$, which is guaranteed on a cube of type $W_e$ by our choice of $N_0$ as in \eqref{e:problematica}.

\subsection{Comparison between different center manifolds}
The proof of Proposition \ref{p:compara} is analogous to that of \cite[Proposition 3.7]{DS4}, since, except for the various estimates of the previous sections, the only property of the current needed here is the Height bound of Theorem \ref{t:height_bound} combined with the stopping condition on cubes ot type $\sW_h$.

\part{Proof of Theorem \ref{t:finale}: the contradiction argument} 

\section{The contradiction sequence}

We now come to the final blow up argument. In this section we will follow very closely \cite{DS5}, pointing out where the main differences are and just referring to results there whenever possible. For a nice outline of the argument, see the introduction to \cite{DS5} and \cite{DLs}.

In order to prove Theorem \ref{t:finale}, we will proceed by contradiction; let us therefore assume the following 

\begin{ipotesi}[Contradiction]\label{assurdo}
There exist $m \geq 2$ and $T$  a compactly supported semicalibrated current such that 
$\cH^{m-2+\alpha}(\sing(T)) >0$ for some $\alpha>0$. 
\end{ipotesi}

\subsection{A good contradiction sequence}
The first step in our contradiction argument is to prove the existence of a singular point and a blow-up sequence of the current around that point which converges to a flat $m$-plane with multiplicity $Q$ and with a uniform bound from below on the $\cH^{m-2+\alpha}_\infty$-measure of the singular set.

\begin{definition}[$Q$-points] \label{d:regular_points}
For $Q\in \N$, we denote by $\rD_Q(T)$ the points of density $Q$ of the current $T$,
and set
\begin{gather*}
\reg_Q (T) := \reg (T)\cap \rD_Q (T) \quad\text{and}\quad
\sing_Q (T) := \sing (T)\cap \rD_Q (T).
\end{gather*}
\end{definition}

For any $r>0$ and $x\in \R^{m+n}$, the blow-up map $\iota_{x,r}: \R^{m+n}\to \R^{m+n}$ is the map
$y\mapsto \frac{y-x}{r}$ and
$T_{x, r} := (\iota_{x, r})_{\sharp} T$.

\begin{proposition}[Contradiction sequence {\cite[Proposition~1.3]{DS5}}]\label{p:seq}
Under Assumption \ref{assurdo}, there are $m,n, Q \geq 2$ and $T$,
reals $\alpha,\eta>0$, and a sequence $r_k\downarrow 0$ such that
$0\in \rD_Q (T)$ and the following holds;
\begin{gather}
\lim_{k\to+\infty}\bE(T_{0,r_k}, \B_{6\sqrt{m}}) = 0,\label{e:seq1}\\
\lim_{k\to+\infty} \cH^{m-2+\alpha}_\infty (\rD_Q (T_{0,r_k}) \cap \B_1) > \eta,\label{e:seq2}\\
\cH^m \big((\B_1\cap \supp (T_{0, r_k}))\setminus \rD_Q (T_{0,r_k})\big) > 0\quad \forall\; k\in\N.\label{e:seq3}
\end{gather}
\end{proposition}

The proof of this result is analogous to the one of \cite[Proposition~1.3]{DS5}. Indeed, by \eqref{e:var_prima_(b)} in Lemma \ref{l:quasi_minimalita}, we know that $T$ is a varifold with bounded generalized mean curvature and therefore the following conclusions hold.
\begin{itemize}
\item[(i)] $e^{\bOmega}\Theta (x, T)$ is nondecreasing, so that the density $\Theta (x, T)$ exists at every point and is upper semicontinuous (cf \cite[Section 17]{Sim}).
\item[(ii)] Allard's regularity theorem applies (cf. \cite[Chapter 5]{Sim} and \cite[Theorem A.1]{DS5}).
\item[(iii)] For every $x\in\supp(T)\setminus \supp(\de T)$, consider the sequence of blow-ups $T_{x,r}$. By the almost monotonicity of the density, there exists a subsequence $(T_{x,r_k})_k$ which converges to $S$. Moreover, it is a standard contradiction argument using \eqref{e:Omega} in Lemma \ref{l:quasi_minimalita}, to prove that $S$ is locally area minimizing in $\R^{m+n}$, $\de S=0$ and the convergence is strong. It follows that $S$ is a cone (i.e., $S_{0,r} = S$ for all $r>0$ and $\partial S=0$), called {\em a tangent cone to $T$ at $x$}.
\item[(iv)] Thanks to (iii), Almgren's stratification holds for $T$ (cf. \cite[Theorem A.3]{DS5}).
\end{itemize}
Since (i)-(iv) are the only properties of $T$ used in the proof of \cite[Proposition~1.3]{DS5}, the same result follows also in our case.

\subsection{Intervals of flattening}\label{ss:flattening}

For the sequel we fix the constant $c_s := \frac{1}{64\sqrt{m}}$ and notice that
$2^{-N_0} < c_s$,
where $N_0$ is the parameter introduced in Assumption \ref{i:parametri}.
It is always understood that
the parameters $\beta_2, \delta_2, \gamma_2, \eps_2, \kappa, C_e, C_h, M_0, N_0$ are fixed in such a way that all the theorems and propositions
of Section \ref{s:CM} are applicable. In particular, all constants which will depend upon
these parameters will be called {\em geometric} and denoted by $C_0$. On the contrary, we will highlight
the dependence of the constants upon the parameters introduced in this part $p_1, p_2, \ldots$ by
writing $C= C(p_1, p_2, \ldots)$. 

By Proposition~\ref{p:seq} and  simple rescaling arguments,
we assume in the sequel the following.

\begin{ipotesi}\label{i:H'} Let $\eps_3\in ]0, \eps_2[$. Under Assumption~\ref{assurdo}, there exist  $m, n, Q \geq 2$, $\alpha, \eta > 0$ and $T$ for which:
\begin{itemize}
\item[(a)] there is a sequence of radii 
$r_k\downarrow 0$ as in Proposition \ref{p:seq};
\item[(b)] the following holds:
\begin{gather}
\supp (\partial T)\cap \B_{6\sqrt{m}} = \emptyset,\quad 0 \in D_Q(T),\label{e:aggiuntive}\\
\|T\| (\B_{6\sqrt{m} r}) \leq
r^m \left( Q\,\omega_m (6\sqrt{m})^m + \eps_3^2\right) \quad \text{for all } r\in (0,1),\label{e:(2.1)-del-cm}\\
 \|d\omega\|_{C^{1,\eps_0}(\B_{7\sqrt{m}})} \leq \eps_3\, .\label{e:piattezza_sigma}
\end{gather}
\end{itemize}
\end{ipotesi}
 
 We set
\begin{equation}\label{e:def_R}
\mathcal{R}:=\big\{r\in ]0,1]: \bE (T, \B_{6\sqrt{m} r}) \leq \eps_3^2\big\}\,.
\end{equation}
Observe that, if $\{s_k\}\subset \mathcal{R}$
and $s_k\uparrow s$, then $s\in \mathcal{R}$.
We cover $\cR$ with a collection $\mathcal{F}=\{I_j\}_j$ of intervals
$I_j = ]s_j, t_j]$ defined as follows.
$t_0:= \max \{t: t\in \mathcal{R}\}$. Next assume, by induction, to have defined $t_j$ (and hence also
$t_0 > s_0\geq t_1 > s_1 \geq \ldots > s_{j-1}\geq t_j$) and consider the following objects:
\begin{itemize}
\item[-] $T_j := ((\iota_{0,t_{j}})_\sharp T)\res \B_{6\sqrt{m}}$; moreover, consider
for each $j$ an orthonormal system of coordinates so that, if we denote by $\pi_0$ the $m$-plane $\mathbb R^m\times \{0\}$, then $\bE (T_j, \B_{6\sqrt{m}}, \pi_0) = \bE(T_j,\B_{6\sqrt{m}})$ (alternatively we can keep the system of coordinates fixed
and rotate the currents $T_j$).
\item[-] Let $\cM_j$ be the corresponding center manifold constructed
in Theorem~\ref{t:cm} applied to $T_j$ with respect to the $m$-plane $\pi_0$;
the manifold $\cM_j$ is then the graph of a map $\phii_j: \pi_0 \supset [-4,4]^m \to \pi_0^\perp$,
and we set $\Phii_j (x) := (x, \phii_j (x)) \in \pi_0\times \pi_0^\perp$.
\end{itemize}
Then, we consider the Whitney decomposition $\sW^{(j)}$ of $[-4,4]^m \subset \pi_0$ as in Definition \ref{d:refining_procedure} and Proposition \ref{p:whitney} (applied to $T_j$) and we define
\begin{equation}\label{e:s_j}
s_j := \max\, \left(\{c_s^{-1} \ell (L) : L\in \sW^{(j)} \mbox{ and } c_s^{-1} \ell (L) \geq \dist (0, L)\} \cup \{0\} \right)\, .
\end{equation} 
We will prove below that $s_j/t_j <2^{-5}$. In particular this ensures that $[s_j, t_j]$ is a (nontrivial) interval. 
Next, if $s_j =0$ we stop the induction. Otherwise we let $t_{j+1}$ be the largest element
in $\mathcal{R}\cap ]0, s_j]$ and proceed as above. Note moreover the following simple consequence of \eqref{e:s_j}:
\begin{itemize}
\item[(Stop)] If $s_j >0$ and $\bar{r} := s_j/t_j$, then there is $L\in \sW^{(j)}$ with 
\begin{equation}\label{e:st}
\ell(L) = c_s\,\bar r\, \qquad \mbox{and} \qquad L\cap \bar B_{\bar r} (0, \pi_0)\neq \emptyset\, 
\end{equation}
(in what follows $B_r (p, \pi)$ and $\bar B_r (p,\pi)$ will denote the open and closed disks 
$\B_r (p)\cap (p+\pi)$, $\bar{\B}_r (p) \cap (p+\pi)$);
\item[(Go)] If $\rho > \bar{r} := s_j/t_j$, then
\begin{equation}\label{e:go}
\ell (L) < c_s \rho \qquad \mbox{for all } L\in \sW^{j(k)} \mbox{ with $L\cap B_\rho (0, \pi_0) \neq \emptyset$.}
\end{equation}
In particular the latter inequality is true for every $\rho\in ]0,3]$ if $s_j =0$.
\end{itemize}

The following is an easy consequence of this construction whose proof can be found in \cite{DS5}.

\begin{proposition}[Intervals of flattening {\cite[Proposition~2.2]{DS5}}]\label{p:flattening}
Assuming $\eps_3$ sufficiently small, then the following holds:
\begin{itemize}
\item[(i)] $s_j < \frac{t_j}{2^5}$ and the family $\mathcal{F}$ is either countable
and $t_j\downarrow 0$, or finite and $I_j = ]0, t_j]$ for the largest $j$;
\item[(ii)] the union of the intervals of $\cF$ cover $\cR$,
and for $k$ large enough the radii $r_k$ in Assumption \ref{i:H'} belong to $\cR$;
\item[(iii)] if $r \in ]\frac{s_j}{t_j},3[$ and $J\in \mathscr{W}^{(j)}_n$ intersects
$B:= \p_{\pi_0} (\cB_r (p_j))$, with $p_j := \Phii_j(0)$ and $\cB_r(p_j)$ the geodesic ball of $\cM$ of radius $r$ and centered in $p_j$,
then $J$ is in the
domain of influence $\mathscr{W}_n^{(j)} (H)$ (see Definition~\ref{d:domains})
of a cube $H\in \mathscr{W}^{(j)}_e$ with
\[
\ell (H)\leq 3 \, c_s\, r \quad \text{and}\quad 
\max\left\{{\rm sep}\, (H, B),  {\rm sep}\, (H, J)\right\}
\leq 3\sqrt{m}\, \ell (H) \leq \frac{3 r}{16};
\]
\item[(iv)] $\bE (T_j, \B_r )\leq C_0 \eps_3^2 \, r^{2-2\delta_2}$ for
every $r\in]\frac{s_j}{t_j},3[$.
\item[(v)] $\sup \{ \dist (x,\cM_j): x\in \supp(T_j) \cap \p^{-1}_j(\cB_r(p_j))\} \leq C_0\, (\bmo^j)^\frac{1}{2m} r^{1+\beta_2}$ for
every $r\in]\frac{s_j}{t_j},3[$, where
$\bmo^j := \max\{ \|\omega\|_{C^{2,\eps_0}}^2 , \bE(T_j, \B_{6\sqrt{m}})\}$. 
\end{itemize}
\end{proposition}

\section{Frequency function and first variations}\label{s:frequency}

Consider the following Lipschitz (piecewise linear) function $\phi:[0+\infty[ \to [0,1]$ given by
\begin{equation*}
\phi (r) :=
\begin{cases}
1 & \text{for }\, r\in [0,\textstyle{\frac{1}{2}}],\\
2-2r & \text{for }\, r\in \,\, ]\textstyle{\frac{1}{2}},1],\\
0 & \text{for }\, r\in \,\, ]1,+\infty[.
\end{cases}
\end{equation*}
For every interval of flattening $I_j = ]s_j, t_j]$,
let $N_j$ be the normal approximation of $T_j$
on $\cM_j$ in Theorem~\ref{t:approx}.

\begin{definition}[Frequency functions]\label{d:frequency}
For every $r\in ]0,3]$ we define: 
\[
\bD_j (r) := \int_{\cM^j} \phi\left(\frac{d_j(p)}{r}
\right)\,|D N_j|^2(p)\, dp\quad\mbox{and}\quad
\bH_j (r) := - \int_{\cM^j} \phi'\left(\frac{d_j (p)}{r}\right)\,\frac{|N_j|^2(p)}{d(p)}\, dp\, ,
\]
where $d_j (p)$ is the geodesic distance on $\cM_j$ between $p$ and $\Phii_j (0)$.
Furthermore, let $\xi^j_1(p),\dots,\xi^j_m(p)$ be an orthonormal frame for $T_p\cM^j$, and consider 
\[
\bGam_j(r):=\bD_j(r)+\bL_j(r):=\bD_j(r)+(-1)^{l-1}\sum_{i=1}^Q\sum_{l=1}^m \int_{\cM^j}\phi\left(\frac{d_j(p)}{r}
\right)\, \langle D_{\xi_l}N_j^i\wedge\hat \xi_l\wedge N_j^i,d\omega\rangle(p)\, dp
\] 
where $\hat \xi_l:=\xi_1\wedge\dots\wedge \xi_{l-1}\wedge\xi_{l+1}\wedge\dots\wedge \xi_m$. Then if $\bH_j (r) > 0$, we define the {\em frequency function}
$\bI_j (r) :=\frac{r\,\bGam_j(r)}{\bH_j(r)}$.
\end{definition}

The following is the main analytical estimate of the paper, which allows us to
exclude infinite order of contact among the different sheets of a minimizing
current.

\begin{theorem}[Main frequency estimate]\label{t:frequency}
If $\eps_3$ is sufficiently small, then
there exists a geometric constant $C_0$ such that, for every
$[a,b]\subset [\frac{s_j}{t_j}, 3]$ with $\bH_j \vert_{[a,b]} >0$, we have
\begin{equation}\label{e:frequency}
\bI_j (a) \leq C_0 (1 + \bI_j (b)).
\end{equation}
\end{theorem}

To simplify the notation, in this section we drop the index $j$ and 
omit the measure $\cH^m$ in the integrals over regions of
$\cM$.
The proof exploits four identities collected in Proposition \ref{p:variation},
which will be proved in the next subsection.

\begin{definition}\label{d:funzioni_ausiliarie}
We let $\de_{\hat r}$ denote the derivative with respect to arclength along geodesics starting at $\Phii(0)$. We set
\begin{align}
&\qquad\qquad\qquad\bE (r) := - \int_\cM \phi'\left(\textstyle{\frac{d(p)}{r}}\right)\,\sum_{i=1}^Q \langle
N_i(p), \de_{\hat r} N_i (p)\rangle\, dp\,  ,\\
&\bG (r) := - \int_{\cM} \phi'\left(\textstyle{\frac{d(p)}{r}}\right)\,d(p) \left|\de_{\hat r} N (p)\right|^2\, dp
\quad\mbox{and}\quad
\bSigma (r) :=\int_\cM \phi\left(\textstyle{\frac{d(p)}{r}}\right)\, |N|^2(p)\, dp\, .
\end{align}
\end{definition}

\begin{remark}\label{r:tough}Observe that all these functions of $r$ are absolutely continuous
and, therefore, classically differentiable at almost every $r$.
Moreover, the following rough estimate easily follows from
Theorem \ref{t:approx} and the condition (Go):
\begin{align}\label{e:rough}
\bD(r) \leq C_0\, \bmo\, r^{m+2-2\delta_2} \quad \mbox{for every}\quad r\in\left]\textstyle{\frac{s}{t}},3\right[.
\end{align}
Indeed, since $N$ vanishes identically on the set $\mathcal{K}$ of Theorem \ref{t:approx}, it suffices to sum the estimate of \eqref{e:Dir_regional} over all the different cubes $L$ (of the corresponding Whitney decomposition)
for which $\Phii (L)$ intersects the geodesic ball $\cB_r$. 
\end{remark}

\begin{proposition}[First variation estimates]\label{p:variation}
For every $\gamma_3$ sufficiently small there is a constant $C = C (\gamma_3)>0$ such that, if
$\eps_3$ is sufficiently small, $[a,b]\subset [\frac{s}{t}, 3]$ and $\bI \geq 1$ on $[a,b]$, then the following
inequalities hold for a.e. $r\in [a,b]$:
\begin{gather}
\left|\bH' (r) - \textstyle{\frac{m-1}{r}}\, \bH (r) - \textstyle{\frac{2}{r}}\,\bE(r)\right|\leq  C \bH (r), \label{e:H'}\\
\left|\bGam (r)  - r^{-1} \bE (r)\right| \leq C \bD (r)^{1+\gamma_3} + C \eps_3^2 \,\bSigma (r),\label{e:out}\\
\left| \bD'(r) - \textstyle{\frac{m-2}{r}}\, \bD(r) - \textstyle{\frac{2}{r^2}}\,\bG (r)\right|\leq
C \bD (r) + C \bD (r)^{\gamma_3} \bD' (r) + C r^{-1}\bD(r)^{1+\gamma_3},\label{e:in}\\
\bSigma (r) +r\,\bSigma'(r) \leq C  \, r^2\, \bD (r)\, \leq C r^{2+m} \eps_3^{2},\label{e:Sigma1}\\
|\bL(r)|\leq C\,\bmo^{\sfrac12}r\,\bD(r)\quad \text{and}\quad |\bL'(r)|\leq C\,\bmo^{\sfrac12} (r^{-1}\,\bD'(r)\,\bH(r))^{\sfrac12}\,.\label{e:LeL'} 
\end{gather}
\end{proposition}

We assume for the moment the proposition and prove the theorem.

\begin{proof}[Proof of Theorem \ref{t:frequency}.]
Set $\bOmega(r) := \log \big(\max \{\bI(r), 1\}\big)$. Fix a $\gamma_3>0$ and an $\eps_3$ sufficiently small
so that the conclusions
of Proposition \ref{p:variation} hold. We can thus treat the corresponding constants in the inequalities as geometric ones,
but to simplify the notation we keep denoting them by $C$.

To prove \eqref{e:frequency} it is enough to show $\bOmega (a) \leq C + \bOmega (b)$.
If $\bOmega (a) = 0$, then there is nothing to prove. If $\bOmega (a)>0$, let $b'\in ]a,b]$
be the supremum of $t$ such that $\bOmega>0$ on $]a,t[$.
If $b'<b$, then $\bOmega (b')=0 \leq \bOmega (b)$. Therefore, by possibly
substituting $]a,b[$ with $]a,b'[$, we can assume that $\bOmega >0$, i.e.~$\bI >1$, on $]a,b[$.
By Proposition~\ref{p:variation}, if $\eps_3$ is sufficiently small, then 
\begin{equation}\label{e:out2}
\frac{\bD(r)}{4}
\stackrel{\eqref{e:LeL'}}{\leq} \frac{\bGam(r)}{2} 
\stackrel{\eqref{e:out}\,\&\,\eqref{e:Sigma1}}{\leq}
\frac{\bE (r)}{r} 
\stackrel{\eqref{e:out}\,\&\,\eqref{e:Sigma1}}{\leq} 
2\, \bGam (r)
\stackrel{\eqref{e:LeL'}}{\leq} 4\,\bD(r) ,
\end{equation}
from which we conclude that both $\bGam,\bE>0$ over the interval $]a, b'[$.
Set for simplicity $\bF(r) := \bGam(r)^{-1} - r \bE(r)^{-1}$, and compute
\[
- \bOmega'  (r) = \frac{\bH' (r)}{\bH(r)} - \frac{\bGam'(r)}{\bGam(r)} - \frac{1}{r} 
\stackrel{\eqref{e:H'}}{\leq} \left(\frac{m-1}{r}+C\right) +\frac{2\bE(r)}{r\,\bH(r)}- \frac{\bGam'(r)}{\bGam(r)}.
\]
Then we have two possibilities: either the RHS of this expression is nonpositive, so that $-\bOmega'(r)\leq 0$, or
\begin{equation}\label{e:favone1}
\left(\frac{m-1}{r}+C\right) +\frac{2\bE(r)}{r\bH(r)}- \frac{\bGam'(r)}{\bGam(r)}\geq 0
\end{equation}
Let us assume to be in this second case, then by \eqref{e:favone1}
\[
\frac{\bD'(r)}{\bGam(r)}\leq \frac{|\bL'(r)|}{\bGam(r)}+\left(\frac{m-1}{r}+C\right) +\frac{2\bE(r)}{r\,\bH(r)}
\]
that is, since $\bH,\bGam>0$ and $\bH(r)\leq r\,\bGam(r)$, by $\bI(r)\geq 1$, 
\begin{align}
\bD'(r)\,\bH(r)
&\leq \frac{2}{r}\,\bE(r)\,\bGam(r)+|\bL'(r)|\,\bH(r)+ \left(\frac{m-1}{r}+C\right)\bGam(r)\,\bH(r) \notag\\
&\leq  \frac{2}{r}\,\bE(r)\,\bGam(r)+C \, \bmo^{\sfrac12} (r^{-1} \bD'(r)\, \bH(r))^{\sfrac12} \bH(r) +C\, \bGam^2(r) \notag\\
&\stackrel{\eqref{e:out2}}{\leq}\frac{1}{2} \bD'(r)\,\bH(r)+ C\, \bGam^2(r)\label{e:favone2}
\end{align}
which, combined with \eqref{e:LeL'}, implies $|\bL'(r)|\leq C\,\bGam(r)/r^{\sfrac12}$.
Next, set for simplicity $\bF(r) := \bGam(r)^{-1} - r \bE(r)^{-1}$, and compute
\[
- \bOmega'  (r) = \frac{\bH' (r)}{\bH(r)} - \frac{\bGam'(r)}{\bGam(r)} - \frac{1}{r} 
\stackrel{\eqref{e:out}}{=} \frac{\bH'(r)}{\bH(r)} - \frac{r\bGam'(r)}{\bE(r)} - \bGam'(r) \bF (r) - \frac{1}{r}.
\]
Again by Proposition~\ref{p:variation} and using $|\bL'(r)|\leq C \bGam(r)/r^{\sfrac12}\leq C\,\bD(r)/r^{\sfrac12}$ in \eqref{e:pezzo_20} and in \eqref{e:pezzo_40},
\begin{equation}\label{e:pezzo_10}
\frac{\bH'(r)}{\bH(r)} \stackrel{\eqref{e:H'}}{\leq} \frac{m-1}{r} + C + \frac2r\,\frac{\bE(r)}{\bH(r)},
\end{equation}
\begin{equation}\label{e:E denominatore}
|\bF(r)| \stackrel{\eqref{e:out}}{\leq} C \, \frac{r (\bD (r)^{1+\gamma_3} + \bSigma (r))}{\bGam (r)\,\bE (r)} \stackrel{\eqref{e:out2}}{\leq} C \, \bD (r)^{\gamma_3-1}+ C\,\frac{\bSigma (r)}{\bD (r)^2},
\end{equation}
\begin{align}\label{e:pezzo_40}
- \bGam'(r) \bF (r)
&\leq C \bD(r)^{\gamma_3-1} \bD' (r)+C \frac{\bSigma (r) \bD' (r)}{\bD(r)^2}+C\bD(r)^{\gamma_3}+C\frac{\bSigma(r)}{\bD(r)}
\end{align}
\begin{align}\label{e:pezzo_20}
- \frac{r\bGam'(r)}{\bE(r)} 
&\stackrel{\eqref{e:in}}{\leq} \left(C - \frac{m-2}{r}\right) \frac{r\bD(r)}{\bE(r)} - \frac2r \, \frac{\bG(r)}{\bE(r)}+  C\frac{r \bD(r)^{\gamma_3} \bD' (r) + \bD (r)^{1+\gamma_3}}{\bE(r)}+\frac{r|\bL'(r)|}{\bE(r)}\nonumber\\
&\leq \frac{C}{r^{\sfrac12}}- \frac{m-2}{r} 
+ \frac{C}{r} \bGam(r) |\bF(r)|
- \frac2r \, \frac{\bG(r)}{\bE(r)}
+  C \bD(r)^{\gamma_3-1} \bD' (r) + C\frac{\bD (r)^{\gamma_3}}{r} +\frac{r\,\bL(r)}{\bE(r)}\notag\\
& \stackrel{\eqref{e:Sigma1},\, \eqref{e:E denominatore}, \,\eqref{e:rough}\&\eqref{e:out2}}{\leq}C - \frac{m-2}{r} 
- \frac2r \, \frac{\bG(r)}{\bE(r)}
+  C \bD(r)^{\gamma_3-1} \bD' (r) + C\,r^{\gamma_3\,m -1}.
\end{align}
Furthermore, since 
by Cauchy-Schwartz, we have
\begin{equation}\label{e:cauchy-schwartz}
\frac{\bE(r)}{r\bH(r)}\leq \frac{\bG(r)}{r \bE(r)}.
\end{equation}
Thus, by \eqref{e:rough}, \eqref{e:pezzo_10}, \eqref{e:pezzo_40},
\eqref{e:pezzo_20} and \eqref{e:cauchy-schwartz}, we conclude
\begin{align}
- \bOmega' (r) &\leq C +C\,r^{\gamma_3\,m -1}+
C r \bD(r)^{\gamma_3-1} \bD' (r) 
- \bD'(r) \bF(r)\notag\\
& \stackrel{\eqref{e:E denominatore}}{\leq}C\,r^{\gamma_3\,m -1}+
C\bD(r)^{\gamma_3-1} \bD'(r) + C \frac{\bSigma (r) \bD' (r)}{\bD(r)^2}.\label{e:grosso} 
\end{align}
Integrating \eqref{e:grosso} we conclude:
\begin{align}
\bOmega (a) - \bOmega (b) &\leq C + C \left(\bD(b)^{\gamma_3} - \bD(a)^{\gamma_3}\right) 
+ C \left[ \frac{\bSigma (a)}{\bD (a)} - \frac{\bSigma (b)}{\bD (b)} + \int_a^b \frac{\bSigma' (r)}{\bD (r)}\, dr\right]
\stackrel{\eqref{e:Sigma1}}{\leq} C.\qquad\qedhere\notag
\end{align}
\end{proof}

\subsection{Proof of Proposition \ref{p:variation}}
The proof of \eqref{e:H'} can be found in \cite[Subsection 3.1]{DS5}, while for \eqref{e:LeL'} we observe that, since $\bH(r)\leq C\,r \bGam(r)\leq C\,r \bD(r)$ by assumption,
\begin{align}\label{e:L}
|\bL(r)| & \leq C\,\bmo^{\sfrac12}\,\int_{\cM} \phi\left(\frac{d_j(p)}{r}
\right) |N|\,|DN| \leq  C_{}\,\bmo^{\sfrac12}\,\left(\int_0^r \bH(t)\,dt\right)^{\frac12} 
\bD^{\frac12}(r)\notag \\
&\leq C\,\bmo^{\sfrac12}\,\left(C\int_0^r 
t\,\bD(t)\,dt\right)^{\frac12} \bD^{\frac12}(r)
\leq C\,\bmo^{\sfrac12}\,r\,\bD(r)\, ,
\end{align}
and similarly
\begin{align}\label{e:L'}
|\bL'(r)| & \leq -C\,\bmo^{\sfrac12}\,\int_{\cM}\frac{1}{r^2} \phi'\left(\frac{d_j(p)}{r}
\right)|N|\,|DN| \leq 
C\,\bmo^{\sfrac12}\,\left(r^{-1}\,\bD'(r)\,\bH(r)\right)^{\frac12}\,.
\end{align}
For what concerns \eqref{e:Sigma1}, we recall the following result

\begin{lemma}[{\cite[Lemma~3.6]{DS5}}]\label{l:poincare'} There exists a dimensional constant $C_0>0$ such that
\begin{equation}\label{e:poincare'}
\bSigma (r) \leq C_0\, r^2\,\bD (r) + C_0 r \bH (r) \quad \text{and} \quad 
\bSigma' (r) \leq C_0 \bH (r),
\end{equation}
\begin{equation}\label{e:L2_pieno}
\int_{\mathcal{B}_r (q)} |N|^2 \leq C_0\,\bSigma (r) + C_0\,r\,\bH(r)\, ,
\end{equation}
\begin{equation}\label{e:Dirichlet_pieno}
\int_{\mathcal{B}_r (q)} |DN|^2 \leq C_0\,\bD (r) + C_0\,r \bD' (r).
\end{equation}
In particular, if $\bI \geq 1$, then \eqref{e:Sigma1} holds and
\begin{equation}\label{e:L2_pieno2}
\int_{\mathcal{B}_r (q)} |N|^2 \leq C_0\,r^2\bD (r).
\end{equation}
\end{lemma}

Finally we come to \eqref{e:out} and \eqref{e:in}. Their proof is very similar to the one in \cite{DS5}, so we give only an outline of it \cite[Section 4]{DS5}. 
We exploit the first variation of $T$ along
some vector fields $X$. The variations are denoted by $\delta T (X)$.
We fix
a neighborhood $\bU$ of $\cM$ and the normal
projection $\p:\bU\to \cM$ as in Assumption~\ref{intorno_proiezione}. Observe that
$\p\in C^{2,\kappa}$ and \cite[Assumption 3.1]{DS2} holds.
We will consider:
\begin{itemize}
\item the {\em outer variations}, where $X (p)= X_o (p) := \phi \left(\frac{d(\p(p))}{r}\right) \, (p - \p(p) )$.
\item the {\em inner variations}, where $X (p) = X_i (p):= Y(\p(p))$ with
\[
Y(p) := \frac{d(p)}{r}\,\phi\left(\frac{d(p)}{r}\right)\, \frac{\de}{\de \hat r} \quad
\forall \; p \in \cM
\]
($\frac{\partial}{\partial \hat{r}}$ is the unit vector field tangent to the geodesics emanating from $\Phii (0)$
and pointing outwards).
\end{itemize}
Note that $X_i$ is the infinitesimal generator of
a one parameter family of bilipschitz homeomorphisms $\Phi_\eps$ defined as
$\Phi_\eps (p):= \Psi_\eps (\p (p)) + p - \p (p)$, where 
$\Psi_\eps$ is the one-parameter family of bilipschitz homeomorphisms of $\cM$ generated by $Y$.

Consider now the map $F(p) := \sum_i \a{p+ N_i (p)}$ and the current $\bT_F$
associated to its image (cf.~\cite{DS2} for the notation). Observe that $X_i$ and $X_o$ are supported in $\p^{-1} (\cB_r (q))$ but none of them is {\em compactly} supported. 
However, recalling Proposition \ref{p:flattening} (v) and 
the almost minimizing property of $T$ in Lemma \ref{l:quasi_minimalita}, we deduce that
\[
\delta T(X) = T(d\omega\ser X)=(T-\bT_F)(d\omega\ser X)+\bT_F(d\omega\ser X)\,.
\]
Then, we have 
\begin{align}\label{e:Err4-5}
&|\delta \bT_F (X)-\bT_F(d\omega\ser X)| \leq |\delta \bT_F (X) - \delta T (X)| + |(T-\bT_F)(d\omega\ser X)|\nonumber\\
\leq& \underbrace{\int_{\supp (T)\setminus \im (F)} (|X|+\left|\dv_{\vec T} X\right|)\, d\|T\|
+ \int_{\im (F)\setminus \supp (T)}(|X|+\left|\dv_{\vec \bT_F} X\right|)\, d\|\bT_F\|}_{{\rm Err}_4}.
\end{align}
Set now for simplicity $\varphi_r (p) :=  \phi \big(\frac{d(p)}{r}\big)$.
We wish to apply \cite[Theorem 4.2]{DS2} to conclude
\begin{align}\label{e:ov graph}
\delta \bT_F (X_o) =& \int_\cM\Big( \ph_r\,|D N|^2 + \sum_{i=1}^Q N_i\otimes \nabla \ph_r : D N_i \Big) + \sum_{j=1}^3 \textup{Err}^o_j,
\end{align}
where the errors ${\rm Err}^o_j$ correspond to the terms ${\rm Err_j}$ of \cite[Theorem 4.2]{DS2}. This
would imply
\begin{gather}
{\rm Err}_1^o = - Q \int_\cM \varphi_r \langle H_\cM, \etaa\circ N\rangle,\label{e:outer_resto_1}\\
|{\rm Err}_2^o| \leq C_0 \int_\cM |\varphi_r| |A|^2|N|^2,\label{e:outer_resto_2}\\
|{\rm Err}_3^o| \leq C_0 \int_\cM \big(|N| |A| + |DN|^2 \big) \big( |\varphi_r| |DN|^2  + |D\varphi_r| |DN| |N| \big)\label{e:outer_resto_3}\,,
\end{gather}
where $H_\cM$ is the mean curvature vector of $\cM$.
Note that \cite[Theorem 4.2]{DS2} requires the $C^1$ regularity of $\varphi_r$.
We overcome this technical obstruction applying
\cite[Theorem 4.2]{DS2} to a standard smoothing of $\phi$ and then passing into the limit (the obvious details are left to the reader). Furthermore we observe that
\begin{align*}
\bT_F(d\omega\ser X) &= \int_{\cM} \ph_r\,\sum_{i=1}^{Q} \langle (\xi_1+D_{\xi_1} N^i)\wedge\dots\wedge(\xi_m+D_{\xi_m} 
N^i)\wedge  N^i\,,\,d\omega(p+ N^i(p))\, .
\end{align*}
Clearly
\begin{align*}
&\Big|\bT_F (d\omega \ser X) - \int_{\cM} \ph_r\,\sum_{i=1}^{Q} \langle (\xi_1+D_{\xi_1} N^i)\wedge\dots\wedge(\xi_m+D_{\xi_m} 
N^i)\wedge  N^i\,,\,d\omega(p)\rangle\Big|\\ 
\leq &\; C \|d\omega\|_1 \int \ph_r |N|^2 
\end{align*}
and we can therefore conclude
\begin{align}\label{e:sem_err}
\Big|\bT_F (d\omega \ser X)  - \bL(r)\Big|
&\leq  C\|d\omega\|_0 \int \ph_r |N| |DN|^2 + C \|d\omega\|_0 \int \ph_r |\etaa\circ N|+C \|d\omega\|_1 \int \ph |N|^2\, . 
\end{align}
Plugging \eqref{e:ov graph} and \eqref{e:sem_err} into \eqref{e:Err4-5} and using the triangular inequality, we then conclude
\begin{equation}\label{e:ov_con_errori}
\left| \bGam (r) - r^{-1} \bE (r)\right| \leq \sum_{j=1}^4 \left|\textup{Err}^o_j\right|\, ,
\end{equation}
where ${\rm Err}_4^o$ corresponds 
to ${\rm Err}_4$ of \eqref{e:Err4-5} when $X=X_o$.
With the same argument, but applying this time \cite[Theorem 4.3]{DS2} to $X=X_i$,
we get
\begin{align}\label{e:iv graph}
\delta \bT_F (X_i) =& \frac{1}{2} \int_\cM\Big(|D N|^2 {\rm div}_{\cM} Y  - 2
\sum_{i=1}^Q  \langle D N_i : ( D N_i \cdot D_{\cM} Y)\rangle \Big) + \sum_{j=1}^3 \textup{Err}^i_j\, ,
\end{align}
where this time the errors ${\rm Err}^i_j$ correspond to the error terms ${\rm Err}_j$ of 
\cite[Theorem 4.3]{DS2}, i.e.
\begin{gather}
{\rm Err}_1^i = - Q \int_{ \cM}\big( \langle H_\cM, \etaa \circ N\rangle\, {\rm div}_{\cM} Y + \langle D_Y H_\cM, \etaa\circ N\rangle\big)\, ,\label{e:inner_resto_1}\\
|{\rm Err}_2^i| \leq C_0 \int_\cM |A|^2 \left(|DY| |N|^2  +|Y| |N|\, |DN|\right), \label{e:inner_resto_2}\\
|{\rm Err}_3^i|\leq C_0 \int_\cM \Big( |Y| |A| |DN|^2 \big(|N| + |DN|\big) + |DY| \big(|A|\,|N|^2 |DN| + |DN|^4\big)\Big)\label{e:inner_resto_3}\, .
\end{gather}
Straightforward computations (again appealing to \cite[Proposition A.4]{DS5}) lead to 
\begin{align}\label{e:DY}
D_{\cM} Y (p) = \phi'\left(\frac{d(p)}{r}\right) \, \frac{d(p)}{r^2} \, \frac{\de}{\de \hat r} \otimes \frac{\de}{\de \hat r} + \phi \left(\frac{d(p)}{r}\right) \left( \frac{\Id}{r} + O(1)\right)\, ,
\end{align}
\begin{align}\label{e:divY}
\dv_\cM\, Y (p) 
& = \phi'\left(\frac{d(p)}{r}\right) \, \frac{d(p)}{r^2} + \phi\left(\frac{d(p)}{r}\right) \, \left(\frac{m}{r} + O(1) \right)\, .
\end{align}
Plugging \eqref{e:DY} and \eqref{e:divY} into \eqref{e:iv graph} and using \eqref{e:Err4-5} we then
conclude
\begin{equation}\label{e:iv_con_errori}
\left| \bD' (r) - (m-2)r^{-1} \bD (r) - 2 r^{-2} \bG (r)\right|
\leq C_0 \bD (r) + \sum_{j=1}^4 \left|{\rm Err}^i_j\right|+|\bT_F(d\omega\ser X)|\, .
\end{equation}
Let $\xi_1,\dots, \xi_m, \nu_1,\dots,\nu_n$ be an orthonormal basis of $\R^{m+n}$ such that $\xi_1(p),\dots, \xi_m(p)\in T_p \cM$, and let $dx^1,\dots,dx^m,dy^1,\dots,dy^n$ be its dual basis. Since $Y\in T\cM$, we can write
\[
Y(p)= \sum_{l=1}^m a_l(p) \, \xi_k(p)
\quad\text{and}\quad 
d\omega(p)= \sum_{k=1}^n b_k(p) dy^k\wedge dx^1\wedge\dots\wedge dx^m +\widehat{\omega}\,,
\] 
$\langle Z, \widehat \omega\rangle=0$, for every vector field $Z$ tangent to $\cM$.
Observe that $|d\omega(p+N^i(p))-d\omega(p)|\leq \bmo^{\sfrac12} |N|$ so that
\begin{align*}
\Bigl| \bT_F(d\omega \ser X) 
- \underbrace{\sum_{j=1}^m (-1)^{j-1}\int_{\cM} \langle D_{\xi_j}(\etaa\circ N)) \wedge \widehat{\xi_j}\wedge Y,  d\omega(p)\rangle}_{=:E} \Bigr|
\leq C\,\bmo^{\sfrac12}\int_{\cM} |Y| |N|\,|DN|
\end{align*} 
where $\widehat{\xi_j}=\xi_1\wedge\dots\wedge \xi_{j-1}\wedge\xi_{j+1}\wedge\dots\wedge \xi_m$ and we used $\xi_1\wedge \dots\wedge \xi_n\wedge Y=0$.
Next, using the above expressions for $Y$ and $d\omega$, and the property of $\widehat{\omega}$ together with $Y(p)\in T_p\cM$ for every $p\in \cM$, we achieve
\begin{align}
E
&= (-1)^{m-1} \sum_{k=1}^n\sum_{j=1}^m \int_{\cM}  b_k \,a_j  \langle D_{\xi_j} (\etaa\circ N), dy^k\rangle \notag\\
&= (-1)^{m-1} \sum_{k=1}^n\sum_{j=1}^m \int_{\cM}  b_k \, a_j \left(D_{\xi_j} (\etaa\circ N)^k- \langle \etaa\circ N ,\,D_{\xi_j} \nu^k\rangle \right)\notag\\
&\leq C\, \bmo^{\sfrac12} \left( \int_{\cM} \left(|DY|+|Y|\right)\, |\etaa\circ N|+\int_{\cM} |Y| \, |A_{\cM}| \,|\etaa\circ N|  \right)
\end{align}
where in the second and third inequalities we used integration by parts.
Therefore, in oreder to estimate $|\bT_F(d\omega\ser X)|$ it is enough to estimate Err$_1^i$, Err$_2^i$ and Err$_3^i$ above.

Proposition~\ref{p:variation} is then proved by the estimates of the errors
terms done in \cite[Section 4]{DS5}.

\begin{remark}
It is important to notice that we can apply \cite[Section 4]{DS5} verbatim because the estimates of Part 1 on the center manifold and the normal approximation are the same as the ones used in \cite{DS4}. In particular, Theorem \ref{t:cm} and Theorem \ref{t:approx} allow us to estimate our errors in terms of powers of the sidelength of the Whitney cubes, and conversely, Proposition \ref{p:separ} and Proposition \ref{p:splitting} imply a bound on these powers in terms of $\bD, \bH$ and $\bSigma$. This is achieved by a suitable covering, which handles the fact that on cubes of type $\sW_{n}$, the estimate on the sidelenght can only come from cubes of type $\sW_{e}$, and is also one of the reason for introducing the cut-off function $\phi$.
\end{remark}

\subsection{Boundedness of the frequency}\label{s:freq_bound}
In this section we recall two corollaries of the previous analysis: the boundedness of the frequency function $\bI_j$  along
the different center manifolds corresponding to the intervals of flattening and a Reverse Sobolev inequality.
To simplify the notation, we set $p_j := \Phii_j (0)$ and write simply $\cB_\rho$
in place of $\cB_\rho (p_j)$ .

\begin{theorem}[Boundedness of the frequency functions {\cite[Theorem 5.1]{DS5}}]\label{t:boundedness}
Let $T$ be as in Assumption \ref{i:H'}.
If the intervals of flattening are $j_0 < \infty$, then there is $\rho>0$ such that
\begin{equation}\label{e:finita1}
\bH_{j_0} >0 \mbox{ on $]0, \rho[$} \quad \mbox{and} \quad \limsup_{r\to 0} \bI_{j_0} (r)< \infty\, .
\end{equation}
If the intervals of flattening are infinitely many, then there is a number $j_0\in \mathbb N$ and a geometric constant $J_1\in \N$ such that
\begin{gather}
\bH_j>0 \mbox{ on $]\frac{s_j}{t_j}, 2^{-j_1}3[$ for all $j\geq j_0$} \quad \mbox{and} \quad 
\sup_{j\geq j_0} \sup_{r\in ]\frac{s_j}{t_j}, 2^{-j_1}3[} \bI_j (r) <\infty\, .\label{e:finita2}\\
\sup\left\{\min\left\{\bI_j(r)\,,\,\frac{r\int_{\cB_r}|DN_j|^2}{\int_{\de \cB_r}|N_j|^2}\right\}\;:\; j\geq j_0 \text{ and }\max\left\{\frac{s_j}{t_j},\frac{3}{2^{j_1}}\right\}\leq r<3\right\}<\infty
\end{gather}
where in the latter inequality we understand $I_j(r) =1$ when $H_j(r) = 0$.
\end{theorem}

\begin{proof}
Same as \cite[Theorem 5.1]{DS5}, using $\sfrac12 \bGam(r)\leq \bD(r)\leq 2 \bGam(r)$. In particular, we use here Proposition \ref{p:compara}. 
\end{proof}

\begin{corollary}[Reverse Sobolev {\cite[Corollary 5.2]{DS5}}]\label{c:rev_Sob}
Let $T$ be as in Assumption \ref{i:H'}.
Then, there exists a constant $C>0$ which {\em depends on $T$ but not on $j$} such that, for every $j$
and for every $r \in ]\frac{s_j}{t_j}, 1 ]$, there is $s\in ]\frac{3}{2}r, 3r[$
such that
\begin{equation}\label{e:rev_Sob}
\int_{\cB_{s}(\Phii_j(0))} |D{N}_j|^2 \leq \frac{C}{r^2} \int_{\cB_{s} (\Phii_j(0))} |{N}_j|^2\, .  
\end{equation}
\end{corollary}

\begin{proof} Same as \cite[Corollary 5.2]{DS5}, using $\sfrac12 \bGam(r)\leq\bD(r)\leq 2 \bGam(r)$.

\end{proof}

\section{Final Blow Up and $\D$-minimality of the limit}

\subsection{Blow-up maps}\label{ss:blowup}
Let $T$ be a current as in the Assumption \ref{i:H'}.
By Proposition~\ref{p:flattening} we can assume that for each radius $r_k$
there is an interval of flattening $I_{j(k)} =]s_{j(k)}, t_{j(k)}]$ containing $r_k$.
We define next the sequence of ``blow-up maps'' which will lead to the proof
Theorem~\ref{t:finale}.
To this aim, for $k$ large enough, we define $\bar{s}_k$ so that the radius $\frac{\bar s_k}{t_{j(k)}} \in \big]\frac32 \frac{r_k}{t_{j(k)}}, 3\frac{r_k}{t_{j(k)}}[$ is the radius
provided in Corollary~\ref{c:rev_Sob} applied to $r = \frac{r_k}{t_{j(k)}}$. We then set
$\bar{r}_k := \frac{2\bar s_k}{3t_{j(k)}}$
and rescale and translate currents and maps accordingly:
\begin{itemize}
\item[(BU1)] $\bar T_k = (\iota_{0,\bar r_k})_\sharp T_{j(k)} = ((\iota_{0,\bar r_k t_{j(k)}})_\sharp T) \res \B_{6\sqrt{m}/\bar{r}_k} $ and $\bar\cM_k := \iota_{0, \bar{r}_k} (\cM_{j(k)})$;
\item[(BU2)] $\bar{N}_k : \bar\cM_k \to \R^{m+n}$ are the rescaled $\bar\cM_k$-normal approximations
given by 
\begin{equation}\label{e:bar_N_k}
\bar{N}_k (p) = \frac{1}{\bar{r}_k} N_{j(k)} (\bar{r}_k p).
\end{equation}
\end{itemize}
Since $\frac12 < \frac{r_k}{\bar r_k t_{j(k)}} < 1$, it follows from
Proposition~\ref{p:seq} that
\[
\bE(\bar T_k, \B_{\frac12}) \leq C\bE(T, \B_{r_k}) \to 0.
\]
By the standard regularity theory of area minimizing currents and
Assumption~\ref{i:H'}, this implies that
$\bar T_k$ locally converge (and supports converge locally in the Hausdorff sense) to (a large portion of) a minimizing tangent cone
which is an $m$-plane with multiplicity $Q$, which we can assume to be $\R^{m}\times\{0\}$, that is
without loss of generality, we can assume that $\bar{T}_k$ locally converge to $Q\a{\pi_0}$.
Moreover, from Proposition~\ref{p:seq} it follows that
\begin{equation}\label{e:sing grande}
\cH^{m-2+\alpha}_\infty (\rD_Q (\bar{T}_k) \cap \B_1)
\geq C_0 r_k^{- (m-2+\alpha)}\cH^{m-2+\alpha}_\infty (\rD_Q ({T}) \cap \B_{r_k})\geq \eta>0\, ,
\end{equation}
where $C_0$ is a geometric constant.

In the next lemma, we show that the rescaled center manifolds $\bar \cM_k$ converge locally to
the flat $m$-plane $\pi_0$, thus leading to the following natural definition for the 
blow-up maps $N^b_k : B_{3}\subset\R^m \to \Iq (\R^{m+n})$:
\begin{equation}\label{e:sospirata_successione}
N^b_k (x) := \bh_k^{-1} \bar{N}_k (\be_k (x))\, ,
\end{equation}
where $\bh_k:=\|\bar N_k\|_{L^2(\cB_{\frac32})}$ and
$\be_k: B_3\subset\R^m\simeq T_{\bar{p}_k} \bar\cM_k\to \bar\cM_k$ denotes the exponential map
at $\bar{p}_k = \Phii_{j(k)} (0)/\bar{r}_k$ (here and in what follows we assume, w.l.o.g.,
to have applied a suitable rotation to each $\bar{T}_k$ so that the tangent plane $T_{\bar{p}_k} \bar{\mathcal{M}}_k$ coincides with
$\mathbb R^m\times \{0\}$).

\begin{lemma}[Vanishing lemma]\label{l:vanishing}
Under the Assumption~\ref{i:H'},  the following hold:
\begin{itemize}
\item[(i)] we can assume, without loss of generality, $\bar{r}_k \bmo^{j(k)}\to 0$; 
\item[(ii)] the rescaled center manifolds $\bar{\cM}_k$ converge
(up to subsequences) to
$\R^m\times \{0\}$ in $C^{3,\kappa/2} (\B_4)$ and
the maps $\be_k$ converge in $C^{2, \kappa/2}$ to the identity map
${\rm id}: B_3 \to B_3$;
\item[(iii)] there exists a constant $C>0$, depending only on $T$, such that, for every $k$,
\begin{equation}\label{e:rev_Sob2}
\int_{B_{\frac32}} |D{N}^b_k|^2 \leq C. 
\end{equation}
\end{itemize}
\end{lemma}

\begin{proof}
To show (i), note that, if $\liminf_k \bar{r}_k > 0$, we can extract a further subsequence and assume that $\lim_k \bar{r}_k >0$. Observe that then $\bar r := \limsup_k \frac{t_{j(k)}}{r_k} <\infty$. Since $r_k\downarrow 0$, we necessarily conclude that $t_{j(k)}\downarrow 0$ and hence $\|d\omega_{j(k)}\|_{C^{1,\eps_0}}\to 0$.
Moreover 
$\bE (T, \B_{6 \sqrt{m} t_{j(k)}}) \leq C (\bar r) \bE (\bar T_k, \B_{6\sqrt{m} \bar r_k^{-1}}) \to 0$
because $\bar T_k$ converges to $Q\a{\pi_0}$. 
We conclude $\bar{r}_k \bmo^{j(k)} \to 0$. On the other hand if $\lim_k \bar{r}_k = 0$ then (i) follows trivially from
the fact that $\bmo^j$ is a bounded sequence.

Next, using $\bar{r}_k\bmo^{j(k)}\to 0$ and the estimate of Theorem~\ref{t:cm},
it follows easily that $\bar \cM_k - \bar p_k$
converge (up to subsequences) to a plane in $C^{3,\kappa/2} (\B_4)$. 
By Proposition~\ref{p:flattening} (v) we deduce easily that such plane is in fact $\pi_0$. Since $0$ belongs to the support of $T_{j(k)}$
we conclude for the same reason that $\bar \cM_k$ is converging to $\pi_0$ as well.
Therefore, by \cite[Proposition~A.4]{DS5} the maps $\be_k$ converge to
the identity in $C^{2,\kappa/2}$ (indeed, by standard arguments
they must converge to the exponential map on the -- totally geodesic! --
submanifold $\R^m\times \{0\}$).
Finally, (iii) is a simple consequence of Corollary~\ref{c:rev_Sob}.
\end{proof}

The main result about the blow-up maps $N^b_k$ is the following.

\begin{theorem}[Final blow-up]\label{t:sospirato_blowup}
Up to subsequences, the maps $N^b_k$ converge strongly in $L^2(B_\frac32)$ to a function
$N^b_\infty: B_{\frac32} \to \Iq (\{0\}\times \R^{\bar{n}}\times \{0\})$ which
is Dir-minimizing in $B_t$ for every $t\in ]\frac{5}{3}, \frac{3}{2}[$ and satisfies
$\|N^b_\infty\|_{L^2(B_\frac32)}=1$ and $\etaa\circ N^b_\infty \equiv 0$.
\end{theorem}

\begin{proof}
We divide the proof in steps.

\emph{Step 1: First estimates.} Without loss of generality we might translate the manifolds $\bar \cM_k$ so that the
rescaled points $\bar{p}_k =\bar{r}_k^{-1}\Phii_{j(k)} (0)$ coincide all with the origin.
Let $\bar F_k: \cB_\frac32 \subset\bar\cM_k \to \Iq(\R^{m+n})$ be the multiple valued map
given by $\bar{F}_k (x) := \sum_i \a{x+ (\bar{N}_k)_i (x)}$ and, to
simplify the notation, set $\p_k := \p_{\bar{\cM}_k}$.
We claim the existence of a 
suitable exponent $\gamma>0$ such that
\begin{gather}
\Lip (\bar{N}_k|_{\cB_{\sfrac{3}{2}}}) \leq C \bh_k^\gamma\quad\text{and}\quad \|\bar{N}_k\|_{C^0 (\cB_{\sfrac32})} \leq C (\bmo^{j(k)} \bar{r}_k)^\gamma,\label{e:Lip_riscalato}\\
\mass((\bT_{\bar{F}_k} - \bar{T}_k) \res (\p_k^{-1} (\cB_\frac32)) \leq C \bh_k^{2+2\gamma},\label{e:errori_massa_1000}\\
\int_{\cB_\frac32} |\etaa\circ \bar{N}_k| \leq C \bh_k^2\label{e:controllo_media}\, .
\end{gather}
These are a consequence of the estimates of Theorem \ref{t:approx}, Propositions \ref{p:separ} and \ref{p:splitting} and Corollary \ref{c:rev_Sob}, as shown in \cite[Section~7]{DS5}.

It is then clear that the strong $L^2$ convergence of $N^b_k$ is a 
consequence of these bounds and of the Sobolev embedding
(cf.~\cite[Proposition~2.11]{DS1}); whereas, by \eqref{e:controllo_media},
\begin{equation*}
\int_{\cB_\frac32} |\etaa\circ N^b_\infty| =
\lim_{k\to +\infty}\int_{\cB_\frac32} |\etaa\circ N^b_k|
\leq C \lim_{k\to +\infty}\bh_k =0\, .
\end{equation*}
Finally, note that $N^b_\infty$ must take its values
in $\{0\}\times \R^{n}$, by the convergence of $\bar\cM_k$ to $\R^{m}\times\{0\}$.

\emph{Step 2: A suitable trivialization of the normal bundle.}
By Lemma~\ref{l:vanishing}, we can consider for every $\bar\cM_k$ an
orthonormal frame of $(T\bar\cM_k)^\perp$, $\nu^k_1, \ldots, \nu^k_{n}$ with the property that (cf.~\cite[Lemma~A.1]{DS2})
\[
\nu^k_j \to e_{m+j} \quad \mbox{in $C^{2,\kappa/2}(\bar \cM_k)$ as $k\uparrow\infty$}
\]
(for every $j$: here $e_1, \ldots , e_{m+n}$ is the standard basis of $\mathbb R^{m+n}$). Furthermore we let $\xi^k_1,\dots,\xi^k_m$ be an orthonormal bases of $T\bar\cM_k$, which is orthogonal to $\nu^k_1, \ldots, \nu^k_{n}$.

\emph{Step 3: Competitor function.}
We now show the $\D$-minimizing property of $N^b_\infty$.
Clearly, there is nothing to prove if its Dirichlet energy vanishes.
We can therefore assume that there exists $c_0>0$ such that
\begin{equation}\label{e:reverse_control}
c_0 \bh_k^2 \leq \int_{\cB_\frac32} |D\bar{N}_k|^2\, .
\end{equation}
Assume there is a radius $t \in \left]\frac54,\frac32\right[$ and a 
function $f: B_\frac32 \to \Iq(\R^{n})$ such that
\[
f\vert_{B_\frac32\setminus B_t} = N^b_\infty\vert_{B_\frac32 \setminus B_t} \quad\text{and}\quad
\D(f, B_t) \leq \D(N^b_\infty, B_t) -2\,\delta,
\]
for some $\delta>0$.
We can apply \cite[Proposition~3.5]{DS3} to the functions $N_k^b=\bh_k^{-1}\bar N_k$
and find competitors $v^b_k$ such that,
for $k$ large enough,
\begin{gather*}
v^b_k\vert_{\de B_r} = \bar N_k\vert_{\de B_r}, \quad
\Lip (v^b_k) \leq C \bh_k^\gamma, \quad |v^b_k| \leq C (\bmo^k \bar{r}_k)^\gamma,\\
\int_{B_\frac32} |\etaa\circ v^b_k| \leq C \bh_k^2\quad\text{and}\quad
\int_{B_\frac32} |Dv^b_k|^2 \leq \int |D\bar N_k|^2 - \delta \, \bh_k^2\, ,
\end{gather*}
where $C>0$ is a constant independent of $k$ and $\gamma$ the exponent of  \eqref{e:Lip_riscalato}-\eqref{e:controllo_media}.
Clearly, by Lemma~\ref{l:vanishing} and the estimate on the regluarity of $\bar \cM_k$, 
the maps 
\[
\tilde{N}_k (x)=\sum_i \a{ ((v^b_k)_i)^j(\be_k^{-1}(x)) \nu_j^k (x)},
\]
where we set
$((v_k^b)_i)^j:= \langle (v_k^b)_i (x), e_{m+j}\rangle$, (again we 
use Einstein's summation convention), satisfy again
\begin{gather*}
\tilde{N}_k \equiv \bar{N}_k \quad \mbox{in $\cB_\frac32\setminus\cB_t$},\quad
\Lip (\tilde{N}_k) \leq C \bh_k^\gamma, \quad |\tilde{N}_k| \leq C (\bmo^k \bar{r}_k)^\gamma,\\
\int_{\cB_\frac32} |\etaa\circ \tilde{N}_k| \leq C \bh_k^2
\quad\text{and}\quad
\int_{\cB_\frac32} |D\tilde{N}_k|^2 \leq \int_{\cB_\frac32} |D\bar{N}_k|^2 - \delta \bh_k^2.
\end{gather*}
and $\tilde N_k(x)\in (T_x\bar\cM_k)^\perp$.

\emph{Step 4: Competitor current and conclusion.}  Consider finally the map
$\tilde{F}_k (x) = \sum_i \llbracket x+(\tilde{N}_i)_k (x)\rrbracket$. The current $\bT_{\tilde{F}_k}$ coincides with 
$\bT_{\bar{F}_k}$ on $\p_k^{-1}( \cB_\frac32\setminus \cB_t)$. Define the current $Z_k:=T-\bT_{\bar F_k}+\bT_{\tilde F_k}$. If $S_k$ is any current such that
\[
\partial S_k = T_k-Z_k = \bT_{\bar F_k} - \bT_{\tilde F_k} 
\]
then the semicalibrated condition gives
\[
\mass (T_k) \leq \mass (Z_k) + S_k (d\omega_k)\, ,
\]
where $\omega_k$ is the calibrating form. 
Define $H^i_k: [0,1] \times \cB_t \to \I{Q}(\R^{m+n})$ for $i=1,2$ by
\begin{gather}
[0,1]\times \cB_t\ni (t,p) \mapsto H^1_k(t,p):= \sum_{i=1}^{Q}\a{p+t\, (\bar N_k)_i(p)}\in \I{Q}(\R^{2+n})\notag\\
[0,1]\times \cB_t \ni (t,p) \mapsto H^2_k(t,p):= \sum_{i=1}^{Q}\a{p+(1-t) \, (\tilde N_k)_i(p)}\in \I{Q}(\R^{2+n})\notag\,.
\end{gather}
We choose $S_k:=S^1_k+S^2_k$, where $S^i_k:=\bT_{H^i_k}$ for $i=1,2$.
Thanks to the homotopy formula in \cite{DS2}, we get
\begin{align*}
\de S_k^1 & = \bT_{{F_k}|_{\cB_t}}-Q \a{\cM} -\bT_{H_1\vert_{[0,1]\times  \partial \cB_t)} },\\
\de S_k^2 & =  Q\a{\cM} - \bT_{\tilde {F_k}|_{\cB_t}} + \bT_{H_2\vert_{[0,1]\times \partial \cB_t}} .
\end{align*}
On the other hand since $\bar {N}_k=\tilde N_k$ on $\partial \cB_t$, we conclude $\partial S_k = \partial (S_k^1+S_k^2) = T_k-Z_k$. 

We next estimate $|S_k^1 (d\omega_k)|$ and $|S_k^2 (d\omega_k)|$. Since the estimates are analogous, we give the details only
for the first. Moreover we drop the index $k$ just for these computations.
We start from the formula
\[
S^1(d\omega) =\int_{\cB_t}\int_0^1  \sum_{i=1}^{Q}
\big\langle \vec\zeta_i(t,p), d\omega((H^1)_i(t,p)) \big\rangle\, d\cH^2(p)\,dt,
\]
with
\begin{align*}
\vec{\zeta}_i(t,p) & = 
\big(\xi_1+t\,\nabla_{\xi_1}\bar N_i(p)\big)\wedge\dots \wedge\big(\xi_m+t\,\nabla_{\xi_m}  \bar N_i(p)\big)
\wedge  \bar N_i(p)\\
&=:\xi_1\wedge\dots\wedge \xi_m \wedge \bar N_i(p) + \vec{E}_i(t, p)\,,
\end{align*}
and
\begin{align}
|\vec E_i(t,p)| \leq C\,\bigl(|D\bar N|(p) + |D\bar N|^2 (p)\bigr)\,|\bar N|(p).
\end{align}
Next we note that
\begin{align}
d\omega((H^1)_i(t,p)) = d\omega(p) + I(t,p),
\end{align}
where $I(t,p)$ can be estimated by 
\begin{align}
|I(t,p)| & = |d\omega((H^1)_i(t,p)) - d\omega(p)|
\leq C\, \|D^2\omega\|_{L^\infty}\, |\bar N|(p).
\end{align}
Therefore, we have
\begin{align}
&\Big\vert\sum_{i=1}^{Q}\big\langle \vec\zeta_i(t,p), d\omega((H^1)_i(t,p)) \big\rangle
\Big\vert
\leq \sum_{i=1}^{Q}\langle\xi_1\wedge \dots\wedge \xi_m\wedge\bar N_i(p), d\omega(p) \rangle
+ \|d\omega\|_{L^\infty} \sum_{i=1}^{Q} |\vec E_i(t,p)|\notag\\
&\quad + C\sum_{i=1}^{Q}\Big((|\bar N_i| +|\vec E_i|)\, |I|\Big)(t,p)\notag\\
& \leq C \|d\omega\|_{C^1}\left( \,|\etaa\circ \bar N| + C |\bar N|^2(p) + C |D\bar N|(p)\,|\bar N|(p) + Cr |D\bar N|^2 (p)\right)\, ,\notag
\end{align} 
since $|\bar N|\leq C (\bmo \bar r)^\gamma$ on $\cB_t$.
Arguing similarly for $S_k^2$ (observe that we have the bound $|\tilde N_k| (p) \leq C(\bmo^k \bar r_k)^\gamma$) 
and estimating $|\bar N||D\bar N| + |\tilde N| |D\tilde N| \leq C (|\bar N|^2 + |\tilde N|^2) + C (|D\bar N|^2 + |D\tilde N|^2)$, we achieve
\begin{align*}
|S_k^1(d\omega_k)| + |S_k^2(d\omega_k)| &\leq 
C\,\|d\omega^k\|_{C^1} \left(  \int_{\cB_t}\big(|\etaa\circ \bar N_k| +|\etaa\circ \tilde N_k|\big)+
C\, r^{-1}\int_{\cB_t)}\big(|\bar N_k|^2 +|\tilde N_k|^2\big) \right.\\
&\quad
+ \left.C r\int_{\cB_t}\big(|D \bar N_k|^2+|D \tilde N_k|^2\big)\right),
\end{align*}
and, using \ref{e:rev_Sob2} and the estimates on $\bar N_k$ and $\tilde N_k$, we conclude that
\begin{equation}\label{e:semicazzi}
S_k(d\omega_k)\leq o (\bh_k^2)
\end{equation}

Finally observe that, taking into account \eqref{e:errori_massa_1000}, we conclude that
\begin{align}\label{e:prima}
\mass (\bT_{{F_k}|_{\cB_t}}) 
&\leq \mass (T_k\res \p_k^{-1} (\cB_t)) + \|T-\bT_{{F_k}|_{\cB_t}}\|( \p_k^{-1} (\cB_t))\notag\\ 
&\leq  \mass (\bT_{\tilde F_k}) +2 \,\|T-\bT_{{F_k}|_{\cB_t}}\|( \p_k^{-1} (\cB_t)) \notag\\
&\leq  \mass (\bT_{\tilde F_k})+ o(\bh_k)^2 \, .
\end{align}
so that 
\begin{align}
 0\leq \mass (\bT_{\tilde{F}_k}) - \mass (\bT_{\bar{F}_k}) 
 &\leq \frac{1}{2} \int_{\cB_\rho} \left(|D\tilde{N}_k|^2 - |D\bar{N}_k|^2\right) 
+ C \|H_k\|_{C^0} \int \left(|\etaa\circ \bar{N}_k| + |\etaa\circ \tilde{N}_k|\right)\nonumber\\
&\qquad + \|A_k\|_{C^0}^2 \int \left( |\bar{N}_k|^2 + |\tilde{N}_k|^2\right) + o (\bh_k^2)
\leq -\frac{\delta}{2} \bh_k^2 + o (\bh_k^2)\, ,\label{e:fine!!}
\end{align}
where in the last inequality we have taken into account Lemma \ref{l:vanishing}.
For $k$ large enough this gives a contradiction and concludes the proof.\end{proof}

\begin{proof}[Proof of Theorem 0.3.]
Let $N^b_\infty$ be as in Theorem \ref{t:sospirato_blowup}, and
\[
\Upsilon:=\{ x \in \overline{B}_1\,:\, N^b_\infty(x)=Q\a{0}\}\,.
\]
Since $\etaa\circ N^b_\infty\equiv 0$ and $\|N^b_\infty\|_{L^2(B_{\sfrac32}}=1$, from the regularity of Dir-minimizing $Q$-valued functions (cf. \cite[Proposition 3.22]{DS1}), we know that $\cH^{m-2+\alpha}_\infty(\Upsilon)=0$. The same capacitary argument as in \cite[Subsection 6.2]{DS5} shows that this is a contradiction with Assumptions \ref{assurdo}, which proves the result.
\end{proof}

\bibliographystyle{plain}
\bibliography{references-Cal}

\end{document}